\documentclass{amsart}
\usepackage{bc}
\newcommand\lasso[5][]{
  \coordinate (source) at #2;
  \coordinate (goal) at #4;
  \coordinate (x0) at ($(goal)!2*#3!(source)$);
  \coordinate (x1) at ($(goal)!#3!45:(source)$);
  \coordinate (x2) at ($(goal)!#3!-45:(source)$);

  \draw[#1] let \p1=($(x1)-(goal)$), \p2=($(x2)-(goal)$),
        \n1={atan2(\y1,\x1)}, \n2={atan2(\y2,\x2)}, \n3={\n2>\n1?\n2:\n2+360}
    in #2 -- #5 (x0) .. controls +($0.5*(goal)-0.5*(x0)$) and +($0.5*(\p2)$) .. (x1) arc (\n1:\n3:#3)
  .. controls +($0.5*(\p1)$) and +($0.5*(goal)-0.5*(x0)$) .. (x0);
 \draw[->,dash pattern=on 0pt off #3]  ($(goal)!1.414*#3!135:(source)$) -- ($(goal)!-#3!(source)$);
}

\begin{document}
\title[Sphere bisets and decidability of Thurston equivalence]{Algorithmic aspects of branched coverings II/V.\\ Sphere bisets and decidability of Thurston equivalence}
\author{Laurent Bartholdi}
\email{laurent.bartholdi@gmail.com}
\author{Dzmitry Dudko}
\email{dzmitry.dudko@gmail.com}
\address{\'Ecole Normale Sup\'erieure, Paris \emph{and} Mathematisches Institut, Georg-August Universit\"at zu G\"ottingen}
\thanks{Partially supported by ANR grant ANR-14-ACHN-0018-01 and DFG grant BA4197/6-1}
\date{June 10, 2018}
\begin{abstract}
  We consider \emph{Thurston maps}: branched self-coverings of the
  sphere with ultimately periodic critical points, and prove that the
  Thurston equivalence problem between them (continuous deformation of
  maps along with their critical orbits) is decidable.

  More precisely, we consider the action of mapping class groups, by
  pre- and post-composition, on branched coverings, and encode them
  algebraically as \emph{mapping class bisets}.  We show how the
  mapping class biset of maps preserving a multicurve decomposes into
  mapping class bisets of smaller complexity, called \emph{small
    mapping class bisets}.

  We phrase the decision problem of Thurston equivalence between
  branched self-coverings of the sphere in terms of the conjugacy and
  centralizer problems in a mapping class biset.  Our decomposition
  results on mapping class bisets reduce these decision problems to
  small mapping class bisets; they correspond to rational maps,
  homeomorphisms and maps double covered by a torus endomorphism, and
  their conjugacy and centralizer problems are solvable respectively
  in terms of complex analysis, group theory and linear algebra.

  Branched coverings themselves are also encoded into bisets, with
  actions of the fundamental groups. We characterize those bisets that
  arise from branched coverings between topological spheres, and
  extend this correspondence to maps between spheres with multicurves,
  whose algebraic counterparts are \emph{sphere trees of bisets}.

  To illustrate the difference between Thurston maps and
  homeomorphisms, we produce a Thurston map with infinitely generated
  centralizer --- while centralizers of homeomorphisms are always
  finitely generated.
\end{abstract}
\maketitle

\setcounter{tocdepth}{1}
\tableofcontents

\section{Introduction}
Consider a topological sphere $S^2$ and a finite subset
$A\subset S^2$. A \emph{Thurston map} is a branched covering
$f\colon(S^2,A)\selfmap$ such that $A$ contains the critical values of
$f$. It is natural to consider Thurston maps up to \emph{combinatorial
  equivalence}, written $f\sim g$: we set $f_0\sim f_1$ if there
exists a path of Thurston maps
$(f_t\colon(S^2,A_t)\selfmap)_{t\in[0,1]}$ connecting them, along
which $\#A_t$ is constant. Thanks to this relation, Thurston maps
become combinatorial objects, given e.g.\ by a triangulation of
$(S^2,A)$ and data encoding how the triangulation is mapped, up to
isotopy, to itself by $f$.  Our main result is:
\begin{mainthm}[= Theorem~\ref{thm:A:new}]\label{thm:main}
  It is decidable whether two Thurston maps are
  equivalent. Furthermore, the centralizer of a Thurston map, namely
  the group of homeomorphisms $(S^2,A)\selfmap$ commuting with the
  Thurston map up to isotopy, is computable.
\end{mainthm}
Partial positive results existed beforehand: in case the Thurston maps
are homeomorphisms, the first statement was proven by
Hemion~\cite{hemion:homeos}, who showed that the conjugacy problem is
solvable in mapping class
groups. Bonnot-Braverman-Yampolsky~\cite{bonnot-braverman-yampolsky:thurstondecidable}
showed that equivalence to a rational map is decidable; this implies
decidability of combinatorial equivalence for ``unobstructed''
maps. Selinger-Yampolsky~\cite{selinger-yampolsky:geometrization}
showed decidability of combinatorial equivalence in case the Thurston
maps decompose canonically into rational maps with ``hyperbolic
orbifold''.

The difficult part of Theorem~\ref{thm:main} is that for general
Thurston maps, as opposed to degree-$1$ or unobstructed ones, the
centralizers of small maps affect rotations of annuli in the canonical
decomposition. This allows Thurston maps to have infinitely generated
centralizers, as we show by example in~\S\ref{ss:complicated
  centralizer}; and understanding centralizers is always at the core
of solving equivalence problems ($A\sim B$ if and only if the
centralizer of $A\sqcup B$ contains an element exchanging $A$ and $B$;
this is decidable if the centralizer can be given with an explicit
generating set).

There are negative results in a more general setting (replacing the
sphere by a finite cell complex, e.g.): in~\cite{bartholdi:wordorder}
it is shown that even the equality problem can become unsolvable. It
should be no surprise, therefore, that our argument makes crucial use,
in a specific place, of complex analysis.

The algorithms themselves, however, can be made almost entirely
symbolic; their complexity will be the topic of the last
article~\cite{bartholdi-dudko:bc5} in this series.

\subsection{The semigroup of Thurston maps}
Let us return to Thurston maps $f\colon(S^2,A)\selfmap$. We refine
combinatorial equivalence into a composition of isotopy and
conjugation: let us call Thurston maps $f_0,f_1\colon(S^2,A)\selfmap$
isotopic, and write $f_0\approx f_1$, if there is a path of Thurston
maps $(f_t\colon(S^2,A)\selfmap)_{t\in[0,1]}$ connecting them.

We denote by $K(S^2,A)$ the semigroup of isotopy classes of Thurston
maps $(S^2,A)\selfmap$, under composition.  In particular, those
isotopy classes of Thurston maps that fix $A$ and are homeomorphisms
form a group $\Mod(S^2,A)$, the \emph{pure mapping class group}.

Two elements $f,g\in K(S^2,A)$ are \emph{conjugate} if there exists
$m\in\Mod(S^2,A)$ with $m g=f m$. Combinatorial equivalence is
essentially a conjugacy problem: to determine whether two Thurston
maps $f\colon(S^2,A)\selfmap$ and $g\colon(S^2,C)\selfmap$ are
combinatorially equivalent, enumerate all bijections
$\phi\colon A\to C$, and extend each $\phi$ arbitrarily to a
homeomorphism $\widehat\phi\colon(S^2,A)\to(S^2,C)$. Then $f\sim g$ if
and only if $f$ and $\widehat\phi^{-1}\circ g\circ\widehat\phi$ are
conjugate in $K(S^2,A)$ for some $\phi\colon A\to C$; see
Definition~\ref{defn:combinatorial equivalence} and
Lemma~\ref{lem:combinatorial equivalence}.

From now on, we therefore concentrate on the semigroup of Thurston
maps. We solve the following decision problems in $K(S^2,A)$:
\begin{description}
\item[The conjugacy problem] Given $f,g\in K(S^2,A)$, are they
  conjugate? If so, give a witness
  $m\in\Mod(S^2,A)$ to $mg=f m$.
\item[The centralizer problem] Given $f\in K(S^2,A)$, compute
  its centralizer $Z(f)\coloneqq\{m\in\Mod(S^2,A)\mid m f=f m\}$.
\end{description}

A \emph{multicurve} on $(S^2,A)$ is a collection $\CC$ of disjoint
pairwise non-homotopic simple closed curves on $S^2\setminus A$, none
of which can be homotoped rel $A$ to a point, and viewed up to
isotopy.  We let $K(S^2,A,\CC)$ denote those Thurston maps
$f\in K(S^2,A)$ that preserve $\CC$ in the sense that
$f^{-1}(\CC)= \CC$, and we denote by $\Mod(S^2,A,\CC)$ the
subgroup of $\Mod(S^2,A)$ consisting of mapping classes that preserve
$\CC$ curvewise, including their orientation.

Consider $(S^2,A)$ and a multicurve $\CC$. The connected components of
$S^2\setminus\CC$ are called \emph{small spheres}, and are themselves
homeomorphic to punctured spheres.  Given $f\in K(S^2,A,\CC)$ and a
periodic component $S\subset S^2\setminus\CC$ viewed as a punctured
sphere $\widehat S$, the first return map
$f^e\colon \widehat S\selfmap$ is called a \emph{small Thurston
  map}. We denote by $R(f,A,\CC)$ the set of small Thurston maps of
$f$.

\subsection{Algorithms}
To prove Theorem~\ref{thm:main}, we devise a collection of algorithms
that manipulate branched coverings; the main task is achieved by
Theorem~\ref{thm:A}. The proof proceeds by considering a refinement of
Pilgrim's \emph{canonical decomposition} of spheres under a Thurston
map~\cite{pilgrim:combinations}.

As a first step, Algorithm~\ref{algo:decidelevy} finds, given two
Thurston maps, their maximal invariant multicurves that are mapped to
themselves by degree $1$.

Results from~\cite{bartholdi-dudko:bc4} imply that the resulting
small maps are either homeomorphisms, or expanding, or double-covered
by a torus endomorphism, and that these are computable.

The expanding maps may themselves be further decomposed along another
invariant multicurve, in such a manner that the small maps are
rational. Again, this multicurve and the small maps are computable.

Algorithm~\ref{algo:conj:LevyFree} solves the conjugacy problem in
each of the resulting classes (homeomorphism, double-covered by torus
endomorphism or rational): in the first case, using standard group
theory, in the second using floating-point approximations, and in the
third case using linear algebra.

The centralizers are also readily computed along the way: as an
important outcome of Nielsen-Thurston classifications, centralizers of
homeomorphisms are computable as finite-index subgroups of products of
mapping class groups. Thurston's rigidity theorem implies that
rational maps with hyperbolic orbifold have trivial
centralizer. Centralizers of affine maps on the torus are computable
as linear groups.

The crux of the proof of Theorem~\ref{thm:main} lies in the following
Theorem~\ref{thm:A}: it assembles the solutions of the conjugacy and
centralizer problems of the small maps to that of the original
ones. We phrase the main, somewhat technical result in the form of an
algorithm with oracle solving the conjugacy and centralizer problems
for small maps:
\begin{mainthm}[See Algorithm~\ref{algo:bisets_decomp}]\label{thm:A}
  Let $(S^2,A)$ be a punctured sphere, and let $\CC$ be a multicurve
  on $(S^2,A)$. Then there is an algorithm with oracle, that,
  \begin{list}{---}{\leftmargin=0pt\itemindent=1.5em}
  \item given two Thurston maps $f,g\in K(S^2,A,\CC)$,
  \item assuming that the centralizers of all small Thurston maps in
    $R(f,A,\CC)$ are computable (e.g.\ as
    finite-index subgroups of products of mapping class groups),
  \item assuming that the conjugacy problems between maps in
    $R(f,A,\CC)$ and in $R(g,A,\CC)$ are solvable, and that witnesses
    can be produced in case maps are conjugate,
  \end{list}
  answers whether $f,g$ are conjugate under $\Mod(S^2,A,\CC)$, if so
  produces a witness, and computes the centralizer of $f$ as the
  kernel of a homomorphism from a finite-index subgroup of a product
  of mapping class groups towards a finitely generated abelian group.
\end{mainthm}
We elected to state the main step of the proof of
Theorem~\ref{thm:main} in the form of an algorithm with oracle, so as
to separate the computational steps of combining small maps and
of computing them. The combination done by Theorem~\ref{thm:A} is
practically efficient, while the oracles can be very much
improved. Nevertheless, Algorithm~\ref{algo:bisets_decomp} is an
unconditional algorithm, because the oracles are algorithmic (and
described in this text). The issues of efficiency will be returned to
in~\cite{bartholdi-dudko:bc5}. In particular, we cannot avoid some use
of floating-point calculations, in~\S\ref{algo:ratmaps}, but this is the
only place where they appear.

We show by an example (see~\S\ref{ss:complicated centralizer}) that
this description of centralizers of Thurston maps is in a sense the
best that can be achieved; this makes centralizers of Thurston maps
significantly more complicated than centralizers of mapping classes,
which are always finite-index subgroups of products of mapping class
groups.

It is possible to approach a proof of Theorem~\ref{thm:main} by
considerations on Teichm\"uller space, as we discuss
in~\S\ref{ss:rem:Thurston iter}, by running Thurston's algorithm
(iteration of the pull-back map). This really amounts to obtaining the
canonical decomposition of Thurston maps into small maps. The main
step, however, remains: the data extracted from Teichm\"uller space
(in the form of rotation parameters, return maps, and their induced
actions) has to be supplied to Theorem~\ref{thm:A} in the form of an
oracle.

\subsection{Algebraic structure of branched coverings}
We delayed until now a discussion of \emph{how} Thurston maps and
$K(S^2,A)$ are to be represented as input and output to our
algorithms. By far the most convenient language is the symbolic,
group-theoretical language of \emph{bisets} pioneered by
Nekrashevych~\cite{nekrashevych:ssg}.

These notions were explained in detail
in~\cite{bartholdi-dudko:bc1}. In this article, we specialize its
results to $2$-dimensional spheres and Thurston maps.  It is in fact
worthwhile to consider a slightly more general situation: fix a
branched covering $f\colon(S^2,C)\to(S^2,A)$ such that $A$ contains
$f(C)$ and the critical values of $f$; we call $f$ a \emph{sphere
  map}. Set
\begin{equation}\label{eq:intro:m}
  M(f,C,A)=\{m'fm''\mid m'\in\Mod(S^2,C),m''\in\Mod(S^2,A)\}\,/\,\text{isotopy},
\end{equation}
the composition being written in algebraic (left-to-right) order. Then
$M(f,C,A)$ is a \emph{biset}: a set endowed with commuting left and
right actions by $\Mod(S^2,C)$ and $\Mod(S^2,A)$ respectively. We call
it the \emph{mapping class biset} of $f$.

Note that $K(S^2,A)$ is a $\Mod(S^2,A)$-$\Mod(S^2,A)$-biset, and that
if $f\colon(S^2,A)\selfmap$ is a Thurston map, $M(f,A,A)$ is a
subbiset of $K(S^2,A)$.  Note that $K(S^2,A)$ is not a finitely
generated semigroup. On the other hand, $M(f,A,A)$ is \emph{finitely
  presented} -- this is one of our main reasons to work with
$M(f,A,A)$. Finite presentation also allows to discuss the
\emph{efficiency} of algorithms, see~\S\ref{s:CombEquiv}
and~\cite{bartholdi-dudko:bc3}.

We give an algorithm (Algorithm~\ref{algo:compute M(f,C,A)}) that
computes the structure of $M(f,C,A)$ as a biset. It uses in a
fundamental manner the facts that $H\coloneqq\pi_1(S^2\setminus C)$
and $G\coloneqq\pi_1(S^2\setminus A)$ are free groups, and that
$\Mod(S^2,C)$ and $\Mod(S^2,A)$ act faithfully respectively on $H$ and
$G$ by outer automorphisms. Decompositions of (certain variants of)
mapping class bisets are important steps in proving
Theorem~\ref{thm:A}; see for example~\S\ref{ss:extensionsmcb}
and~\S\ref{ss:DecompOfMCB}.

Sphere maps $f\colon(S^2,C)\to (S^2,A)$ themselves are represented as
$H$-$G$-bisets $B(f)$, see~\eqref{eq:Dfn:SphBis}
and~\cite{bartholdi-dudko:bc1}*{\S\ref{bc1:ss:dynamics}}. It was shown
by Kameyama~\cite{kameyama:thurston} that $B(f)$ is a complete
invariant of isotopy, namely $B(f)\cong B(g)$ if and only if $f,g$ are
isotopic. We give a converse to this result in
Theorem~\ref{thm:dehn-nielsen-baer+} by showing that to every sphere
biset there is an associated sphere map, unique up to isotopy. This
allows us to switch freely between geometric and algebraic settings of
sphere maps and bisets.

Recall from~\cite{bartholdi-dudko:bc1} that left-free bisets are
naturally associated with topological correspondences, namely pairs of
maps $Y\leftarrow Z\to X$ such that $Z\to X$ is a covering. The
mapping class biset $M(f,C,A)$ is the biset of a well-known
correspondence, one between the moduli spaces $\mathscr M_C$ and
$\mathscr M_A$; see Proposition~\ref{prop:modular correspondence}
in~\S\ref{ss:examples}.

Multicurves $\CC$ on $(S^2,A)$ are represented as a collection of
conjugacy classes in $G$. The main result
from~\cite{bartholdi-dudko:bc1} is that, given a $1$-dimensional
cover of a space such as that afforded by the small spheres of
$(S^2,A)$, and a map compatible with this cover such as a Thurston map
$f$ preserving $\CC$, there exists a tree of bisets decomposition of
$f$ whose ``fundamental biset'' is isomorphic to $B(f)$. We show that
this decomposition is computable:
\begin{mainthm}[See Theorem~\ref{thm:DecompOfBiset}]
  There is an algorithm that, given a Thurston map
  $f\colon(S^2,A)\selfmap$ by its biset $B(f)$ and given a multicurve
  $\CC$ on $(S^2,A)$ with $f^{-1}(\CC)\supseteq\CC$, computes the tree
  of bisets decomposition of $B(f)$ along $\CC$.
\end{mainthm}

We extend Kameyama's result to maps with multicurves: we define
\emph{sphere} trees of bisets in Definition~\ref{defn:SphTreeOfBis},
and prove that the decomposition of a sphere map is a sphere tree of
bisets and conversely, giving a complete invariant up to isotopy or
combinatorial equivalence:
\begin{mainthm}[See Theorem~\ref{thm:kameyama-mc}, Corollary~\ref{cor:kameyama-mc} and Corollary~\ref{cor:dehn-nielsen-baer+-mc}]\label{thm:pilgrim}
  Let $f,f'\colon(S^2,A,\CC)\selfmap$ be Thurston maps with the same
  (possibly empty) multicurve $\CC$ satisfying
  $f^{-1}(\CC)\supseteq\CC\subseteq (f')^{-1}(\CC)$. Then $f,f'$ are
  isotopic rel $A\cup\CC$ if and only if the sphere trees of bisets of
  $f,f'$ are isomorphic.

  Consequently, if $g\colon(S^2,C,\DD)\selfmap$ is a Thurston map with
  $\DD$ satisfying $g^{-1}(\DD)\supseteq\DD$, then $f$ and $g$ are
  combinatorially equivalent by a homeomorphism sending $A$ to $C$ and
  $\CC$ to $\DD$ if and only if the sphere trees of bisets of $f,g$
  are \emph{conjugate}.

  Finally, for every sphere tree of bisets
  $\subscript\gfY\gfB_\gf$ with spheres
  $(S^2,C,\DD)$ and $(S^2,A,\CC)$ associated respectively with the
  trees of groups $\gfY$ and $\gf$ there exists a sphere map
  $f\colon(S^2,C,\DD)\to(S^2,A,\CC)$, unique up to isotopy rel
  $C\cup\DD$, whose graph of bisets is isomorphic to $\gfB$.
\end{mainthm}

Theorem~\ref{thm:pilgrim} expresses, in the algebraic language of
bisets, the decomposition and combination theorems of Kevin Pilgrim
(see e.g.~\cite{pilgrim:combinations}*{Theorem~5.1}). Numerous
examples of sphere tree of bisets decompositions appeared
in~\cite{bartholdi-dudko:bc0}*{\S\ref{bc0:ss:examples}}.  Here two
graphs of bisets $\subscript\gf\gfB_\gf$ and $\subscript\gfY\gfC_\gfY$
are called conjugate if there exists a biprincipal
(see~\cite{bartholdi-dudko:bc1}*{\S\ref{bc1:ss:BiprGrBis}}) tree of
bisets $\subscript\gf\gfI_\gfY$ such that $\gfB\otimes_\gf\gfI$ and
$\gfI\otimes_\gfY\gfC$ are isomorphic, namely have same underlying
graphs and isomorphic bisets and intertwiners.

We extend in~\S\ref{ss:orbispheres} the equivalence between sphere
maps and bisets to the setting of \emph{orbispheres} (spheres with
singularities of cone type $2\pi/n$): given an orbisphere map $f$, its
biset $B(f)$ is a complete invariant of isotopy; and conversely given
an orbisphere biset $B$ there is an orbisphere map $f$, unique up to
isotopy, realizing $B$, see Theorem~\ref{thm:dehn-nielsen-baer++}.

We end this article, in~\S\ref{ss:examples}, with a description of
$M(f,A,A)$ for some Thurston maps $f\colon(S^2,A)\selfmap$, recasting
the calculations of~\cite{bartholdi-n:thurston} and its interpretation
via Teichm\"uller theory in the language of mapping class
bisets.

The language of trees of groups and trees of bisets is particularly
well suited to describe the dynamical operations of \emph{tuning} and
\emph{renormalizing}. Consider a Thurston map
$f\colon (S^2,A)\selfmap$. Tuning refers to removing a neighbourhood
of a periodic critical cycle from $S^2$, and replacing the map locally
by a polynomial of same degree. Cutting $S^2$ along the boundary of
the neighbourhood yields a tree of groups decomposition of the
fundamental group, and a tree of bisets decomposition of $B(f)$, and
tuning amounts to replacing cyclic bisets (that of a map $z^d$ for
some $d\in\N$) by bisets of polynomials in the tree of bisets.

Renormalizing refers to considering a periodic component of a
decomposition of $S^2$ and taking its first return map. In terms of
graphs of groups, this amounts to taking an invariant subtree of a
tensor power of the biset.

Branched coverings $f\colon(S^2,C)\to(S^2,A)$, and the
biset~\eqref{eq:intro:m} may be considered for $C=\emptyset$, and
correspond to well-studied objects. When $C=\emptyset$, the biset
$B(f)$ is just a right $\pi_1(S^2\setminus A,*)$-set of cardinality
$\deg(f)$, namely an $\#A$-tuple of permutations in $\deg(f)\perm$
satisfying some conditions and considered up to conjugation
(see~\S\ref{ss:maps between spheres}). Two branched coverings
$f,f'\colon S^2\to(S^2,A)$ are isomorphic as coverings (they are also
called ``Hurwitz equivalent'') if and only if $f\approx m\circ f'$ for
a mapping class $m\in\Mod(S^2,A)$, if and only if
$f'\in M(f,\emptyset,A)$. The problem of classifying Hurwitz
equivalence classes for which the tuple of permutations has given
cycle structure (its ``Hurwitz passport'') has been extensively
investigated, even when $\#A=3$, see
e.g.~\cite{lando-zvonkin:graphsonsurfaces}*{Chapter~5}. From our
perspective, this problem can be interpreted as a classification of
mapping class bisets with given portrait, as subbisets of $K(S^2,A)$.

We formulate a symbolic version of the Thurston iteration, running on
the mapping class group, and conjecture its convergence to a certain
attractor, see Conjecture~\ref{conj:GenralNucl}. This supports the
following conjecture, in which the input to the algorithm is an
element $h\cdot f\cdot k$ with $h,k$ expressed as words in generators
of $\Mod(S^2,A)$:
\newtheorem{mainconj}[mainthm]{Conjecture}
\begin{mainconj}
  Let $f\colon(S^2,A)\selfmap$ be a Thurston map. Then the conjugacy
  and centralizer problem in $M(f)$ are solvable in polynomial time.
\end{mainconj}

\subsection{Notation}
We continue with the notation of the previous articles in the
series. In particular, we write $\looparrowright$ for group actions,
and $S\perm$ for the symmetric group on $S$.

Concatenation of paths is written $\gamma\#\delta$ for ``first
$\gamma$, then $\delta$''; inverses of paths are written
$\gamma^{-1}$.

If $\beta\colon[0,1]\to X$ is a path and $f\colon Y\to X$ is a
covering map, we denote by $\beta\lift f y$ the $f$-lift of the
path $\beta$ starting at $y\in f^{-1}(\beta(0))$.

We write $\approx$ for isotopy or homotopy of paths, maps etc, $\sim$
for conjugacy or combinatorial equivalence, and $\cong$ for
isomorphism of algebraic objects.  

The main objects of study are \emph{marked}, or \emph{punctured}
spheres. There is no fundamental difference between $S^2\setminus A$
and a pair $(S^2,A)$ with $A\subset S^2$; it is usually more
convenient to keep the points in $A$ while remembering that they have
a special status (for example, they are frozen by homotopies), but
when marked spheres are cut along multicurves the boundary components
appear topologically as punctures, not marked points. We therefore
sometimes have to switch between these notations. We allow
$A=\emptyset$, but forbid $\#A=1$.

In all sections except~\S\ref{ss:orbispheres}, by
$f\colon(S^2,C)\to(S^2,A)$ we denote an orientation preserving map
between marked spheres, called a \emph{sphere map}; see~\S\ref{ss:maps
  between spheres} for the precise
definition. In~\S\ref{ss:orbispheres} the map
$f\colon(S^2,C)\to(S^2,A)$ will be allowed to reverse orientation.

For a small sphere $S\subset(S^2,A)$, cut by boundary curves
$\CC_1,\dots,\CC_m\subset S^2\setminus A$, we denote by $\overline S$
its topological closure, namely $S\cup\CC_1\cup\cdots\cup\CC_m$, and
by $\widehat S$ the quotient of $\overline S$ obtained by shrinking
every boundary curve of $\overline S$ to a point. It is a sphere
marked by the image of $A\cap S$ and the boundary curves. Depending on
context, we call either of $S$, $\overline S$ and $\widehat S$ a
\emph{small sphere}.

A \emph{map} $f\colon (S^2,C,\DD)\to(S^2,A,\CC)$ with marked curves
$\CC,\DD$ means $f(C)\subseteq A$ and $f^{-1}(\CC)=\DD$, unless it is
explicitly stated that $f^{-1}(\CC)\supseteq\DD$. A \emph{small sphere
  map} of $f$ is the restriction of $f$ to a small sphere of
$(S^2,C,\DD)$. For a Thurston map $f\colon(S^2,A,\CC)\selfmap$, we
always require $f^{-1}(\CC)=\CC$ up to ignoring peripheral, trivial
and duplicate curves; we call \emph{small Thurston maps} of $f$ the
first return maps $f^e$ on periodic small spheres of $(S^2,A,\CC)$,
see Lemma-Definition~\ref{lem:SmallMaps}. We reserve the terminology
`$\CC$ is $f$-invariant' to mean $f^{-1}(\CC)=\CC$; this is sometimes
called \emph{completely invariant} in the literature.

We denote by $\one$ the identity map. By a slight abuse of notation,
if $A\subset C$ then we denote by $\one \colon (S^2,C)\to (S^2,A)$ the
identity map on $S^2$ which erases the marked points in
$C\setminus A$.

We recall briefly the notation for biset operations
from~\cite{bartholdi-dudko:bc1}: if $B$ is a $G$-$H$-biset and $C$ is
an $H$-$K$-biset, then $B\otimes C$ is the $G$-$K$-biset
$B\times C/\{(b g,c)=(b,g c)\text{ for all }b\in B,c\in C,g\in
G\}$. The \emph{contragredient} of $B$ is the $H$-$G$-biset $B^\vee$
with elements written $b^\vee$ and operations
$h\cdot b^\vee\cdot g\coloneqq (g^{-1}\cdot b\cdot h^{-1})^\vee$.

We refer to~\cite{bartholdi-dudko:bc0} for more details on the overall
strategy, algorithms, and examples of Thurston maps and combinatorial
equivalence.

\subsection{Acknowledgments}
We are grateful to Jean-Pierre Spaenlehauer for having helped in
computing the algebraic correspondence in~\S\ref{ss:pilgrim}.

\section{Spheres}
The previous article~\cite{bartholdi-dudko:bc1} in the series set up the
general theory of bisets associated with maps between topological
spaces. Much sharper results may be obtained in a more specific context,
that of maps between punctured spheres.

Let us consider a topological sphere $S^2$ with a finite, ordered
collection of marked points $A=\{a_1,\dots,a_n\}\subset S^2$. We allow
$n=0$, but expressly forbid $n=1$. Choose a basepoint $*\in
S^2\setminus A$. The fundamental group $\pi_1(S^2\setminus A ,*)$ may
then be presented in the following, explicit manner: choose for each
$i\in\{1,\dots,n\}$ a loop $\gamma_i$ in $S^2\setminus A $ that starts
at $*$, goes towards $a_i$, encircles it counterclockwise in a very
small loop, and returns to $*$ along the same path. Make sure that all
paths $\gamma_i$ are disjoint except at their endpoints, and that they
are arranged counterclockwise around $*$. We call the $\gamma_i$
\emph{Hurwitz generators}, or colloquially \emph{lollipops}. We then
have
\begin{equation}\label{eq:spheregp}
  G=\pi_1(S^2,A,*)\coloneqq \pi_1(S^2\setminus A ,*)=\langle \gamma_1,\dots,\gamma_n\mid \gamma_1\cdots\gamma_n\rangle.
\end{equation}
Indeed, if one cuts $S^2$ open along paths from the basepoint $*$ to
the marked points $a_i$ along the beginnings of the $\gamma_i$, one is
left with a disk, whose perimeter gives the unique relation of $G$.

A cycle (loop without basepoint) in $S^2\setminus A $ is represented
by a conjugacy class in $G\coloneqq\pi_1(S^2\setminus A ,*)$. The
group $G$ comes with extra data: the collection
$\{\gamma_1^G,\dots,\gamma_n^G\}$ of conjugacy classes, called
\emph{peripheral conjugacy classes}, defined by the property that
$\gamma_i^G$ represents can be homotoped to a loop circling $a_i$ once
counterclockwise.

We use the notation $\beta\approx_A\gamma$ to mean that curves
$\beta,\gamma$ are isotopic relative to their endpoints and to $A$; in
other words, that $\beta\#\gamma^{-1}$ is trivial in the fundamental
group $\pi_1(S^2,A,\beta(0))$.

\begin{defn}[Sphere groups]\label{defn:sphere groups}
  A \emph{sphere group} is a tuple $(G,\Gamma_1,\dots,\Gamma_n)$
  consisting of a group and $n\neq1$ conjugacy classes $\Gamma_i$ in $G$,
  such that $G$ admits a presentation as in~\eqref{eq:spheregp} for
  some choice of $\gamma_i\in\Gamma_i$. The classes
  $\Gamma_1,\dots,\Gamma_n$ are called the \emph{peripheral conjugacy
    classes} of $G$.
  The \emph{rank} of $G$ is $\rank(G)=n$.
\end{defn}
In other words, $G$ is a free group of rank $n-1$, with $n$
distinguished conjugacy classes $\Gamma_i$ satisfying a certain
relation. If $n=0$, then $G$ is a trivial group. For every $d\in\N$ we
denote by $\Gamma_i^d$ the subset $\{g^d \mid g\in\Gamma_i\}$. We also
set $\Gamma^+_i=\bigcup_{d\ge 1} \Gamma^d_i$.  By the remarks above,
\begin{lem}\label{lem:sphere to sphere group}
  Let $(S^2,A)$ be a marked sphere, and choose $*\in S^2\setminus A$. Then
  $\pi_1(S^2\setminus A,*)$ is a sphere group. Conversely, given a sphere
  group $(G,\Gamma_i)$ there exists a sphere, unique up to homeomorphism,
  whose fundamental group is isomorphic to $G$ and whose peripheral
  conjugacy classes are $\{\Gamma_i\}$.\qed
\end{lem}

From the perspective of a computer, a sphere group can be represented
in two different manners. It may be considered as just a number `$n$',
and its elements are words in the symbols $\pm1,\dots,\pm n$, subject
to appropriate reduction rules. It may also be considered as a
collection of $n$ reduced words $\gamma_1,\dots,\gamma_n$ in the
standard free group $F_{n-1}$ on $n-1$ generators, and its elements
are in bijection with reduced words in $F_{n-1}$'s generators. These
two forms are obviously equivalent, and both have their advantages in
terms of implementation.

\subsection{Maps between spheres}\label{ss:maps between spheres}
Let $\Mod(S^2,A)$ denote the pure mapping class group of $(S^2,A)$,
namely the set of isotopy classes of homeomorphisms $S^2\selfmap$
fixing $A$ pointwise. Given a choice of basepoint
$*\in S^2\setminus A$, consider the fundamental group
$G=\pi_1(S^2\setminus A,*)$. Then every homeomorphism
$\phi\colon(S^2,A)\selfmap$ induces an isomorphism
$G\to\pi_1(S^2\setminus A,\phi(*))$; choosing a path $\ell$ from $*$
to $\phi(*)$ in $S^2\setminus A$ leads to an automorphism of $G$
defined by $[\gamma]\mapsto[\ell\#(\phi\circ\gamma)\#\ell^{-1}]$,
which is well-defined up to inner automorphisms, and yields a
homomorphism $\Mod(S^2,A)\to\Out(G)$. The classical Dehn-Nielsen-Baer
Theorem asserts that this map is injective, and describes its image:
\begin{thm}[Dehn-Nielsen-Baer, see~\cite{farb-margalit:mcg}*{Theorem~8.8}]\label{thm:dehn-nielsen-baer}
  Let $(G,\Gamma_1,\dots,\Gamma_n)$ be a sphere group. Then the
  natural map $\Mod(S^2,A)\to\Out(G)$ is injective, and its image in
  $\Out(G)$ consists of those automorphisms of $G$ that map each
  $\Gamma_i$ to itself.\qed
\end{thm}

\begin{defn}
  The \emph{mapping class group} $\Mod(G)$ of a sphere group $G$ as in
  Definition~\ref{defn:sphere groups} is the group of outer
  automorphisms of $G$ that preserve all peripheral conjugacy classes
  of $G$.
\end{defn}
For calculations, it is useful to know that $\Mod(G)$ is generated by
\emph{Dehn twists}: for example, the automorphisms $\tau_{i,j}$, for
$1\le i<j<n$, that map $\gamma_k$ to
$\gamma_k^{\gamma_i\gamma_{i+1}\cdots\gamma_j}$ for
$k\in\{i,\dots,j\}$ while fixing the other generators,
see~\cite{margalit+mccammond:braid}. Geometrically, $\tau_{i,j}$ is a
homeomorphism that acts as a ``Chinese burn'' on an annulus
surrounding the points $a_i,\dots,a_j$ and acts trivially elsewhere.

A \emph{sphere map} $f\colon(S^2,C)\to(S^2,A)$ between marked spheres
is a branched covering between the underlying spheres, locally
modelled at $c\in S^2$ in oriented complex charts by
$z\mapsto z^{\deg_c(f)}$ for some integer $\deg_c(f)\ge1$, and such
that $A$ contains $f(C)$ and all $f(c)$ for which $\deg_c(f)>1$. Those
$c\in S^2$ with $\deg_c(f)>1$ are called \emph{critical points}, and
their images $f(c)$ are called \emph{critical values}. The map $f$ is
a \emph{covering} if $C=f^{-1}(A)$. The restriction
$f\restrict C\colon C\to A$ is called the \emph{portrait} of
$f\colon(S^2,C)\to(S^2,A)$. Two sphere maps
$f_0,f_1\colon (S^2,C)\to (S^2,A)$ are \emph{isotopic} if there is a
continuous path $(f_t\colon (S^2,C)\to (S^2,A))_{t\in[0,1]}$ of sphere
maps connecting them.

\subsubsection{Sphere bisets}\label{ss:sphere bisets}
This notion of maps admits a precise algebraic translation, to which
we turn now. Consider a sphere map $f\colon (S^2,C)\to (S^2,A)$; we
may interpret it as a correspondence
\begin{equation}\label{eq:CorrOfSphMaps}
  (S^2,C)\overset{\one}\leftarrow (S^2,f^{-1}(A))\overset{f}{\rightarrow} (S^2,A)\end{equation}
such that $(S^2,C)\overset{\one}\leftarrow (S^2,f^{-1}(A))$ forgets
points in $f^{-1}(A)\setminus C$. Note that
$(S^2,C)\overset{\one}\leftarrow (S^2,f^{-1}(A))$ is not a sphere map
(unless $C=f^{-1}(A)$), however its inverse
$(S^2,C)\overset{\one}\rightarrow (S^2,f^{-1}(A))$ is a sphere
map. In~\cite{bartholdi-dudko:bc1}*{\S\ref{bc1:ss:biset of
    correspondence}} we gave a general definition for the biset of a
correspondence; applied to~\eqref{eq:CorrOfSphMaps} this definition
takes the following form (see
\cite{bartholdi-dudko:bc1}*{Equation~\eqref{bc1:eq:Bfi:lmm:FibrCorr}})
\begin{equation}\label{eq:Dfn:SphBis}
  B(f,\dagger,*)=\{\delta\colon[0,1]\to S^2\setminus C\mid\delta(0)=\dagger,f(\delta(1))=*\}\,/\,{\approx_C},
\end{equation}
and with left and right actions of $\pi_1(S^2,C,\dagger)$ and
$\pi_1(S^2,A,*)$ given by preconcatenation, respectively
postconcatenation through lifting by $f$. If $(S^2,C)=(S^2,A)$, then it will be usually assumed that $\dagger=*$; thus~\eqref{eq:Dfn:SphBis} is a $\pi_1(S^2,A,*)$-biset. The
biset~\eqref{eq:Dfn:SphBis} possesses extra properties coming from the
peripheral conjugacy classes of the sphere groups
$\pi_1(S^2,C,\dagger)$ and $\pi_1(S^2,A,*)$; they will be described in
Definition~\ref{dfn:SphBis} and Lemma~\ref{lem:SphBisOfSphMap}.

Hurwitz describes in~\cite{hurwitz:ramifiedsurfaces} an
elegant classification of branched self-coverings $S^2\selfmap$ in
terms of \emph{admissible $n$-tuples} of permutations
$(\sigma_1,\dots,\sigma_n)$ in $d\perm$. Such an $n$-tuple is
admissible if $\sigma_1\cdots\sigma_n=1$, the group
$\langle\sigma_1,\dots,\sigma_n\rangle$ acts transitively on
$\{1,\dots,d\}$, and the cycle lengths of the permutations
$\gamma_1,\dots,\gamma_n$ satisfy
\begin{equation}\label{eq:riemannhurwitz}
  \sum_{i=1}^n\sum_{\substack{c\text{ cycle}\\\text{of }\sigma_i}}\big(\text{length}(c)-1\big)=2d-2.
\end{equation}
Hurwitz gives a bijection between degree-$d$ branched self-coverings
$S^2\selfmap$ with critical values contained in $\{a_1,\dots,a_n\}$,
viewed up to isomorphism of coverings, and admissible $n$-tuples in
$d\perm$, viewed up to global conjugacy by $d\perm$ and the ``mapping
class group action''. For the latter, set $G=\pi_1(S^2\setminus A,*)$
and identify $(\sigma_1,\dots,\sigma_n)$ with the homomorphism
$\pi\colon G\to d\perm$ given by $\gamma_i\mapsto\sigma_i$, and let
the mapping class group $\Mod(G)$ act on $\pi$ by
precomposition. Condition~\eqref{eq:riemannhurwitz} amounts to the
statement that the Euler characteristic of the cover is $2=\chi(S^2)$.

\begin{defn}[Multiset of lifts~\cite{bartholdi-dudko:bc1}*{\S\ref{bc1:ss:ccgroups}}]\label{defn:lifts}
  Let $\subscript H B_G$ be a left-free biset, and contract every left
  orbit of $B$ to a point by considering $\{\cdot\}\otimes_H
  B$. Consider $g\in G$. Then $\{\cdot\}\otimes_H B$ decomposes into
  orbits $S_1\sqcup\cdots\sqcup S_\ell$ under the right action of $g$,
  of respective cardinalities $d_1,\dots,d_\ell$; and for all
  $i=1,\dots,\ell$, choosing $s_i\in B$ with
  $\{\cdot\}\otimes_H s_i\in S_i$ there are elements $h_i\in H$ with
  $h_is_i=s_i g^{d_i}$. The multiset
  $\{(d_i,h_i^H)\mid i=1,\dots,\ell\}$ consisting of degrees and
  conjugacy classes in $H$ is independent of the choice of the $s_i$,
  and depends only on the conjugacy class of $g$; it is called the
  \emph{multiset of lifts} of $g^G$.
\end{defn}
If $B$ be the biset of a sphere map $f\colon (S^2,C)\to (S^2,A)$ and
$g^G$ be interpreted as a simple closed curve, then the multiset
$\{(h_i^H)\mid i=1,\dots,\ell\}$ is, up to homotopy, the set
$f^{-1}(g^G)$, and the $d_i$ are the degrees with which $f$ maps the
corresponding closed curves onto $g^G$.

\begin{defn}[Sphere bisets]\label{dfn:SphBis}
  Consider $(H,\Delta_j)$ and $(G,\Gamma_i)$ two sphere groups. A
  \emph{sphere biset} is an $H$-$G$-biset $B$ such that the following
  hold:
  \begin{enumerate}\renewcommand\theenumi{SB\ensuremath{{}_\arabic{enumi}}}
  \item $B$ is left-free and right-transitive;\label{cond:1:dfn:SphBis}
  \item choose representatives
    $\gamma_1\in\Gamma,\dots,\gamma_n\in\Gamma_n$; then the
    permutations of $\{\cdot\}\otimes_H B$ induced by the right action
    of $\gamma_1,\dots,\gamma_n$
    satisfy~\eqref{eq:riemannhurwitz};\label{cond:2:dfn:SphBis}
  \item the multiset of all lifts of $\Gamma_1,\dots,\Gamma_n$
    contains exactly once every $\Delta_j$, all the other conjugacy
    classes being trivial.
   \label{cond:3:dfn:SphBis} 
  \end{enumerate}
  By the last condition, to every peripheral conjugacy class
  $\Delta_j$ in $H$ is associated a well-defined \emph{degree}
  $\deg_{\Delta_j}(B)\in\N$ and conjugacy class
  $\Gamma_i\eqqcolon B_*(\Delta_j)$, such that
  $(\deg_{\Delta_j}(B),\Delta_j)$ belongs to the lift of
  $\Gamma_i$. We define in this manner a map $B_*$ from the peripheral
  conjugacy classes in $H$ to those of $G$, called the \emph{portrait}
  of $B$.
\end{defn}
In case the peripheral conjugacy classes of $G,H$ are indexed as
$(\Gamma_a)_{a\in A}$ and $(\Delta_c)_{c\in C}$ respectively, we write
$B_*(c)=a$ rather than $B_*(\Delta_c)=\Gamma_a$, defining in this
manner a map $B_*\colon C\to A$. It is easy to see that if $B$ be the
biset of $f\colon (S^2,C)\to (S^2,A)$ and $\Gamma_a$ and $\Delta_c$
denote the peripheral conjugacy classes around $a\in A$ and $c\in C$
respectively, then $\deg_{\Delta_c}(B)=\deg_{c}(f)$, and $f(c)=a$ if
and only if $B_*(c)=a$.

\begin{lem}\label{lem:SphBisOfSphMap}
  The biset defined in~\eqref{eq:Dfn:SphBis} of a sphere map $f\colon
  (S^2,C)\to (S^2,A)$ is a sphere
  $\pi_1(S^2,C,\dagger)$-$\pi_1(S^2,A,*)$-biset.
\end{lem} 
\begin{proof}
  The biset $B(f)$ is clearly left-free. Consider $b_1,b_2\in
  B(f)$. Up to homotopy we may assume that $b_1,b_2$, and $\dagger$
  are away from $f^{-1}(A)$. Set
  $g\coloneqq f(b_1^{-1})\# f(b_2) \in \pi_1(S^2,A,*)$. Then
  $b_1g =b_2$; this verifies Condition~\eqref{cond:1:dfn:SphBis} of
  Definition~\ref{dfn:SphBis}.

  Consider a peripheral conjugacy class $\Gamma$, say around a
  puncture $a\in A$. Let $\{(d_i,h_i^H)\mid i=1,\dots,\ell\}$ be the
  multiset of lifts of $\Gamma$. We may enumerate $f^{-1}(a)$ as
  $\{c_1,c_2,\dots , c_\ell\}$ such that $f$ maps $c_i$ to $a$ with
  degree $d_i$ and such that $h_i^H$ is the peripheral conjugacy class
  around $c_i$, respectively the trivial class, if $c_i\in C$,
  respectively $c_i\notin C$. This verifies
  Condition~\eqref{cond:3:dfn:SphBis} of
  Definition~\ref{dfn:SphBis}. Since $f^{-1}(A)$ contains $2d-2$
  critical points counting with multiplicities, we
  have~\eqref{eq:riemannhurwitz}; this is
  Condition~\eqref{cond:2:dfn:SphBis}.
\end{proof}

\subsubsection{Equivalence between sphere maps and bisets} Kameyama essentially proves in~\cite{kameyama:thurston}*{Theorem~3.6}
that to an isomorphism class of sphere bisets corresponds a unique
isotopy class of sphere maps. In fact, we give a bijection between
isotopy classes of maps and isomorphism classes bisets, extending at
the same time the Dehn-Nielsen-Baer
Theorem~\ref{thm:dehn-nielsen-baer} to non-invertible maps:
\begin{thm}\label{thm:dehn-nielsen-baer+}
  Let $f_0,f_1\colon(S^2,C)\to(S^2,A)$ be sphere maps, and consider
  $H\coloneqq\pi_1(S^2\setminus C,\dagger)$ and
  $G\coloneqq\pi_1(S^2\setminus A,*)$ for choices of $\dagger\in
  S^2\setminus C$ and $*\in S^2\setminus A$. Then $B(f_0)$ is a sphere
  $H$-$G$-biset, and $B(f_0)\cong B(f_1)$ if and only if
  $f_0\approx f_1$.

  Conversely, for every sphere $H$-$G$-biset $B$ there exists a sphere
  map $f\colon(S^2,C)\to(S^2,A)$, unique up to isotopy, such
  that $B\cong B(f)$.
\end{thm}

The proof of Theorem~\ref{thm:dehn-nielsen-baer+} appears below,
after some preparation.  Consider a sphere $H$-$G$-biset $B$.  Choose
$b\in B$, and let $G_b$ be the stabilizer of $\{\cdot\}\otimes b$ in
the $G$-set $\{\cdot\}\otimes_H B$. It is a subgroup of $G$ of index
equal to the degree of $B$. Then $B$ naturally splits as
\begin{equation}
 \label{eq:DecOfSphBis}
 \subscript H B_{G} \cong H b G_b  \otimes \subscript{G_b}G_G. 
\end{equation}

This splitting is the algebraic counterpart
of~\eqref{eq:CorrOfSphMaps}, and we use it to give, along the way, the
structure of a sphere group to $G_b$:
\begin{lem}\label{lem:CorrOfSphBis}
  Suppose that $\subscript H B_G$ is the biset of a sphere map
  $f\colon (S^2,C)\to (S^2,A)$ as in~\eqref{eq:Dfn:SphBis}  and with
  $G=\pi_1(S^2\setminus A,*)$ and $H=\pi_1(S^2\setminus
  C,\dagger)$. Consider $b\in B$ and let $*'$ be the endpoint of
  $b$.

  Then $\pi_1(S^2,f^{-1}(A), *')$ is identified via
  $f_*\colon \pi_1(S^2,f^{-1}(A), *')\to \pi_1(S^2,A, *)$ with $G_b$,
  and via this identification the $\pi_1(S^2,f^{-1}(A),*')$-$G$-biset
  of $f\colon (S^2,f^{-1}(A))\to (S^2,A)$ is isomorphic to
  $\subscript{G_b}G_G$ while the $H$-$\pi_1(S^2,f^{-1}(A),*')$-biset of
  $(S^2,C)\overset{\one}{\rightarrow} (S^2,f^{-1}(A))$ is isomorphic
  to $H b G_b$.
 
  Moreover, via the identification of $G_b$ with
  $\pi_1(S^2,f^{-1}(A),*')$ the peripheral conjugacy classes of $G_b$
  are $(\Xi_{i,j})_{i,j}$ constructed as follows. Let
  $\Gamma_1,\dots,\Gamma_n$ be the peripheral conjugacy classes of
  $G$. Then for each $\Gamma_i$ there is a unique decomposition
  \begin{equation}\label{eq:SplGamToXi}
    \Gamma_i^+ \cap G_b = \Xi^{+}_{i,1}\sqcup \Xi^{+}_{i,2}\sqcup\dots \sqcup \Xi^{+}_{i,s} 
  \end{equation}
  such that every $\Xi_{i,j}$ is a conjugacy class of $G_b$. Assuming
  $\Xi_{i,j}$ is generated by $\gamma_{i,j}^{d(i,j)}$ with
  $\gamma_{i,j}\in \Gamma_i$, we let $\{(d(i,j),\Xi_{i,j})\}$ be the
  multiset of lifts of $\Gamma_i$ via $\subscript{G_b}G_{G}$.
\end{lem}

\begin{proof} 
  Since $f\colon S^2\setminus f^{-1}(A)\to S^2\setminus A$ is a
  covering map, $\pi_1(S^2,f^{-1}(A), *')$ is identified with $G_b$ via
  $f_*$ and the biset of $f\colon (S^2,f^{-1}(A))\to (S^2,A)$ is
  $\subscript{G_b}G_G$. It follows immediately from~\eqref{eq:Dfn:SphBis}
  that the $H$-$\pi_1(S^2,f^{-1}(A),*')$-biset of
  $(S^2,C)\overset{\one}{\rightarrow} (S^2,f^{-1}(A))$ is $H b G_b$.
  
  Let us prove the claims concerning $(\Xi_{i,j})_{i,j}$. Suppose that
  $\Gamma_i$ is a peripheral conjugacy class around $a_i\in A$,
  suppose that $f^{-1}(a_i)=\{c_{i,1},c_{i,2},\dots, c_{i,s}\}$ and
  that $f$ has degree $d(i,j)$ at $c_{i,j}$, and suppose that
  $\Xi_{i,j}$ is the peripheral conjugacy class of
  $\pi_1(S^2,f^{-1}(A),*')$ around $c_{i,j}$. Observe that $\Xi_{i,j}$
  are pairwise disjoint as peripheral conjugacy classes around
  different points in $f^{-1}(A)$. Therefore, $\Xi^{+}_{i,j}$ are
  pairwise disjoint. Then $f_*$ identifies
  $\Xi^{+}_{i,1}\sqcup \Xi^{+}_{i,2}\sqcup\dots \Xi^{+}_{i,s}$ with
  $\Gamma_i^+\cap G_b$ and we get~\eqref{eq:SplGamToXi}. If $d(i,j)$
  be the local degree of $f$ at $c_{i,j}$, then $f_*(\Xi_{i,j})$ is
  generated by $\gamma_{i,j}^{d(i,j)}$ with
  $\gamma_{i,j}\in \Gamma_i$.
\end{proof}
  
\begin{lem}\label{lem:CorrOfSphBis2}
  Consider a sphere $H$-$G$-biset $B$ and choose $b\in B$. Endow
  $G_b$ with the sphere structure given in Lemma~\ref{lem:CorrOfSphBis}:
  assuming that $\Gamma_1,\dots,\Gamma_n$ are peripheral conjugacy
  classes of $G$, the peripheral conjugacy classes of $G_b$ are
  $(\Xi_{i,j})_{i,j}$ specified by~\eqref{eq:SplGamToXi}. Then in the
  decomposition~\eqref{eq:DecOfSphBis} the bisets $H b G_b$ and
  $\subscript{G_b}G_G$ are sphere bisets.
\end{lem}
\begin{proof}
  Observe that $\{\cdot\}\otimes_H B \cong \{\cdot\}\otimes_{G_b} G$
  as right $G$-sets; thus the action of $G$ on $\{\cdot\}\otimes_{G_b}
  G$ satisfies the Hurwitz condition~\eqref{eq:riemannhurwitz}. Write
  $G=\pi_1(S^2\setminus A,*)$ as a sphere group; then by Hurwitz's
  theorem~\cite{hurwitz:ramifiedsurfaces}, the cover
  $(\widetilde{S^2\setminus A})/G_b$ of $S^2\setminus A$ associated
  with $G_b\le G$ is a finitely-punctured sphere. By
  Lemmas~\ref{lem:SphBisOfSphMap} and~\ref{lem:CorrOfSphBis} the group
  $G_b$ is a sphere group with peripheral conjugacy classes
  $(\Xi_{i,j})_{i,j}$ specified by~\eqref{eq:SplGamToXi} and,
  moreover, $\subscript{G_b}G_G$ is a sphere biset.
  
  Observe that the multiset of the lifts of $\Gamma_i$ via $\subscript H B_G$
  is the multiset of the lifts via $H b G_b$ of the lifts via
  $\subscript{G_b}G_G$ of $\Gamma_i$ with the corresponding degrees being
  multiplied. Since $H b G_b$ has degree one, the multiset of the lifts
  of $\Xi_{i,j}$ is a singleton $\{(1,\widetilde \Xi_{i,j})\}.$
  Therefore, the multiset of the lifts of $\Gamma_i$ is
  $\{(d(i,j),\widetilde \Xi_{i,j})\}$ with $d(i,j)$ as in
  Lemma~\ref{lem:CorrOfSphBis}. This verifies
  Conditions~\eqref{cond:2:dfn:SphBis} and~\eqref{cond:3:dfn:SphBis}
  of Definition~\ref{dfn:SphBis}; Condition~\eqref{cond:1:dfn:SphBis}
  is immediate.
\end{proof}

\begin{proof}[Proof of Theorem~\ref{thm:dehn-nielsen-baer+}]
  If $f_0\approx f_1$, say via a path
  $(f_t\colon (S^2,C)\to (S^2,A))_{t\in [0,1]}$, then
  $B(f_0)\cong B(f_1)$ because $B(f_t)$, being a discrete object,
  remains constant along the isotopy.
  
  Assume $B(f_0)\cong B(f_1)$; say $\beta\colon B(f_0)\to B(f_1)$ is a
  biset isomorphism. Let us construct a homeomorphism
  $\psi \colon (S^2,C)\selfmap$ such that $f_0=\psi f_1$ and such that
  $B(\psi)\cong \subscript H H_H$. Then using
  Theorem~\ref{thm:dehn-nielsen-baer} we will get $\psi\approx \one$,
  and therefrom $f_0\approx f_1$.
  
  Choose $b\in B(f_0)$ and decompose
  $\subscript H B(f_0)_{G} \cong H b G_b \otimes \subscript{G_b}G_G.$
  and
  $\subscript H B(f_1)_{G} \cong H \beta(b) G_{\beta(b)} \otimes
  \subscript{G_{\beta(b)}}G_G$ as in~\eqref{eq:DecOfSphBis}. Since
  $\beta$ is an isomorphism we have $G_{b}=G_{\beta(b)}$, so
  $\subscript{G_{b}}G_G = \subscript{G_{\beta(b)}}G_G$.  By
  Lemma~\ref{lem:CorrOfSphBis} the bisets of
  $f_0\colon (S^2,f_0^{-1}(A))\to (S^2,A)$ and
  $f_1\colon (S^2,f_1^{-1}(A))\to (S^2,A)$ are respectively isomorphic
  to $\subscript{G_{b}}G_G$ and $ \subscript{G_{\beta(b)}}G_G$ after
  identifying $\pi_1(S^2,f_0^{-1}(A))$ and $\pi_1(S^2,f_1^{-1}(A))$
  with $G_{b}$ and $G_{\beta(b)}$ as in
  Lemma~\ref{lem:CorrOfSphBis}. Therefore,
  $f_0\colon (S^2,f_0^{-1}(A))\to (S^2,A)$ and
  $f_1\colon (S^2,f_1^{-1}(A))\to (S^2,A)$ are covering maps
  associated with the same subgroup of $G$; thus the map
  $\psi\coloneqq f_0\circ f_1^{-1}\colon (S^2,f_0^{-1}(A))\to
  (S^2,f_1^{-1}(A))$ is a well defined homeomorphism specified so that
  $\psi(\text{the endpoint of }b)=\text{the endpoint of
  }\beta(b)$. Finally
  \[\subscript H B(\psi \colon (S^2,C)\selfmap)_H
    =H b G_{b}\otimes_{G_b=G_{\beta(b)}} G_{\beta(b)} \beta(b) H\cong
    \subscript H H_H.\]

  Let us now prove the second part of the theorem. Decompose
  $\subscript H B_{G} \cong H b G_b \otimes \subscript{G_b}G_G$ as
  in~\eqref{eq:DecOfSphBis}. By Hurwitz's
  theorem~\cite{hurwitz:ramifiedsurfaces} the cover of
  $S^2\setminus A$ associated with $G_b\le G$ is a finitely punctured
  sphere; write this cover as
  $f\colon S^2\setminus f^{-1}(A) \to S^2\setminus A$. As in
  Lemma~\ref{lem:CorrOfSphBis} we identify $G_b$ with
  $\pi_1(S^2,f^{-1}(A))$ and we denote by $(\Xi_a)_{a\in f^{-1}(A)}$
  the peripheral conjugacy classes of $G_b\cong \pi_1(S^2,f^{-1}(A)).$

  By Lemma~\ref{lem:CorrOfSphBis2} the biset $H b G_b$ is a sphere
  biset. Since $H b G_b$ has a single left orbit, for every
  $a\in f^{-1}(A)$ the multiset of the lifts of $\Xi_a$ is a singleton
  $\{(1,\widetilde \Xi_a)\}$. Then the homomorphism
  $\iota\colon G_b\to H$ given by $b g = \iota(g) b$ has the property
  that $\iota(\Xi_a)=\widetilde \Xi_a$; i.e. $\iota$ ``forgets'' all
  $\Xi_a$ with trivial $\widetilde \Xi_a$.

  Let $\one\colon (S^2,f^{-1}(A))\to (S^2,C')$ be the map forgetting
  all $a\in f^{-1}(A)$ with trivial $\widetilde \Xi_a$. Since the
  biset of this map is naturally identified with $H b G_b$, we obtain
  a natural isomorphism of $\pi_1(S^2,C')$ with $H$, and therefore of
  $C'$ with $C$. Using these identifications, the biset of $(S^2,C')
  \overset{\one}{\rightarrow} (S^2,f^{-1}(A)) \overset{f}{\rightarrow}
  (S^2,A)$ is isomorphic to $\subscript H B_G$.

  It follows from Theorem~\ref{thm:dehn-nielsen-baer} that there is a
  map, unique up to isotopy, identifying $(S^2,C)$ with $(S^2,C')$ and
  such that the biset of this map is isomorphic to $\subscript H H_H$. This
  finishes the construction of a sphere map
  $f\colon (S^2,C)\to (S^2,A)$ with the biset $\subscript H B_G$.
\end{proof}

\subsection{Thurston maps}
Consider a branched self-covering $f\colon(S^2,A)\selfmap$. This means
$f(A)\subseteq A$ and $A$ contains the critical values of $f$. In
particular, $A$ contains the \emph{post-critical set} of $f$,
\[P_f=\bigcup_{k\ge1}f^k(\text{critical points of }f).
\]
On the other hand, $A$ could also contain fixed points of $f$, or more
generally preperiodic points together with their forward orbit. If $A$
is finite, then $f$ is called a \emph{Thurston map}.

Let $f\colon(S^2,A)\selfmap$ be a Thurston map. In particular, $f$
induces a map $A\selfmap$. Then its portrait $f\colon A\selfmap $ is a
finite dynamical system, sometimes called its \emph{critical portrait}.

\begin{defn}[Combinatorial equivalence of maps]\label{defn:combinatorial equivalence}
  Let $f_0\colon(S^2,A_0)\selfmap$ and $f_1\colon(S^2,A_1)\selfmap$ be
  two Thurston maps. We say that $f_0$ and $f_1$ are
  \emph{combinatorially equivalent}, aka ``Thurston-equivalent'', if
  there exists a path of Thurston maps
  $(f_t\colon(S^2,A_t)\selfmap)_{t\in[0,1]}$ connecting $f_0$ to
  $f_1$. In that case, we write $f\sim g$.
\end{defn}

\noindent There is an equivalent formulation of combinatorial
equivalence in terms of isotopy: we shall use the following
\begin{lem}\label{lem:factorization}
  Two sphere maps $f_0,f_1\colon(S^2,C)\to(S^2,A)$, are isotopic if
  and only if $h f_0=f_1$ for a homeomorphism $h\colon(S^2,C)\selfmap$
  that is isotopic to the identity.
\end{lem}
\begin{proof}
  If there exists an isotopy $(h_t)$ witnessing $\one\approx_C h$, then
  $f_t\coloneqq h_t f_0$ witnesses $f_0\approx_C f_1$. Conversely,
  since all critical values of $f_t$ are frozen in $A$, the set
  $f_t^{-1}(y)$ moves homeomorphically for every $y\in S^2$
  (equivalently, no critical points collide). Therefore, we may factor
  $f_t = h_t f_0$, with $h_t(z)$ the trajectory of $z\in f_t^{-1}(y)$;
  this defines an isotopy from $\one$ to $h\coloneqq h_1$.
\end{proof}

\begin{lem}\label{lem:combinatorial equivalence}
  Two Thurston maps $f_0\colon(S^2,A_0)\selfmap$ and
  $f_1\colon(S^2,A_1)\selfmap$ are combinatorially equivalent if and
  only if there exists homeomorphisms
  $\phi_0,\phi_1\colon(S^2,A_0)\to(S^2,A_1)$ with
  $\phi_0\circ f_0=f_1\circ\phi_1$ and $\phi_0$ isotopic to $\phi_1$
  rel $A_0$:
  \[\begin{tikzcd}
      (S^2,A_0)\ar[r,"\phi_1"]\ar[d,"f_0" left]\ar[dr,phantom,"\circlearrowleft"] & (S^2,A_1)\ar[d,"f_1"]\\
      (S^2,A_0)\ar[r,"\phi_0"] & (S^2,A_1)
    \end{tikzcd}\]
  commutes up to isotopy rel $A_0$.
\end{lem}
\begin{proof}
  Given a path $(f_t)$ of Thurston maps connecting $f_0$ to $f_1$,
  factor each map $f_t$ as $f_t=\lambda_t\circ f_0\circ\rho_t^{-1}$
  for homeomorphisms $\lambda_t,\rho_t\colon(S^2,A_0)\to(S^2,A_t)$
  depending continuously on $t$, with $\lambda_0=\rho_0=\one$. This
  can be done by first choosing $\lambda_t$ moving $A_0$ to $A_t$, and
  then applying Lemma~\ref{lem:factorization}.  Define then
  $\phi_t=\rho_1\circ\rho_{1-t}^{-1}\circ\lambda_{1-t}$, and note that
  $\phi_t\colon(S^2,A_0)\to(S^2,A_1)$ is a homeomorphism with
  $\phi_0=\lambda_1$ and $\phi_1=\rho_1$ so
  $\phi_0\circ f_0=f_1\circ\phi_1$. Furthermore,
  $(\phi_t)_{t\in[0,1]}$ is an isotopy rel $A_0$ from $\phi_0$ to
  $\phi_1$.

  Conversely, let $\phi_0,\phi_1\colon(S^2,A_0)\to(S^2,A_1)$ be
  isotopic homeomorphisms, and let $(\phi_t)_{t\in[0,1]}$ be an
  isotopy rel $A_0$ between them. Let $(\lambda_t)_{t\in[0,1]}$ be an
  isotopy (possibly non-constant on $A_0$) from $\lambda_0=\one$ to
  $\lambda_1=\phi_0$.  Define $A_t\coloneqq\lambda_t(A_0)$ and
  $f_t\coloneqq\lambda_t\circ
  f_0\circ\phi_t^{-1}\circ\phi_0\circ\lambda_t^{-1}$, and note
  $f_t(A_t)\subseteq A_t$, so $f_t$ is a Thurston map along an isotopy
  from $f_0$ to $f_1$.
\end{proof}

\begin{cor}[Kameyama~\cite{kameyama:thurston}*{Corollary~3.7}; see
  also~\cite{nekrashevych:ssg}*{Theorem~6.5.2}]\label{cor:kameyama}
  Two Thurston maps $f\colon (S^2,A)\selfmap,g\colon (S^2,C)\selfmap$ with $|C|,|A|\ge 2$ are combinatorially equivalent if and only
  if their bisets $B(f)$ and $B(g)$ are conjugate.
\end{cor}
\begin{proof}
  If $f\colon(S^2,A)\selfmap$ and $g\colon(S^2,C)\selfmap$ are
  combinatorially equivalent, then by Lemma~\ref{lem:combinatorial
    equivalence} there exists a homeomorphism
  $\phi\colon(S^2,A)\to(S^2,C)$ such that $f$ and
  $\phi^{-1}\circ g\circ\phi$ are isotopic rel $A$; so $B(f)$ and
  $B(\phi)\otimes B(g)\otimes B(\phi)^\vee$ are isomorphic by
  Theorem~\ref{thm:dehn-nielsen-baer+}, and $B(f)$ and $B(g)$ are
  conjugate.

  Conversely, if $\subscript G B(f)_G$ and $\subscript H B(g)_H$ are
  conjugate with respective sphere groups $G=\pi_1(S^2\setminus A,*)$
  and $H=\pi_1(S^2\setminus C,\dagger)$, let $\varphi\colon G\to H$ be
  a sphere group isomorphism such that the bisets $B(f)$ and
  $B_\varphi\otimes B(g)\otimes B_\varphi^\vee$ are isomorphic. By
  Theorem~\ref{thm:dehn-nielsen-baer}, the homomorphism $\varphi$ may
  be (uniquely) realized as $\phi_*$ for a homeomorphism
  $\phi\colon(S^2,A)\to(S^2,C)$, so $f$ and
  $\phi^{-1}\circ g\circ\phi$ are isotopic rel $A$ by
  Theorem~\ref{thm:dehn-nielsen-baer+}, so $f$ and $g$ are
  combinatorially equivalent by Lemma~\ref{lem:combinatorial
    equivalence}.
\end{proof}

\section{Decompositions of sphere bisets}
\label{s:DecSphBis}
We now explore how spheres and sphere bisets can be decomposed into
simpler components. As before, we describe side by side the topology
and its associated group theory.

We will use the following conventional notations. Two closed curves
$\ell,\ell'$ in $S^2\setminus A$ are \emph{homotopic rel $A$} if their
parameterizations $\ell,\ell'\colon S^1\to S^2\setminus A$ are
homotopic rel $A$. Given a subset $T\subset S^2$, we say that $T$ is
\emph{homotopic rel $A$} to a point $x\in S^2$ (or $T$ is
\emph{contractible} to $x$) if there is a homotopy
$h\colon T\times [0,1)\to S^2$ rel $A$, namely $h(-,0)=\one$ and
$h(y,t)=h(y,0)$ for all $y\in T\cap A$, and $h(-,t)\to x$ as $t\to
1$. Note that if $T$ is homotopic to $x$, then either
$T\cap A\subseteq\{x\}$.

Finally, we say that $T\subset S^2$ is \emph{homotopic to a curve
  $\ell$} if there is a homotopy $h\colon T\times [0,1]\to S^2$ rel
$A$ and a curve $\ell'\subset T$ such that $h(-,0)=\one$,
$h(T,1)\subset \ell$, and the restriction of $h$ to $\ell'\times[0,1]$
is a homotopy rel $A$ between $\ell'$ and $\ell$.

\subsection{Multicurves}\label{ss:multicurves}
Let $(S^2,A)$ be a sphere. A \emph{multicurve} $\CC$ on $(S^2,A)$ is a
collection of non-trivial, non-peripheral, mutually non-homotopic,
non-intersecting, simple closed curves on $S^2\setminus A$.

Psychologically, a sphere with multicurve $(S^2,A,\CC)$ should be
thought of as a sphere on which the curves $\CC$ are extremely short,
so that the sphere looks more like a cactus of the genus
\emph{opuntia}.

If $(S^2,A,\CC)$ be a noded sphere, one may cut $S^2$ along $\CC$, and
shrink the boundary components to punctures. Algebraically, this
amounts to the following. Each curve in $\CC$ may be expressed as a
conjugacy class $\Gamma$ in $G\coloneqq\pi_1(S^2,A)$. Choose in each
$\Gamma\in\CC$ a representative $t_\Gamma\in\Gamma$. Then $G$
decomposes as a tree of groups, with one vertex per component $S$ of
$S^2\setminus\CC$ and associated vertex group $\pi_1(S)$, and one edge
per curve $\Gamma\in\CC$ with associated edge group $\langle
t_\Gamma\rangle$. The underlying graph is a tree.

So as to
follow~\cite{bartholdi-dudko:bc1}*{Definition~\ref{bc1:defn:gog_1dimcovers}},
we consider rather the barycentric subdivision of the above tree of
groups:
\begin{defn}[Tree of groups of a multicurve]\label{defn:gfofgroups}
  Let $(S^2,A,\CC)$ be a sphere given with a multicurve $\CC$. Its
  associated \emph{sphere tree of groups} $\gf$ is defined as follows.

  Consider the finite $1$-dimensional cover
  \cite{bartholdi-dudko:bc1}*{Definition~\ref{bc1:defn:1dimcovers}} of
  $(S^2,A)$ by components of $\CC$ and the set $\mathscr S$ of
  closures of components of $S^2\setminus\CC$; call the former
  \emph{curves} and the latter \emph{small spheres}.  The underlying
  graph of $\gf$ has one vertex per curve and one per small sphere. It
  has four edges per curve, with an edge connecting each curve to its
  two neighbouring spheres and back. Thus
  $V=\CC\sqcup\mathscr S$ is the vertex set, and
  $E=\{(v,w)\in\CC\times\mathscr S\sqcup\mathscr S\times\CC\mid
  v\cap w\neq\emptyset\}$ is the edge set. We have $(v,w)^-=v$ and
  $\overline{(v,w)}=(w,v)$.

  Consider a vertex $v\in V$; it is represented by a subset $S_v$ of
  $S^2$, and is either a curve or a small sphere. Choose a basepoint
  $*_v\in S_v\setminus A$, and for each edge $(v,w)$ choose a path
  $\ell_{v,w}$ in $(S_v\cup S_w)\setminus A$ from $*_v$ to $*_w$,
  assuming $\ell_{v,w}^{-1}=\ell_{w,v}$. The group associated to the
  vertex $v$ is the fundamental group $G_v=\pi_1(S_v\setminus
  A,*_v)$. The group associated with the edge $(v,w)$ is $G_v$ if
  $v\in\CC$ and is $G_w$ if $w\in\CC$. The homomorphism $G_{(v,w)}\to
  G_{(w,v)}$ is the identity. The homomorphism $G_{(v,w)}\mapsto G_v$
  is the identity if $v\in\CC$ and is
  $\gamma\mapsto\ell_{v,w}\#\gamma\#\ell_{w,v}$ otherwise.
\end{defn}  
In particular, the groups associated with small spheres are sphere
groups (there is no distinction, from their point of view, between
boundary components in $A$ or in $\CC$), and the vertex groups
associated with curves are infinite cyclic (they could be thought as
sphere groups of $S^2\setminus \{0,\infty\}$), as are the edge
groups. The underlying graph is a tree.  The van Kampen theorem,
see~\cite{bartholdi-dudko:bc1}*{Theorem~\ref{bc1:thm:vankampen}},
asserts that the fundamental group of $(S^2,A)$ is isomorphic to the
fundamental group of $\gf$.  We recall
from~\cite{bartholdi-dudko:bc1}*{Lemma \ref{bc1:lem:gfCongClass}} that
a tree of groups is defined uniquely up to congruence
(see~\cite{bartholdi-dudko:bc1}*{Definition~\ref{bc1:def:congruence}}).

For a small sphere $\overline S\subset(S^2,A)$, we write $\widehat S$
for the quotient of $\overline S$ in which every boundary curve is
shrunk to a point; so $\widehat S$ is a topological sphere, marked by
the image of $A$ and the boundary curves.

\subsubsection{Algebraic multicurves}
Recall from~\cite{bartholdi-dudko:bc1}*{\S\ref{bc1:ss:vk}} that the
\emph{barycentric subdivision} of a graph $\gf=V\sqcup E$ is a new
graph $\gf'=V'\sqcup E'$ with vertex set $V'=\gf/\{x=\overline x\}=V
\sqcup E/\{x=\overline x\}$ and edge set $E'=E\times\{+,-\}$; for
$e\in E$ and $\varepsilon\in\{\pm1\}$, set
$(e,\varepsilon)^\varepsilon=e^\varepsilon$ and
$(e,\varepsilon)^{-\varepsilon}=[e]$ and
$\overline{(e,\varepsilon)}=(\overline e,-\varepsilon)$. Let us say
that a vertex $v\in V'$ is \emph{old} if $v\in V$ and is \emph{new} if
$v\in E/\{x=\overline x\}$.  Old and new vertices form a bipartite
structure of $\gf'$. We are now ready to state an algebraic
counterpart of Definition~\ref{defn:gfofgroups}; the equivalence of
the objects is proven in Lemma~\ref{lem:MultToAlgMult}.
\begin{defn}[Sphere groups and algebraic multicurves]\label{def:multicurve}
A \emph{stable sphere tree of groups} is 
\begin{enumerate}
\item a tree $\gf$ that is the barycentric subdivision of a smaller
  tree; old vertices of $\gf$ are called \emph{sphere vertices} while
  new vertices of $\gf$ are called \emph{curve vertices};
\item a sphere group at every vertex, so that curve vertex groups are
  cyclic (thought of as $\pi_1(\hC\setminus\{0,\infty\})$) while sphere
  vertex groups have at least $3$ peripheral conjugacy classes;
\item a cyclic group at every edge, which embeds into vertex groups by
  mapping a generator to an element of a peripheral conjugacy class,
  in such a manner that different edge groups attach to different
  peripheral conjugacy classes.
\end{enumerate}

A peripheral conjugacy class in a vertex group $G_v$ is \emph{vacant}
if it doesn't intersect the image of any edge group. Clearly, vacant
peripheral classes in a sphere tree of groups $\gf$ are identified
with peripheral conjugacy classes in $\pi_1(\gf)$.

  Let $G$ be a sphere group. An algebraic \emph{multicurve} is a
  collection $\CC$ of distinct, non-peripheral conjugacy classes in
  $G$, such that there exists a decomposition of $G$ as a sphere tree
  of groups with edges in bijection with $\CC$, such that the edge
  group associated with $\Gamma\in\CC$ is cyclic and generated by a
  representative $t_\Gamma\in\Gamma$.
\end{defn}

\noindent The following extends Lemma~\ref{lem:sphere to sphere group} to
spheres with multicurves:
\begin{lem}\label{lem:MultToAlgMult}
  Let $(S^2,A)$ be a marked sphere with multicurve $\CC$, and choose
  $*\in S^2\setminus A$. Then the tree of groups decomposition $\gf$
  of $G=\pi_1(S^2\setminus A,*)$ along $\CC$ is a sphere tree of
  groups as in Definition~\ref{def:multicurve}. Via this
  decomposition, peripheral conjugacy classes in $G$ form vacant
  peripheral conjugacy classes in $\gf$ while conjugacy classes of $G$
  encoding $\CC$ form non-vacant peripheral conjugacy classes of
  $\gf$.
  
   Conversely, to a sphere tree of groups corresponds a marked
  sphere with multicurve, which is unique up to homeomorphism. 
\end{lem}
\begin{proof}
  It is easy to check that the tree of groups decomposition of
  $\pi_1(S^2\setminus A,*)$ along $\CC$ satisfies
  Definition~\ref{def:multicurve}. Conversely, given a sphere tree of
  groups as in Definition~\ref{def:multicurve}, construct a sphere
  with multicurve by first realizing every small sphere vertex group
  as the fundamental group of a sphere with marked points; remove a
  small disk around each marked point whose corresponding peripheral
  class is not vacant, and attach cylinders, with a marked curve along
  their core, between these spheres by gluing their boundary to the
  boundary of the removed disks. The uniqueness of the obtained sphere
  follows from Theorem~\ref{thm:dehn-nielsen-baer}.
\end{proof}

Note that there exists a general notion, that of ``JSJ decomposition''
of groups~\cite{rips-sela:jsj}, over cyclic subgroups. One usually
considers it for freely indecomposable groups; while, on the other
hand, we apply it here to study bisets over decompositions of
free groups over cyclic subgroups.

\begin{thm}\label{thm:mc to gog}
  There is an algorithm that, given $(G,\Gamma_i)$ a sphere group and $\CC$
  a collection of conjugacy classes in $G$, computes a sphere tree of
  groups decomposition as in Definition~\ref{def:multicurve} (if $\CC$ is
  an algebraic multicurve), or returns \texttt{fail} (if it is not).
\end{thm}
\begin{proof}
  There are algorithms
  (see~\cites{birman-series:sccs,cohen-lustig:intersectionnumbers})
  that compute geometric intersection numbers of curves; and, in
  particular, whether a conjugacy class represents a simple closed
  curve.

  Given a multicurve $\CC$, apply such an algorithm to each curve
  $g^G$ of $\CC$, in turn. If $g^G$ is a simple closed curve, then
  $\{1,\dots,n\}$ may be partitioned as
  $\{i_1,\dots,i_s\}\sqcup\{j_1,\dots,j_t\}$ in such a manner that $g$
  may be written as a product of conjugates of Hurwitz generators of
  $G$ as
  \[g=\gamma_{i_1}^{u_1}\cdots\gamma_{i_s}^{u_s}=(\gamma_{j_1}^{v_1}\cdots\gamma_{j_t}^{v_t})^{-1}\]
  for $u_1,\dots,u_s,v_1,\dots,v_t\in G$. The partition
  $\{i_1,\dots,i_s\}\sqcup\{j_1,\dots,j_t\}$ is obtained by computing
  the image of $g$ in $G/[G,G]\cong\Z^n/(1,\dots,1)$ and writing it as
  $e_{i_1}+\cdots+e_{i_s}=-(e_{j_1}+\cdots+e_{j_t})$ with
  $e_1,\dots,e_n$ the image of the standard basis of $\Z^n$.

  The claimed decomposition of $g$ exists for topological reasons:
  picture the sphere with $*$ at the north pole, all $a_i$ on the
  equator, and all lollipop generators along meridians except for
  their little loop around $a_i$. Consider $g$ as a closed curve on
  $S^2$. There is then an isotopy of $S^2$ that brings $g$ into a thin
  neighbourhood of the equator, turning once around it. Under the
  isotopy, every lollipop generators gets deformed to a conjugate of
  itself. The curve $g$ is isotopic to the product of the original
  lollipop generators $\gamma_i$ that lie above it, namely such that
  the curve $g$ passes below $a_i$ as it revolves around $S^2$.  Once
  we know that such a decomposition exists, we can find it by
  enumerating all choices of $u_1,\dots,u_s,v_1,\dots,v_t$.

  This shows how to split $G$ as an amalgamated free product over
  $\Z$: the edge group $\Z$ is generated by $g$, one vertex group is
  the sphere group $G'$ generated by
  $\gamma'_1=\gamma_{i_1}^{u_1},\dots,\gamma'_s=\gamma_{i_s}^{u_s},g^{-1}$
  and relation $\gamma'_1\cdots\gamma'_s g^{-1}=1$, and the other
  vertex group $G''$ is similarly generated by $g$ and the
  $\gamma''_1=\gamma_{j_1}^{v_1},\dots,\gamma''_t=\gamma_{j_t}^{v_t}$
  with relation $\gamma''_1\cdots\gamma''_t g=1$.

  The remaining curves in $\CC\setminus\{g^G\}$ should now be
  rewritten in terms of the generators of $G'$ or of $G''$. If this is
  impossible, then the curves in $\CC$ are not disjoint. Otherwise,
  proceed inductively with $G'$ and $G''$, which have smaller
  complexity.
\end{proof}

\subsection{Invariant multicurves}\label{ss:invariant mc}
Consider a sphere map $f\colon(S^2,C)\to(S^2,A)$ and a multicurve
$\CC$ in $(S^2,A)$. Since $f$ is a covering away from $A$, there is a
\emph{lifted} multicurve $f^{-1}(\CC)$ in $(S^2,C)$.

Algebraically, $f$ is represented as an $H$-$G$-biset $B(f)$, and
$\CC$ is represented by conjugacy classes in $G$ up to
$g^G\leftrightarrow (g^{-1})^G$, which we denote by $g^{\pm G}$. The
multiset of lifts, Definition~\ref{defn:lifts}, can be defined for
$g^{\pm}$ since the inverse of a lift is the lift of the inverse. The
multiset of lifts of a multicurve under $B(f)$ is the $f$-preimage of
the multicurve.

We consider more generally the situation of a sphere map
$f\colon(S^2,C,\DD)\to(S^2,A,\CC)$, meaning $\DD$ is isotopic to a
subset of $f^{-1}(\CC)$. For our purpose, the most valuable
information about lifts of multicurves is contained in the
\emph{Thurston matrix},
see~\cite{bartholdi-dudko:bc1}*{\eqref{bc1:eq:conjclasses}}: it is the
linear operator $T_{f,\CC}\colon\Q\CC\to\Q\DD$ given by
\begin{equation}\label{eq:thurston matrix}
  T_{f,\CC}(\gamma)=\sum_{\substack{\varepsilon\in
      f^{-1}(\gamma)\\\varepsilon\approx\delta\in\DD}}\frac1{\deg(f\restrict\varepsilon\colon\varepsilon\to\gamma)}\delta.
\end{equation}
Here by $\deg$ one means the usual positive degree of $f$; i.e.\ for
$d\in\Z$ the degree of $z^d\colon\{|z|=1\}\selfmap$ is $|d|$.

If $\subscript H B_G$ be a sphere biset, $\CC$ be a multicurve in $G$
and $\DD$ be a multicurve in $H$, both represented as collections of
conjugacy classes, then
by~\cite{bartholdi-dudko:bc1}*{\S\ref{bc1:ss:ccgroups}} there is a
linear operator $T_{B,\CC}\colon\Q\CC\to\Q\DD$, which coincides
with~\eqref{eq:thurston matrix} in case $B=B(f)$.

In Proposition~\ref{prop:MCBedge} we shall later interpret the
Thurston matrix as giving the structure of abelian subbisets for the
mapping class group $\Mod(G)$.

Consider a Thurston map $f\colon(S^2,A)\selfmap$. A multicurve $\CC$
is \emph{$f$-invariant} if $\CC=f^{-1}(\CC)$, up to isotopy,
identifying isotopic curves and removing inessential curves. The
Thurston matrix $T_{f,\CC}$ is then an endomorphism of $\Q\CC$.

The importance of invariant multicurves was made clear by a
fundamental result of Thurston characterizing Thurston maps that are
combinatorially equivalent to rational maps. An \emph{annular
  obstruction} for a Thurston map $f$ is an $f$-invariant multicurve
$\CC$ such that the spectral radius of $T_{f,\CC}$ is $\ge1$.
\begin{thm}[Thurston~\cite{douady-h:thurston}]\label{thm:thurston}
  Let $f\colon(S^2,P_f)\selfmap$ be a Thurston map with hyperbolic
  orbifold, see~\S\ref{ss:orbispheres}. Then $f$ is combinatorially
  equivalent to a rational map if and only if $f$ admits no annular
  obstruction.  Furthermore, in that case the rational map is unique
  up to conjugation by M\"obius transformations.
\end{thm}

\subsection{Decomposition along multicurves}
Consider again a sphere map $f\colon(S^2,C,\DD)\to(S^2,A,\CC)$ with
multicurves satisfying as before $\DD\subseteq f^{-1}(\CC)$. By
Theorem~\ref{thm:mc to gog}, there are sphere tree of groups decompositions
$\gf$ of $G=\pi_1(S^2\setminus A,*)$ and $\gfY$ of
$H=\pi_1(S^2\setminus C,\dagger)$.  Using the van Kampen Theorem for
bisets, we shall decompose the $H$-$G$-biset $B(f)$ into an
$\gfY$-$\gf$-tree of bisets $\gfB(f)$.

Up to homotopy, the map $f\colon(S^2,C,\DD)\to(S^2,A,\CC)$ may be
written as a correspondence
\begin{equation}
\label{eq:corr:f_i}(S^2,C,\DD)\overset i\leftarrow(S^2,f^{-1}(A),f^{-1}(\CC))\overset
f\to(S^2,A,\CC);
\end{equation}
the map $f$ is the same as the original map $f$, but is now a
covering, and we specify the map $i$ as follows. It first forgets all
points in $f^{-1}(A)\setminus C$ and all curves in $f^{-1}(\CC)$ that
are not isotopic to curves in $\DD$. It squeezes all annuli between
the remaining curves in $f^{-1}(\CC)$ that are isotopic in
$S^2\setminus C$, and maps them to the corresponding curve in
$\DD$. This defines uniquely $i$ as a monotone map (preimages of
connected sets are connected), up to isotopy.

\noindent Let us say that a curve $\gamma\in f^{-1}(\CC)$ is
\begin{idescription}
\item[essential] if $\gamma$ is homotopic rel $C$ to a curve in $\DD$,
\item[non-essential] otherwise.   
\end{idescription}

\begin{lem}\label{lem:PropOfEssCurves}
  Consider $\gamma\in f^{-1}(\CC)$. If both connected components of
  $S^2\setminus \gamma$ contain essential curves of $f^{-1}(\CC)$,
  then $\gamma$ is non-trivial rel $C$.
\end{lem}
\begin{proof}
  If $\gamma$ is a non-essential curve, then it surrounds a disc
  homotopic rel $C$ to a point in $S^2$. Therefore, all curves in this
  disc are peripheral or trivial rel $C$.
\end{proof}

The respective multicurves $\DD$, $f^{-1}(\CC)$ and $\CC$
in~\eqref{eq:corr:f_i} define $1$-dimensional covers as in
Definition~\ref{defn:gfofgroups};
following~\cite{bartholdi-dudko:bc1}*{\S\ref{bc1:ss:VanKampCorr}} we
define the tree of bisets associated with $f$ as follows:

\begin{defn}[Sphere tree of bisets $\gfB(f)$]\label{defn:gfofbisets}
  Consider a map $f\colon(S^2,C,\DD)\to(S^2,A,\CC)$ and decompose it
  as in~\eqref{eq:corr:f_i}. Let $\gfY$ and $\gf$ be the sphere trees
  of groups associated with $(S^2,C,\DD)$ and $(S^2,A,\CC)$
  respectively as in Definition~\ref{defn:gfofgroups}. Then the
  $\gfY$-$\gf$ \emph{sphere tree of bisets $\gfB=\gfB(f)$} is
  constructed as follows.

  Vertices of $\gfB$ are in bijection with components of
  $f^{-1}(\CC)$ and closures of components of $S^2\setminus
  f^{-1}(\CC)$; call the former \emph{curve vertices} and the latter
  \emph{sphere vertices}.  Consider a vertex $z$ of $\gfB(f)$;
  it is represented by a subset $S_z$ of $S^2$, and is either a curve
  or a small sphere. The graph morphism $\rho\colon \gfB\to
  \gf$ is given by the covering $f$ so that $f(S_z)=S_{\rho(z)}$.

  Define the map $\lambda \colon \gfB \to \gfY$ such
  that $i(S_z)\subset S_{\lambda(z)}$ for every object
  $z\in \gfB$ and such that
  \begin{enumerate}\renewcommand\theenumi{TB\ensuremath{{}_\arabic{enumi}}}
  \item $\lambda$ is a monotone map; i.e. $\lambda^{-1}(z)$ is a subtree
    of $\gfB$ for every $z\in \gfB$;
    and\label{cond:1:defn:gfofbisets}
  \item $\lambda$ maps essential curve vertices to essential curve
    vertices.\label{cond:2:defn:gfofbisets}
  \end{enumerate}

  \noindent For every $z\in \gfB$ the
  $G_{\lambda(z)}$-$G_{\rho(z)}$-biset attached to $z$ is
  \begin{equation}\label{eq:dfn:B_z}
    B(i\colon S_z\setminus f^{-1}(A) \to S_{\lambda(z)}\setminus C)^\vee\otimes B(f\colon S_z\setminus f^{-1}(A) \to S_{\rho(z)}\setminus A).\qedhere
  \end{equation} 
\end{defn}
\noindent
By~\cite{bartholdi-dudko:bc1}*{Theorem~\ref{bc1:thm:vankampenbis}}, the
fundamental biset of $\gfB(f)$ is isomorphic to $B(f)$.

In the next Lemma~\ref{lem:PropOfLambda} we will show that
Conditions~\eqref{cond:1:defn:gfofbisets}
and~\eqref{cond:2:defn:gfofbisets} uniquely specify
$\lambda \colon \gfB \to \gfY$; in particular these
conditions are realizable. Let us also note that without these
conditions there are in general many choices of $\lambda$, and
therefore many slightly different graphs of bisets as
in~\cite{bartholdi-dudko:bc1}*{\S\ref{bc1:ss:VanKampCorr}} decomposing
$B$.

\begin{lem}[Uniqueness of $\gfB$]\label{lem:UniqSphGrBis}
  Definition~\ref{defn:gfofbisets} determines the map $\lambda \colon
  \gfB\to \gfY$ uniquely. Conditions~\eqref{cond:1:defn:gfofbisets}
  and~\eqref{cond:2:defn:gfofbisets} are equivalent to 
  \begin{enumerate}\renewcommand\theenumi{TB'\ensuremath}
  \item for every curve vertex $y\in \gfY$ and every vertex
    $z\in \gfB$ we have $\lambda(z)=y$ if and only if
    $i(S_z)\subset S_y$.\label{cond1:lem:UniqSphGrBis}
  \end{enumerate}
  Up to congruence of trees of bisets (see
  \cite{bartholdi-dudko:bc1}*{Definition~\ref{bc1:def:congruence}}),
  the sphere tree of bisets $\gfB=\gfB(f)$ is determined by the
  isotopy type of $f\colon(S^2,C,\DD)\to(S^2,A,\CC)$.
  
  Furthermore, for every curve vertex $e\in \gfY$ we have
  $i^{-1}(S_y)= \bigcup_{z\in \lambda^{-1}(y)}S_z$ and this set is an
  annulus (possibly a closed curve) bounded by two (possibly equal)
  curves in $ f^{-1}(\CC)$ that are isotopic rel $C$ to $S_e$.
\end{lem}
\begin{proof}
  Clearly, Condition~\eqref{cond1:lem:UniqSphGrBis} implies the
  uniqueness of $\lambda$. Let us prove that
  Conditions~\eqref{cond:1:defn:gfofbisets}
  and~\eqref{cond:2:defn:gfofbisets} are equivalent
  to~\eqref{cond1:lem:UniqSphGrBis}. This last property implies the
  monotonicity of $\lambda$ (by monotonicity of $i$) and, clearly,
  $\lambda$ maps essential curve vertices to curve
  vertices. Conversely, for every curve vertex $y\in \gfY$ the set
  $i^{-1}(S_y)$ is an annulus (possibly a closed curve) bounded by two
  curves $S_{z}, S_{z'}\in f^{-1}(\CC)$ isotopic rel $C$ (possibly
  with $S_{z}=S_{z'}$). By Condition~\eqref{cond:2:defn:gfofbisets} we
  have $\lambda(z)=\lambda(z')$, and by monotonicity of $\lambda$ we
  have $i^{-1}(S_y)\subset \bigcup_{z\in
    \lambda^{-1}(y)}S_z$. Therefore,
  $i^{-1}(S_y)= \bigcup_{z\in \lambda^{-1}(y)}S_z$ because
  $i(S_v)\subset S_{\lambda(v)}$ for all vertices $v\in\gfB$, and the
  claim of the lemma follows.  The second claim follows from
  \cite{bartholdi-dudko:bc1}*{Lemma~\ref{bc1:lem:gfCongClass}} because
  $\lambda,\rho$ are uniquely specified. The last claim has already
  been verified.
\end{proof}

Note here one of the reasons we considered in
Definition~\ref{defn:gfofgroups} the barycentric subdivision of the
simpler decomposition with one vertex per small sphere and one edge
per curve: our definition of graphs does not allow vertices to be
mapped to midpoints of edges, as we would like to do in case $z$ is a
small sphere vertex and $i(S_z)$ is homotopic to a curve. Without
taking a barycentric subdivision, we would have been forced to choose
one of its two neighbouring spheres to which to map $z$; and no such
choice can be made canonical.

\begin{thm}\label{thm:DecompOfBiset}
  There is an algorithm that, given a sphere biset $\subscript H B_G$ and
  algebraic multicurves $\CC$ in $G$ and $\DD$ in $H$ with $\DD$
  contained in the $B$-lift of $\CC$, computes the decomposition of
  $B$ as a sphere tree of bisets.
\end{thm}
\begin{proof}
  By Theorem~\ref{thm:dehn-nielsen-baer+} we may assume that $\subscript H B_G$
  is the biset of a sphere map $f\colon (S^2,C)\to (S^2,A)$. Then
  $\CC$ and $\DD$ define multicurves, still denoted by $\CC$ and
  $\DD$, in $(S^2,A)$ and $(S^2,C)$.

  Choose $b\in B$ and decompose $\subscript H B_G = H b G_b \otimes \subscript{G_b} G_G$
  as in~\eqref{eq:DecOfSphBis}. Recall from
  Lemmas~\ref{lem:CorrOfSphBis} and~\ref{lem:CorrOfSphBis2} that this
  decomposition is the algebraic counterpart
  of~\eqref{eq:CorrOfSphMaps} and that $G_b$ is identified with the
  fundamental group of $(S^2,f^{-1}(A))$.

  Denote by $(\Gamma_a)_{a\in A}$ the peripheral conjugacy classes in
  $G$, by $(\Delta_c)_{c\in C}$ the peripheral conjugacy classes in
  $H$, and by $(\Xi_a)_{a\in f^{-1}( A)}$ the peripheral conjugacy
  classes in $G_b$. Note that $(\Xi_a)_{a\in f^{-1}( A)}$ is easily
  computable as the set of all lifts of $(\Gamma_a)_{a\in A}$ through
  $\subscript{G_b}{G}_G$, see Lemma~\ref{lem:CorrOfSphBis}.

  Denote by $f^{-1}(\CC)$ the set of all lifts of $\CC$ through
  $\subscript{G_b}{G}_G$. Again $f^{-1}(\CC)$ is easily computable and is the
  algebraic counterpart of the multicurve $f^{-1}( \CC)$ in
  $(S^2,f^{-1}( A))$.  By Theorem~\ref{thm:mc to gog}, the sphere
  groups $G,G_b,H$ can be decomposed relatively to $\CC,f^{-1}( \CC),\DD$
  as trees of groups $\gf,\mathfrak K,\gfY$ respectively.
  
  By definition, the underlying tree of bisets decomposing $B$ is the
  graph $\mathfrak K$. The map $\rho\colon\mathfrak K\to\gf$ is
  induced by inclusion: consider a vertex $z\in\mathfrak K$, with
  corresponding group $G_z$. Then $\rho(z)$ is the unique vertex of
  $\gf$ such that $G_z\subseteq G_{\rho(z)}$.

  By lifting curves in $f^{-1}(\CC)$ through $H b G_b$ we compute the
  set of essential curve vertices in $\mathfrak K$ as well the
  $\lambda$ on this set. By monotonicity of $\lambda$ (see the second
  claim in Lemma~\ref{lem:UniqSphGrBis}) it has a unique, and thus
  computable, extension on the whole graph $\mathfrak K$.

  As a graph, $\gfB$ is the underlying graph of $\mathfrak K$.
  Consider now a vertex $z\in\mathfrak K$; we define a
  $H_{\lambda(z)}$-$G_{\rho(z)}$-biset $B_z$, to be put at $z$ in the
  tree of bisets, as follows. We compute $b G_z=L_z b$ for a finitely
  generated subgroup $L_z$ of $H$; this is possible since elements of
  $G_z$ stabilize $b$. We find (e.g.\ by enumeration) an element
  $h\in H$ with $L_z\le H_{\lambda(z)}^h$. This is possible because
  $i(S_z)\subset S_{\lambda(z)}$ in the notation
  of~\eqref{eq:dfn:B_z}. We finally set
  $B_z\coloneqq H_{\lambda(z)}h b G_{\rho(z)}$, the subbiset of $B$
  given in~\eqref{eq:dfn:B_z}.
\end{proof}

\subsection{Kameyama's theorem extended to spheres with multicurves}
We are ready to extend Kameyama's algebraic characterization of
combinatorial equivalence, taking multicurves into account. This will
lead to algorithmic constructibility of sphere decompositions.

\begin{thm}\label{thm:kameyama-mc}
  Let $f,g\colon(S^2,C,\DD)\to(S^2,A,\CC)$ be two sphere maps and let
  $\gfY$ and $\gf$ be the sphere trees of groups associated with
  $(S^2,C,\DD)$ and $(S^2,A,\CC)$ respectively as in
  Definition~\ref{defn:gfofgroups}. Then $f$ is isotopic to $g$ by an
  isotopy mapping $\DD$ to $\CC$ if and only if the sphere trees of
  bisets $\subscript\gfY\gfB(f)_\gf$ and
  $\subscript{\gfY}{\gfB(g)}_\gf$ are isomorphic.
\end{thm}
\begin{proof}
  Since the sphere tree of bisets $\gfB(f)$ is constructed out of
  $B(f)$ and $\CC,\DD$ viewed as conjugacy classes in $\pi_1(S^2,C)$
  and $\pi_1(S^2,A)$; and similarly $\gfB(g)$ is constructed
  out of $B(g)$ and $\CC,\DD$, the trees of bisets $\gfB(f)$
  and $\gfB(g)$ are isomorphic if and only if $B(f)$ and $B(g)$
  are isomorphic by an isomorphism preserving the dynamics on
  multicurves.

  By Theorem~\ref{thm:dehn-nielsen-baer+}, this happens if and only if
  $f$ is isotopic to $g$ by an isotopy preserving $\DD$ and $\CC$
\end{proof}

We recall that two graphs of bisets $\subscript\gf\gfB_\gf$ and
$\subscript\gfY\gfC_\gfY$ are called \emph{conjugate} if there exists
a biprincipal
(see~\cite{bartholdi-dudko:bc1}*{\S\ref{bc1:ss:BiprGrBis}}) graph of
bisets $\subscript\gf\gfI_\gfY$ such that $\gfB\otimes_\gf\gfI$ and
$\gfI\otimes_\gfY\gfC$ are isomorphic, namely have same underlying
graphs and isomorphic vertex and edge bisets. If
$\subscript\gf\gfB_\gf$ and $\subscript\gfY\gfC_\gfY$ are sphere trees
of bisets, then $\subscript\gf\gfI_\gfY$ is also required to be a
sphere tree of bisets.

Two Thurston maps $f_0\colon(S^2,A_0,\CC_0)\selfmap$ and
$f_1\colon(S^2,A_1,\CC_1)\selfmap$ with respective invariant
multicurves $\CC_0,\CC_1$ are \emph{combinatorially equivalent} if
there exists a path of Thurston maps
$(f_t\colon(S^2,A_t,\CC_t)\selfmap)_{t\in[0,1]}$ connecting them.

Recall finally
from~\cite{bartholdi-dudko:bc1}*{Definition~\ref{bc1:def:congruence}}
that two graphs of bisets $\subscript\gf\gfB_\gf$ and
$\subscript{\gf'} \gfC_{\gf'}$ are called \emph{conjugate} if
$\subscript\gf\gfB_\gf$ is isomorphic to
$\gfI \otimes \subscript{\gf'} \gfC_{\gf'} \otimes \gfI^\vee$ for a
biprincipal graph of bisets $\subscript\gf\gfI_{\gf'}$.

\begin{cor}\label{cor:kameyama-mc}
  Let $f\colon(S^2,A,\CC)\selfmap$ and $g\colon(S^2,C,\DD)\selfmap$ be
  Thurston maps with respective invariant multicurves $\CC,\DD$. Then
  $f,g$ are combinatorially equivalent if and only if the sphere trees
  of bisets of $f,g$ are conjugate.
\end{cor}
\begin{proof}
  As in the case of Thurston maps without multicurve (see
  Lemma~\ref{lem:combinatorial equivalence}), a combinatorial
  equivalence between Thurston maps with multicurves factors as a
  composition of an isotopy and a conjugation. The corollary therefore
  follows from Theorem~\ref{thm:kameyama-mc}.
\end{proof}

\section{Renormalization}
Let us consider a sphere map $f\colon(S^2,C,\DD)\to(S^2,A,\CC)$. In
this section we assume $\DD= f^{-1}(\CC)$ rel $C$. In this case $\DD$
naturally cuts $f$ into exactly $\#\DD+1$ sphere bisets which we call
\emph{small bisets of the decomposition}.
 
It easily follows from the assumption $\DD= f^{-1}(\CC)$ rel $C$ that
every non-essential curve $\gamma\in f^{-1}(\CC)$ is either trivial or
peripheral. Since $\gfB$ is a tree, every edge vertex $z\in
\gfB$ disconnects $\gfB$ into exactly two connected
components.

\begin{lem}\label{lem:PropOfLambda}
  Let $z\in \gfB$ be a curve vertex. Then $z$ is essential if and only
  if every connected component of the graph $\gfB\setminus \{z\}$
  contains an object $v$ with $\lambda(v)\neq\lambda(z)$.
\end{lem}
\begin{proof}
  If $z$ is a non-essential curve vertex, then $S_z$ surrounds a disc
  $D$ homotopic rel $C$ to a point. Therefore, all objects
  representing small spheres and curves within $D$ have the same image
  under $\lambda$ as $z$. Conversely, suppose $z$ is an essential
  curve vertex. Then
  $A_z\coloneqq i^{-1}(S_{\lambda(z)})=\bigcup_{v\in
    \lambda^{-1}(\lambda(z))}S_v$ is an annulus (possibly a closed
  curve) such that its boundary components are homotopic to
  $S_z$. Therefore, there are $S_v,S_w\subset S^2\setminus A_z$ in
  different components of $S^2\setminus S_z$ such that
  $\lambda(v)\neq\lambda(z)\neq\lambda(w)$.
\end{proof}

\begin{cor}\label{cor:EssSphVert}
  Suppose that the underlying graph of $\gfB$ is not a
  singleton. Let $y\in \gfY$ be a sphere vertex. Then
  $\lambda^{-1}(y)$ contains a unique $z\in \gfB$ that
  neighbours at least one essential curve vertex.
\end{cor}
\begin{proof}
  Clearly there is at least one sphere vertex in $\lambda^{-1}(y)$
  neighbouring at least one essential curve vertex. If there are two
  such sphere vertices $v,w\in \lambda^{-1}(y)$, then they are
  separated by a non-essential curve vertex; this contradicts
  Lemma~\ref{lem:PropOfEssCurves}.
\end{proof}

\subsection{Essential spheres}
We return to a sphere map $f\colon(S^2,C,\DD)\to(S^2,A,\CC)$, keeping
the convention $\DD=f^{-1}(\CC)$ rel $C$, and consider its
decomposition as tree of bisets. The curve vertices in
$\gfB(f)$ correspond to edges in the decomposition of
$(S^2,f^{-1}(A),f^{-1}(\CC))$, and to connected components of
$f^{-1}(\CC)$, while the sphere vertices $z\in\gfB(f)$
correspond to small spheres $S_z$. We classify sphere vertices in
$\gfB(f)$ as follows: a sphere vertex $z\in\gfB(f)$ is
\begin{idescription}
\item[trivial] if $S_z$ is homotopic rel $C$ to a point in $(S^2,C)$;
\item[annular] if $S_z$ is homotopic rel $C$ to a curve in $\DD$;
\item[essential] otherwise.
\end{idescription}

\noindent A trivial vertex is thus a dynamically-irrelevant small
sphere in $S^2\setminus f^{-1}(\CC)$: it gets blown down to a point
(marked or not) by projecting to $(S^2,C)$. We have the following
characterization in terms of graph combinatorics:
\begin{lem}\label{lem:unique algebraic essential}
  Suppose that the underlying tree of $\gfB$ is not a
  singleton. A sphere vertex $z\in \gfB$ is trivial if there is
  a non-essential curve vertex in $\gfB$ separating $z$ from at
  least one essential curve vertex; it is annular if it separates two
  essential curve vertices of $\gfB$ mapping to the same curve
  vertex of $\gfY$; and it is a sphere vertex if $\lambda(z)$
  is a sphere vertex and $z$ is adjacent to an essential curve vertex.
\end{lem}
\begin{proof}
  A sphere vertex $z\in \gfB$ is trivial if and only if $S_z$
  lies within a disc that is contractible rel $C$ and is surrounded by
  a non-essential curve in $f^ {-1}(\CC)$; this is equivalent to the
  condition stated in the lemma.

  If a sphere vertex $z\in \gfB$ is annular, then there are two
  components in $\partial S_z$ homotopic rel $C$ to a curve in $\DD$;
  namely $S_z$ separates two homotopic rel $C$ essential curves in
  $f^ {-1}(\CC)$; this is equivalent to the condition stated in the
  lemma. Conversely, if $S_z$ separates two essential curves in
  $f^ {-1}(\CC)$ that are homotopic rel $C$, then $S_z$ is itself
  homotopic rel $C$ to these curves.

  By definition, if $z\in \gfB$ is an essential sphere vertex,
  then $\lambda(z)$ is a sphere vertex. Conversely, if
  $z\in \gfB$ is a sphere vertex while $\lambda(z)$ is a curve
  vertex, then $z$ is either a trivial or an annular sphere vertex. By
  Corollary~\ref{cor:EssSphVert}, for a vertex $y\in \gfY$
  there is a unique vertex $z'\in \lambda^{-1}(y)$ that neighbours an
  essential curve vertex of $\gfB$. On one hand, all sphere
  vertices in $\lambda^{-1}(y)\setminus \{z'\}$ are trivial (by the
  first claim of this lemma), on the other hand $z'$ is neither
  trivial not annular (by the first two claims of this lemma);
  i.e. $z'$ is an essential sphere vertex.
\end{proof}

\begin{cor}\label{cor:unique essential}
  There is a computable bijection between the small spheres in
  $S^2\setminus\DD$ and the essential spheres in
  $S^2\setminus f^{-1}(\CC)$.
\end{cor}
\begin{proof}
  This follows from Lemma~\ref{lem:unique algebraic essential} and
  Corollary~\ref{cor:EssSphVert}. All conditions are expressed in
  terms of the map of finite graphs $\lambda\colon\gfB\to\gfY$, so are
  computable by Theorem~\ref{thm:DecompOfBiset}.
\end{proof}
\noindent Let us remark that the classification of small spheres in
$(S^2,f^{-1}(A), f^{-1}(\CC))$ depends only on the forgetful map
$(S^2,f^{-1}(A), f^{-1}(\CC))\overset{\one}\rightarrow (S^2,C,\DD)$;
see Figure~\ref{fig:taesmallspheres}.

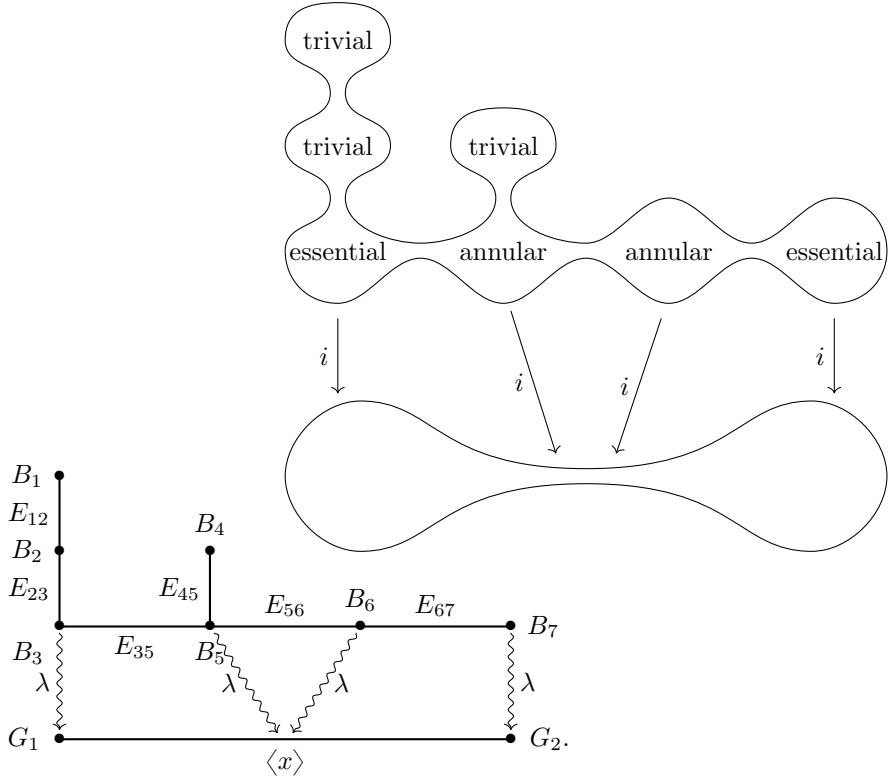
\begin{figure}
  \begin{center}
    \begin{tikzpicture}
      \def\upup{.. controls +(90:0.4) and +(270:0.4) ..}
      \def\upright{.. controls +(90:0.4) and +(180:0.4) ..}
      \def\rightdown{.. controls +(0:0.4) and +(90:0.4) ..}
      \def\downdown{.. controls +(270:0.4) and +(90:0.4) ..}
      \def\downright{.. controls +(270:0.4) and +(180:0.4) ..}
      \def\rightup{.. controls +(0:0.4) and +(270:0.4) ..}
      \def\rightright{.. controls +(0:0.4) and +(180:0.4) ..}
      \def\downleft{.. controls +(270:0.4) and +(0:0.4) ..}
      \def\leftleft{.. controls +(180:0.4) and +(0:0.4) ..}
      \def\leftup{.. controls +(180:0.4) and +(270:0.4) ..}
      \draw (-4,3) \upup (-3.4,3.7) \upup (-4,4.4) \upup (-3.4,5.1) \upup
      (-4,5.8) \upright (-3.3,6.3) \rightdown (-2.6,5.8) \downdown
      (-3.2,5.1) \downdown (-2.6,4.4) \downdown (-3.2,3.7)
      \downright (-2.2,3.1) \rightup (-1.2,3.7) \upup (-1.8,4.4) \upright
      (-1.1,4.9) \rightdown (-0.4,4.4) \downdown (-1.0,3.7) \downright
      (0.0,3.1) \rightright (1.1,3.7) \rightright (2.2,3.1) \rightright
      (3.3,3.7) \rightdown (4.0,3.0) \downleft (3.3,2.3) \leftleft
      (2.2,2.9) \leftleft (1.1,2.3) \leftleft (0.0,2.9) \leftleft (-1.1,2.3)
      \leftleft (-2.2,2.9) \leftleft (-3.3,2.3) \leftup (-4,3); 

      \draw (0,0.1) .. controls +(180:2) and +(0:1) .. (-3,1)
      .. controls +(180:0.5) and +(90:0.5) .. (-4,0)
      .. controls +(270:0.5) and +(180:0.5) .. (-3,-1)
      .. controls +(0:1) and +(180:2) .. (0,-0.1);
      \draw (0,0.1) .. controls +(0:2) and +(180:1) .. (3,1)
      .. controls +(0:0.5) and +(90:0.5) .. (4,0)
      .. controls +(270:0.5) and +(0:0.5) .. (3,-1)
      .. controls +(180:1) and +(0:2) .. (0,-0.1);

      \foreach\c in {(-3.3,5.8),(-3.3,4.4),(-1.1,4.4)} {\node at \c {trivial};}
      \foreach\c in {(-1.1,3),(1.1,3)} {\node at \c {annular};}
      \foreach\c in {(-3.3,3),(3.3,3)} {\node at \c {essential};}
      \draw[->] (-3.3,2.1) -- node[left] {$i$} +(0,-1.0);
      \draw[->] (-1.0,2.2) -- node[left] {$i$} (-0.4,0.3);
      \draw[->] (1.0,2.1) -- node[left] {$i$} (0.4,0.3);
      \draw[->] (3.3,2.1) -- node[left] {$i$} +(0,-1.0);
      \begin{scope}[xshift=-4cm,yshift=-4cm]
    \node[inner sep=0pt] (B1) at (-3,4) [label={left:$B_1$}] {$\bullet$};
    \node[inner sep=0pt] (B2) at (-3,3) [label={left:$B_2$}] {$\bullet$};
    \node[inner sep=0pt] (B3) at (-3,2) [label={below left=-1mm:$B_3$}] {$\bullet$};
    \node[inner sep=0pt] (B4) at (-1,3) [label={above:$B_4$}] {$\bullet$};
    \node[inner sep=0pt] (B5) at (-1,2) [label={below:$B_5$}] {$\bullet$};
    \node[inner sep=0pt] (B6) at (1,2) [label={above:$B_6$}] {$\bullet$};
    \node[inner sep=0pt] (B7) at (3,2) [label={right:$B_7$}] {$\bullet$};
    \draw[thick] (B1.center) -- node[left] {$E_{12}$} (B2.center) --
    node[left] {$E_{23}$} (B3.center) -- node[below] {$E_{35}$} (B5.center) --
    node[left] {$E_{45}$} (B4.center) (B5.center) -- node[above] {$E_{56}$}
    (B6.center) -- node[above] {$E_{67}$} (B7.center);

    \node[inner sep=1pt] (G1) at (-3,0.5) [label={left:$G_1$}] {$\bullet$};
    \node[inner sep=1pt] (G2) at (3,0.5) [label={right:$G_2$.}] {$\bullet$};
    \draw[thick] (G1.center) -- node[below] {$\langle x\rangle$} (G2.center);

    \draw (B3) edge[lambda] node[left] {$\lambda$} (G1);
    \draw (B5) edge[lambda] node[left] {$\lambda$} (-0.1,0.58);
    \draw (B6) edge[lambda] node[right] {$\lambda$} (0.1,0.58);
    \draw (B7) edge[lambda] node[right] {$\lambda$} (G2);
      \end{scope}
    \end{tikzpicture}
  \end{center}
  \caption{Trivial, annular and essential small spheres, and the corresponding tree of bisets (note that the actual graph of bisets is a barycentric subdivision of the one shown on the picture)  }\label{fig:taesmallspheres}
\end{figure}

\subsection{Epic-monic factorizations}\label{ss:epic-monic}
Recall
from~\cite{bartholdi-dudko:bc1}*{Lemma~\ref{bc1:lem:biset=fibre}} that
for a transitive $H$-$G$-biset $B$ there exist a group $K$ and
homomorphisms $\phi\colon K\to G$, $\psi\colon K\to H$ such that
$B=B_\psi^\vee\otimes_K B_\phi$. Moreover, there is a minimal such
$K$.

Every group homomorphism factors as a composition of a surjective
(``epic'') morphism with an injective (``monic'') morphism. Let us
apply this factorization to $\phi$ and $\psi$; that is, set
$P\coloneqq \psi(K)$, $Q\coloneqq \phi(K)$, and view $\psi$ and $\phi$
as homomorphisms onto $P$ and $Q$ respectively, followed by natural
inclusions. We get the factorization
\[B= \subscript{H}H_{P} \otimes B_\psi^\vee\otimes_K B_\phi \otimes
  \subscript{Q}G_G.
\]
Combining the last three terms into a biset $\subscript{P}B'_G$ we get the
\emph{left epic-monic} factorization
\[B=\subscript H H_{P} \otimes_{P}B'_{G}.
\]
Below is a direct way how to calculate $P$ and $B'$; its proof is
immediate.

\begin{lem}\label{Lem:leftInjDecomp}
  Let $B$ be a transitive $H$-$G$-biset. Choose $b\in B$ and let $P$
  be the stabilizer of $b\otimes \{\cdot\}$ in
  $B\otimes_{G} \{\cdot\} $. Then $B'=P b G$ and
  $B=\subscript{H}H_{P} \otimes P b G$.\qed
\end{lem}

\noindent We remark that the left epic-monic decomposition of a sphere
biset $\subscript{H}B_G$ is trivial (i.e.~$P=H$) because the
$G$-action is transitive. Let us now generalize
Lemma~\ref{Lem:leftInjDecomp} to the context of graphs of
bisets. Recall from~\cite{bartholdi-dudko:bc1} that for bisets
$\subscript H B_G,\subscript{H'}{B'}_{G'}$ an \emph{intertwiner} is a
triple of maps
$(\gamma\colon H\to H',\beta\colon B\to B',\alpha\colon G\to G')$
satisfying $\beta(h b g)=\gamma(h)\beta(b)\alpha(g)$ for all
$h\in H,b\in B,g\in G$.
\begin{prop}\label{Lem:GrBisleftInjDecomp}
  Let $\subscript\gfY\gfB_\gf$ be a graph of bisets and
  assume that every $B_z$ is a transitive
  $G_{\lambda(z)}$-$G_{\rho(z)}$-biset.  Let $\gf^\circ$ be the graph
  of groups obtained from $\gf$ by replacing each group $G_z$ with the
  trivial group. Denote by $(\gf,\gf^\circ)$ the
  $\gf$-$\gf^\circ$-graph of bisets with $\lambda =\rho =\one$ and
  bisets $B_z=\subscript{G_{z}}\{\cdot\}_{\one}$ for all
  $z\in (\gf,\gf^\circ)$.

  Consider the graph of bisets $\gfB \otimes
  (\gf,\gf^\circ)$. For every $z\in\gfB$, choose $b_z\in B_z$.
  Let $P_{\lambda(z)}\le G_{\lambda(z)}$ be the stabilizer of $b_z
  \otimes \{\cdot\}$ in $B_{(z,\rho(z))}$, with $(z,\rho(z))$ viewed
  as an object in $\gfB\otimes (\gf,\gf^\circ)$. As in
  Lemma~\ref{Lem:leftInjDecomp}, decompose $B_z=
  G_{\lambda(z)}\otimes_{P_{\lambda(z)}} B'_z$.
 
  Let $\mathfrak P$ be the graph of groups with underlying graph
  $\gfB$ and group $P_z$ at every $z\in \mathfrak P$, and with
  morphisms $P_{z}\to P_{z^-}$ given by $g\to (g^-)^{\gamma_z}$ with
  $\gamma_z$ chosen so that
  $\gamma^{-1}_z (b_{(z,\rho(z))})^{-}=b_{(z^-,\rho(z^-))}$.

  Let $\gfB'_z$ be the $\mathfrak P$-$\gf$-graph of bisets with biset
  $\subscript{P_{z}}(B'_z)_{G_{\rho(z)}}$ attached to each $z\in \gfB'$
  and biset intertwiners $B'_z\to B'_{z^-}$ given by $b\to
  \gamma_z^{-1}b^{-}$. Let $(\gfY,\mathfrak P)$ be the
  $\gfY$-$\mathfrak P$-graph of bisets with underlying graph $\gfB$
  and with biset $B_z=\subscript{G_{\lambda(z)}} {G_{\lambda(z)}}
  _{P_z}$ attached to $z\in (\gfY,\mathfrak P)$ with biset intertwiners
  $B_z\to B_{z^-}$ given by $g\to g\gamma_z$.
   
  Then $\gfB$ decomposes as
  \begin{equation}\label{eq:Lem:GrBisleftInjDecomp}
    \gfB= (\gfY,\mathfrak P) \otimes \gfB'.
  \end{equation}
\end{prop}
\begin{proof}
  Follows directly from the construction and the definition of the
  tensor product: $(\gfY,\mathfrak P) \otimes \gfB'$ is
  $\gfB$ as a graph with
  $\subscript{G_{\lambda(z)}} {G_{\lambda(z)}}_{{P_z}}\otimes B'_z= B_z$
  attached to $z\in (\gfY,\mathfrak P)$.
\end{proof}

\subsection{Small sphere maps}
The classification of small spheres can be directly seen at the
algebraic level. Consider a sphere biset $\subscript H B_G$,
multicurves $\CC,\DD$ in $G,H$ respectively such that $\DD$ is
contained in the $B$-lift of $\CC$.  Let $\gf,\gfY$ denote the tree of
groups decompositions of $G,H$ along $\CC,\DD$ respectively, and
consider the corresponding tree of bisets decomposition
$\subscript\gfY\gfB_\gf$ of $B$.

We express algebraically the notions of trivial, annular and essential
sphere vertices using the epic-monic factorization
from~\S\ref{ss:epic-monic}:
\begin{defn}[Algebraically trivial, annular, essential sphere vertices]
\label{defn:AlgSphClass}
  Let $\subscript\gfY\gfB_\gf$ be a sphere tree of bisets. A
  sphere vertex of $z\in\gfB_\gf$ is
  \begin{idescription}
  \item[trivial] if $B_z$ is of the form
    $G_{\lambda(z)}\otimes_{P_z} B'_z$ for a sphere
    $P_z$-$G_{\rho(z)}$-biset $B'_z$ and a subgroup
    $P_z\le G_{\lambda(z)}$ generated by a
    representative of a peripheral in $\pi_1(\gfY)$ or trivial
    conjugacy class, so that the sphere structure of $P_z$ is viewed
    accordingly as either that of $\pi_1(S^2)$ or as that of
    $\pi_1(S^2\setminus \{0,\infty\})$;
  \item[annular] if $B_z$ is of the form
    $G_{\lambda(z)}\otimes_{P_z} B'_z$ for a sphere
    $P_z$-$G_{\rho(z)}$-biset $B'_z$ and a subgroup
    $P_z\le G_{\lambda(z)}$ generated by a
    representative of a class in $\DD$, so that the sphere structure
    of $P_z$ is viewed as that of $\pi_1(S^2\setminus \{0,\infty\})$;
  \item[essential] if $B_z$ is a sphere biset.
  \end{idescription}

  Similarly, a curve vertex $e\in \gfB$ is
  \begin{idescription}
  \item[non-essential] if $B_e$ is of the form
    $G_{\lambda(e)}\otimes_{P_e} B'_e$ for a sphere
    $P_e$-$G_{\rho(e)}$-biset $B'_e$ and a subgroup $P_e\le
    G_{\lambda(e)}$ generated by a
    representative of a conjugacy class that is peripheral or trivial
    in $\pi_1(\gfY)$, so that the sphere structure of $P_e$ is
    viewed accordingly as that of $\pi_1(S^2)$ or as that of
    $\pi_1(S^2\setminus \{0,\infty\})$;
  \item[essential] if $B_e$ is a sphere biset.\qedhere
  \end{idescription}
\end{defn}

\noindent Observe that all these cases are exclusive; for example, in
the first, second and fourth cases the biset $B_z$ is not a sphere
biset.

\begin{lem}\label{lem:AlgClassOfBis}
  Let $f\colon(S^2,C,\DD)\to(S^2,A,\CC)$ be a sphere map with
  associated tree of bisets $\gfB(f)$. Then a small sphere
  vertex of $\gfB(f)$ is trivial, annular or essential if and
  only if the corresponding small sphere is respectively trivial,
  annular or essential, and a curve vertex of $\gfB(f)$ is
  essential or non-essential if and only if the corresponding curve is
  respectively essential or non-essential.
\end{lem}
\begin{proof}
  If $z\in \gfB$ is a trivial sphere vertex, then
  $P_z\coloneqq i_* \pi_1(S_z\setminus f^{-1}(A))$ is either trivial
  or generated by a representative of a peripheral conjugacy
  class. Decompose $i\colon S_z\to S_{\lambda(z)}$ as
  $S_z \overset{i}{\rightarrow} i(S_z)\hookrightarrow S_{\lambda(z)}$;
  so $B(i\colon S_z\to S_{\lambda(z)})^\vee$ factors as
  $G_{\lambda(z)}\otimes_{P_z} B''_z$ for a sphere biset
  $B''_z$. By~\eqref{eq:dfn:B_z}, $B_z$ is of the form
  $G_{\lambda(z)}\otimes_{P_z} B'_z$ for a sphere biset $B'_z$. If
  $z\in \gfB$ is an annular sphere vertex, then
  $P_z\coloneqq i_*\pi_1(S_z\setminus f^{-1}(A))$ is generated by a
  representative of a conjugacy class in $\DD$. As above, $B_z$ is of
  the form $G_{\lambda(z)}\otimes_P B'_z$.

  Suppose that $z\in \gfB$ is an essential sphere
  vertex. Observe
  $i_*\pi_1(S_z\setminus f^{-1}(A))= \pi_1(S_{\lambda(z)}\setminus C)$
  because $i^{-1}(S_{\lambda(z)})\setminus S_z$ consists of finitely
  many contractible rel $C$ discs. Since $G_{\lambda(z)}$ has at least
  three peripheral conjugacy classes, $B_z$ does not factor as
  $G_{\lambda(z)}\otimes_P B'$ for a subgroup $P\lneq G_{\lambda(z)}$
  and it follows that $B_z$ is a sphere biset.
  
  The characterization of essential curve vertices is verified
  similarly.
\end{proof}

Let $\mathscr S,\mathscr T,\mathscr U$ denote the collection of small
spheres in the decomposition of
$S^2\setminus\CC,S^2\setminus f^{-1}(\CC),S^2\setminus\DD$
respectively. By Corollary~\ref{cor:unique essential}, there is a
well-defined map, still written $f\colon\mathscr U\to\mathscr S$,
sending each small sphere $U_w\subset S^2\setminus\DD$ to the image by
$f$ of the unique essential sphere of $S^2\setminus f^{-1}(\CC)$
contained in $U_w$. To avoid cumbersome indices, we write
interchangeably $f(U_w)=S_v$ and $f(w)=v$, defining by the latter a
map from the vertex set of $\gfY$ to the vertex set of $\gf$.

\begin{lemdef}\label{lem:SmallMaps}
  The induced map $\widehat f\colon \widehat{U_w}\to \widehat{S_{f(w)}}$ is a
  sphere map, called a \emph{small sphere map} of $f$.
\end{lemdef}
\begin{proof}
  Let $T_z$ be the essential sphere above $U_w$. By construction,
  $i\colon T_z\to U_w$ is a homeomorphism onto its image, and
  components of $U_w\setminus i(T_z)$ are contractible.

  The space $\widehat{U_w}$ can equivalently be constructed by
  attaching a copy of the unit disk $\mathbb D$, along its boundary,
  to every $S^1$-boundary component. The disk is marked at its
  centre. In this manner, $U_w$ is seen as a subset of
  $\widehat{U_w}$.  We perform the same construction on $T_z$ and on
  $S_{f(w)}$.

  The map $f\colon T_z\to S_{f(w)}$ extends naturally to a branched
  covering $g\colon\widehat{T_z}\to\widehat{S_{f(w)}}$ with branched
  values in $A\cup\{\text{centres of disks}\}$: on $T_z$ it is defined
  as $f$, while on a disk $\mathbb D$ with boundary curve $C$ mapping
  by degree $d$ to $f(C)$ it is defined as $z^d\colon\mathbb
  D\to\mathbb D$.

  The map $i\colon T_z\to U_w$ extends naturally to a homeomorphism
  $j\colon\widehat{T_z}\to\widehat{U_w}$: on $T_z$ it is defined as
  $i$, while on disks $\mathbb D$ it is defined as the identity. We
  set $\widehat f=j^{-1}g$ and note that it is a branched covering.
\end{proof}

These definitions admit direct analogues in the algebraic setting of a
sphere tree of bisets $\subscript\gfY\gfB_\gf$. By
Lemma~\ref{lem:unique algebraic essential}, for every vertex
$w\in\gfY$ there is a unique essential vertex $z\in\gfB$
in $\lambda^{-1}(w)$, and $\rho(z)=v$ for some sphere vertex
$v\in\gf$. We define a map $B_*$ from the vertex set of $\gfY$
to the vertex set of $\gf$ by
\[B_*(w)\coloneqq v\text{ if $\lambda(z)=w$ and $\rho(z)=v$ and $z$ essential.}\]
We may also extend $B_*$ into a map from the geometric realization of
$\gfY$ to that of $\gf$. Recall that, for a graph $\gf$, its
geometric realization $\gf_0$ is the topological space
\[\gf_0=\gf\times[0,1]\,/\,\{(x,t)=(\bar x,1-t),(x^-,t)=(x,0)\,\forall x\in\gf,\forall t\in[0,1]\}.\]
The map
\[\lambda^{-1}\colon \{\text{sphere vertices of }\gfY\}\to
\{\text{essential sphere vertices of }\gfB\}
\]
extends into an essentially unique (up to isotopy rel sphere vertices)
tree embedding $\lambda^{-1}\colon \gfY_0\hookrightarrow
\gfB_0$. Indeed, any geodesic $\ell$ in $\gfY_0$ lifts under $\lambda$
into a unique geodesic traveling through essential curve vertices, see
Lemma~\ref{lem:PropOfEssCurves}. Composing with $\rho$, we obtain the
map $B_*\colon \gfY_0\to\gf_0$.

\begin{exple}\label{exple:MattOFConjPolyn}
  The map $B_*$ is closely related to the self-map of the ``Hubbard
  tree'' associated with a complex polynomial; we refer
  to~\cite{bartholdi-dudko:bc1}*{\S\ref{bc1:ss:hubbardtrees}} and the
  original references~\cites{douady-h:edpc1,douady-h:edpc2}.

  Consider a complex polynomial $p$ and let
  $g\coloneqq p\FM \overline p$ be the formal mating of $p$ with its
  complex conjugate. Let us denote by $p\colon \HT\selfmap$ the
  Hubbard tree of $p$; the Hubbard tree
  $\overline p\colon \overline \HT\selfmap $ of $\overline p$ is
  obtained from $p\colon \HT\selfmap$ by applying complex conjugation.

  For every edge $e\in \HT$ there is a simple closed curve $\gamma_e$
  intersecting $\HT$ once at $e$ and intersecting $\overline \HT$ once
  at $\overline e$ such that $\gamma_e$ follows external rays of $p$
  and $\overline p$ away from $e$ and $\overline e$. If the edge
  $e_i\in \HT$ covers $d_{i,j}$ times $e_j\in \HT$, then
  $g^{-1}(\gamma_{e_j})$ has exactly $d_{i,j}$ components isotopic to
  $\gamma_{e_i}$ and each of these components is a degree-one preimage
  of $\gamma_{e_j}$. Therefore,
  $\CC_\HT\coloneqq \{\gamma_e\colon e\in \HT\}$ is an invariant
  multicurve with Thurston matrix $T_{g,\CC_\HT}=(d_{i,j})$, an
  integer matrix; all the small maps in this decomposition are
  finite-order or monomial rational maps, so $\CC_\HT$ is the
  canonical obstruction of $g$.

  Let $\subscript\gfY\gfB_\gfY$ be the graph of
  bisets encoding the decomposition of $g$ relative to $\CC_\HT$, and
  let $\subscript\gf{\mathfrak T}_\gf$ be the graph of bisets associated
  with $p\colon \HT\selfmap$ as in
  \cite{bartholdi-dudko:bc1}*{\S\ref{bc1:ss:hubbardtrees}}. Sphere
  vertices of $\gfY$ are in bijection with vertices of $\HT$
  and with non-edge vertices of $\gf$. Curve vertices of $\gfY$
  are in bijection with edges of $\HT$. Essential sphere vertices of
  $\gfB$ are in bijection with essential vertices of $\HT^1$.

  We conclude that $B_*$, as a self-map of sphere vertices of $\gfY$,
  coincides with $p$, as a self-map of vertices of $\mathcal T$.

  At the graph level $\lambda,\rho \colon \gfB\to \gfY$
  are the same as $\lambda,\rho \colon \mathfrak T\to
  \gf$. Furthermore, there is a semiconjugacy
  $\pi\colon \subscript\gfY\gfB_\gfY\to
  \subscript\gf{\mathfrak T}_\gf$ given by identifying the northern and
  southern hemisphere. More precisely, consider a sphere vertex
  $y\in\gfY$. Its associated group $G_y$ maps to the
  corresponding group $G_{\pi(y)}=\Z/\ord(\pi(y))$ in $\gf$ by sending
  the peripheral generator in the northern and southern hemisphere to
  $1$ and $-1$ respectively and all other peripheral generators to
  $0$.

\end{exple}

\subsection{Refinement of sphere trees of groups}
Recall from~\cite{bartholdi-dudko:bc1}*{\S\ref{bc1:ss:vk}} that the
fundamental group of a graph of groups does not change under the
operations ``split an edge'' and ``add an edge''. Recall also that,
in a sphere tree of groups, a peripheral conjugacy class in a vertex
group is vacant if it intersects no image of an edge group.

If $\gf$ is a stable sphere tree of groups, then we adjust these
operations to respect the sphere structure as follows:
\begin{description}
\item[(1) split a curve vertex] choose a curve vertex $v\in \gf$. Let
  $e_0,\bar e_0$ and $e_1,\bar e_1$ be the pair of edges adjacent to
  $v$. Declare $v$ to be a sphere vertex add split $e_0,\bar e_0$ and
  $e_1,\bar e_1$ as follows. Add a curve vertex $\ell_0$ to $\gf$ and
  replace $e_0,\bar e_0$ by
  $e_{00},\bar e_{00},e_{01},\bar e_{01}$ with
  $e_{00}^-=e_0^-$, $e_{00}^+=\ell_0=e^-_{01}$, $e_{01}^+=e_0^+$; and make a similar operation for the pair $e_1,\bar e_1$: replace it by edges $e_{10},\bar e_{10},e_{11},\bar e_{11}$ and a new curve vertex $\ell_1$. Declare all groups to be
  equal to $G_v$ with the obvious maps between them;
\item[(2) add an edge] choose a vertex $v\in V$, and either trivial or
  a vacant peripheral class $\Gamma$ of $v$. Let $H$ be a (trivial or
  cyclic) subgroup of $G_v$ generated by a representative in
  $h\in \Gamma$. Add a new sphere vertex $w$ to $\gf$, add a new curve
  vertex $\ell$ to $\gf$, and add new edges $e_0,\overline e_0$ and
  $e_1,\overline e_1$ with $e_0^-=v$, $e_0^+=\ell$ and $e_0^-=\ell$,
  $e_1^+=w$. Define the new groups by
  $G_{e_0}=G_{\ell}=G_{e_1}=G_w=H$, with the obvious maps between
  them. The peripheral conjugacy classes of $H$ form either empty set in
  case $H$ is trivial or $\{h,h^{-1}\}$ in case $H$ is cyclic.
\end{description} 
A \emph{refinement} of a graph of groups $\gf$ is a graph of groups
obtained from $\gf$ by applying finitely many times the above
operations (1) and (2). An \emph{unstable sphere tree of groups} is a
refinement of a stable sphere tree of groups.

Let $\mathfrak P$ be a refinement of $\gf$. Define a graph map
$\lambda \colon \mathfrak P \to \gf$ by sending vertex
$v\in \mathfrak P$ to the vertex of $\gf$ from which $v$ was
refined. Define the tree of bisets $(\gf,\mathfrak P)$ to be
$ \mathfrak P$ as a graph with $\lambda$ as above,
$\rho\coloneqq \one$, and with
$B_z= \subscript{G_{\lambda(z)}} {G_{\lambda(z)}}_{G_{\rho(z)}}$ with natural
maps between them. Clearly, the fundamental biset of
$(\gf,\mathfrak P)$ is isomorphic to
$\subscript{\pi_1(\gf)}{\pi_1(\gf)}_{\pi_1(\gf)}$.

\begin{lem}\label{lem:StabOfVacClass}
  Let $\mathfrak P$ be a sphere refinement of a sphere tree of groups
  $\gf$. Then for any vacant peripheral class $\Gamma$ of $\gf$ there
  is a unique vacant peripheral class of $\Gamma'$ of $\mathfrak P$
  identified with $\Gamma$ via $(\gf,\mathfrak P)$.
\end{lem}
\begin{proof}
  The tree of groups $\mathfrak P$ is obtained from $\gf$ via
  operations ``split a curve vertex'' and ``add an edge''; these
  operations clearly respect the set of vacant peripheral classes in a
  sphere tree of groups.
\end{proof}

\subsection{Sphere trees of bisets}
We give in this section a characterization of those trees of bisets
that come from sphere maps with invariant multicurve. Combined with
Theorem~\ref{thm:kameyama-mc}, this extends
Theorem~\ref{thm:dehn-nielsen-baer+} to branched coverings with
multicurves.

\begin{defn}\label{defn:SphTreeOfBis}
  A tree of bisets $\subscript\gfY\gfB_\gf$ is a
  \emph{sphere tree of bisets} if
  \begin{enumerate}\renewcommand\theenumi{ST\ensuremath{{}_\arabic{enumi}}}
  \item $\gfY$ and $\gf$ are stable sphere trees of
    groups; \label{cond:1:defn:SphTreeOfBis}
  \item $\gfB$ is left-free, see~
    \cite{bartholdi-dudko:bc1}*{Definition~\ref{bc1:dfn:FibrationGrBis}},
    and $B_z$ is transitive for all
    $z\in\gfB$; \label{cond:2:defn:SphTreeOfBis}
  \end{enumerate}
  write $\gfB= (\gfY,\mathfrak P) \otimes \gfB'$
  as in~\eqref{eq:Lem:GrBisleftInjDecomp};
  \begin{enumerate}\renewcommand\theenumi{ST\ensuremath{{}_3}}
  \item $\mathfrak P$ has a sphere structure and is a sphere
    refinement of $\gfY$, every biset in $\gfB'$ is a
    sphere biset.\label{cond:3:defn:SphTreeOfBis}\qedhere
  \end{enumerate}
\end{defn}

\begin{lem}
  The tree of bisets $\gfB(f)$ associated with a map
  $f\colon (S^2,C,\DD)\to (S^2,A,\CC)$ as in
  Definition~\ref{defn:gfofbisets} with $\DD=f^{-1}(\CC)$ rel $C$ is a
  sphere tree of bisets.
\end{lem}
\begin{proof}
  Condition~\eqref{cond:1:defn:SphTreeOfBis} is
  immediate. Condition~\eqref{cond:2:defn:SphTreeOfBis} follows from
  \cite{bartholdi-dudko:bc1}*{Proposition~\ref{bc1:prop:DecompFibrMaps}}
  and
  Lemma~\ref{lem:SphBisOfSphMap}. Condition~\eqref{cond:3:defn:SphTreeOfBis}
  follows from Lemma~\ref{lem:AlgClassOfBis}: for every $z\in
  \gfB$ in the decomposition $B_z=G_{\lambda(z)}\otimes_{P_z}
  B'_z$ as in Lemma~\ref{Lem:leftInjDecomp} the biset $B'_z$ is sphere
  with the group $P_z$ either trivial or generated by a peripheral
  conjugacy class in $B_z$.
\end{proof}

\begin{thm}\label{thm:sphere tree gives sphere biset}
  The fundamental biset of a sphere tree of bisets
  $\subscript\gfY\gfB_\gf$ is a sphere biset.
\end{thm}
\begin{proof}
  Let $\gfB$ be a sphere tree of bisets. We check that
  $\pi_1(\gfB)$ satisfies Definition~\ref{dfn:SphBis}.

  By~\cite{bartholdi-dudko:bc1}*{Corollary~\ref{bc1:cor:HatBisLeftFree}},
  the biset $\pi_1(\gfB)$ is left free; and since each biset
  $B_z$ is right transitive so is $\pi_1(\gfB)$. This
  is~\eqref{cond:1:dfn:SphBis}.

  Next we check~\eqref{cond:3:dfn:SphBis}. Write
  $\gfB= (\gfY,\mathfrak P) \otimes \gfB'$ as
  in~\eqref{eq:Lem:GrBisleftInjDecomp}. By
  Lemma~\ref{lem:StabOfVacClass} the set of vacant conjugacy classes
  in $\gf$ and in $\mathfrak P$ are naturally identified. Let $\Gamma$
  be a vacant conjugacy class in $\mathfrak P$, say in $P_z$. Since
  $B'_z$ is a sphere biset, $\Gamma$ appears exactly once in the
  multiset of all lifts of all vacant peripheral conjugacy classes of
  $G_{\rho(z)}$.

  We check~\eqref{cond:3:dfn:SphBis} by a counting argument. A cycle
  $c$ as in~\eqref{eq:riemannhurwitz} is \emph{critical} of
  \emph{multiplicity} $\text{length}(c)-1$ if
  $\text{length}(c)>1$. Let $d$ be the degree (i.e.\ the number of
  left orbits) of $\pi_1(\gfB)$. We show that the number,
  counting multiplicity, of all critical cycles associated with vacant
  peripheral classes in $\gfB$ is $2d-2$. Let $n$ be the number
  of sphere vertices in $\gf$ and $m$ be the number of sphere vertices
  in $\gfB$. Then the number of curve vertices in $\gf$ and in
  $\gfB$ is $n-1$ and $m-1$ respectively. We also note
  $n\le m< d n$.

  Consider a sphere vertex $x\in \gf$. Then the number of critical
  cycles in $\bigsqcup_{z\in \rho^{-1}(x)}B_z$ that are associated with
  all peripheral classes of $G_x$ is equal to
  $2 (d- \#(\rho^{-1}(x)))$. By taking the sum over all sphere
  vertices $x\in \gf$ we see that the number of all critical cycles in
  $\gfB$ that are associated with all peripheral classes of all
  sphere groups in $\gf$ is equal to $2(d n-m)$. From this number we
  will now subtract the number of whose critical cycle that are
  associated with non-vacant peripheral classes.

  Consider a curve vertex $x\in \gf$; in two neighbouring sphere
  vertices, say $v$ and $v'$, there are exactly two non-vacant
  peripheral conjugacy classes, say $\Gamma$ and $\Gamma'$, defined by
  embedding $G_x$ into $G_v$ and into $G_{v'}$. The number of critical
  cycles in $\sqcup_{z\in \rho^{-1}(x)}B_z$ that are associated with
  $\Gamma$ and $\Gamma'$ is equal to $2 (d- \#(\rho^{-1}(x)))$. Taking
  the sum over all curve vertices $x\in \gf$, we get
  $2(d(n-1)-m+1)$. Finally,
  \[2(d n-m) -2(d(n-1)-m+1) = 2d-2,\] 
  which is~\eqref{cond:3:dfn:SphBis}.
\end{proof}

\begin{cor}\label{cor:dehn-nielsen-baer+-mc}
  Let $\subscript\gfY\gfB_\gf$ be a sphere tree of
  bisets. Write $\gfY$ as the sphere tree of groups associated
  with $(S^2,C,\DD)$ and associate similarly $\gf$ with
  $(S^2,A,\CC)$. Then there exists a sphere map
  $f\colon(S^2,C,\DD)\to(S^2,A,\CC)$, unique up to isotopy rel
  $C\cup\DD$, whose graph of bisets is isomorphic to $\gfB$.
\end{cor}
\begin{proof}
  Follows from Theorems~\ref{thm:sphere tree gives sphere biset},
  \ref{thm:dehn-nielsen-baer+} and~\ref{thm:kameyama-mc}.
\end{proof}

\subsection{The dynamical situation}\label{ss:return maps}
In the dynamical situation in which $A=C$ and $\CC=\DD$, we have a
Thurston map $f\colon(S^2,A,\CC)\selfmap$ with an invariant multicurve
$\CC$. We then also have dynamics $f\colon\mathscr S\selfmap$ on the
set $\mathscr S$ of small spheres of $S^2\setminus\CC$. We define the
set of \emph{return maps}, also called \emph{small Thurston maps},
\[R(f,A,\CC)\coloneqq\{f^e\colon S\selfmap\;\mid S\in\mathscr S,\;f^e(S)=S\text{ with $e$ minimal}\}.\]

Similarly algebraically, in the dynamical situation in which $G=H$ and
$\CC=\DD$, we have a biset $\subscript G B_G$ with an invariant multicurve
$\CC$. Let us denote by $\subscript\gf\gfB_\gf$ the tree of bisets
decomposition of $B$, and by $V$ the set of sphere vertices of $\gf$,
with essential biset $B_v$ corresponding to $v\in V$, as given by
Lemma~\ref{lem:unique algebraic essential}. We define the set of
\emph{return bisets}, also called \emph{small Thurston bisets},
\begin{multline}\label{eq:def:SmallThBis}
  R(\gfB)= R(B,\CC)\\
  \coloneqq \{B_v\otimes B_{B_*(v)}\otimes\cdots\otimes
  B_{B_*^{e-1}(v)}\selfmap\;\mid v\in V,\;B_*^e(v)=v\text{ with $e$ minimal}\},
\end{multline}
namely the bisets obtained by following a cycle in the tree of bisets
into which $B$ decomposes. By Lemma~\ref{lem:AlgClassOfBis} all bisets
in $R(B,\CC)$ are sphere bisets.

\subsection{Distinguished conjugacy classes}\label{ss:distinguished cc}
Let $(S^2,A)$ be a marked sphere, and let $\CC$ be a multicurve. It is
often convenient to treat similarly the conjugacy classes describing
elements of $A$ and of $\CC$. Consider $\gf$ a sphere tree of
groups. The set of \emph{distinguished conjugacy classes} $X$ of $\gf$
is the set of all peripheral conjugacy classes of all vertex groups
with two conjugacy classes identified if they are related by an
edge. Equivalently, if $\gf$ is the tree of groups decomposition of
$(S^2,A,\CC)$, then $X$ is in bijection with $A\cup \CC$. Note that
the set of geometric edges of $\gf$ is naturally a subset of the
distinguished conjugacy classes.

The following algorithm determines when a bijection between (possibly
peripheral) multicurves is induced by a homeomorphism between the
underlying spheres:
\begin{algo}\label{algo:MarkCl:Prom}
  \textsc{Given} $\gf$ and $\gfY$ two sphere trees of groups
  with distinguished conjugacy classes $X$ and $Y$, and given a
  bijection $h\colon X\to Y$,\\
  \textsc{Decide} whether $h\colon X\to Y$ promotes to a conjugator
  $\gfI$ from $\gf$ to $\gfY$, and if so
  \textsc{construct} $\gfI$ \textsc{as follows:}\\\upshape
  \begin{enumerate}
  \item Check whether $h$ restricts to an isomorphism between the geometric
    edge sets of $\gf$ and $\gfY$. If not return \texttt{fail}.
  \item Check whether the isomorphism between the edge sets promotes
    into a graph-isomorphism $h\colon\gf\to\gfY$. If
    not, return \texttt{fail}.
  \item For a sphere vertex $v\in \gf$ let $\Gamma_v\subset X$ be the set
    of peripheral conjugacy classes of $G_v$. Check whether
    $h(\Gamma_v)=\Gamma_{h(v)}$ for all vertices $v\in \gf$. If
    not, return \texttt{fail}.
  \item For every sphere vertex $v\in\gf$ choose an isomorphism
    $\phi(v)\colon G_v\to G_{h(v)}$ compatible with $h\colon \Gamma_v\to
    \Gamma_{h(v)}$. For every edge $e\in \gfY$ choose an isomorphism
    $\phi(e)\colon G_e\to G_{h(e)}$.
  \item Set $\gfI\coloneqq \gf$, $\lambda\coloneqq\one$,
    $\rho \coloneqq h$, $B_z\coloneqq G_{h(z)}$ for all $z\in \gfI$;
    the left action of $G_z$ on $B_z$ is via $\phi(z)$, the right
    action is natural, and the inclusion of $B_e$ into $B_{e^-}$ is
    via $1\mapsto g$, for any $g\in G_{h(z)}$ with
    $()^g\circ\phi(e)=\phi(e^-)$, if we identify $G_e$ with a subgroup
    of $G_{e^-}$.
  \item Return $\gfI$.
  \end{enumerate}
\end{algo}

Let $\subscript\gf\gfB_\gf$ and $\subscript\gfY\gfC_\gfY$ be two
sphere trees of bisets. If $\gfI$ be a conjugator between $\gf$ and
$\gfY$, then $\gfI \otimes \gfC \otimes \gfI^{\vee}$ is an $\gf$-graph
of bisets. Recall from~\S\ref{ss:mcb} the notations $\Mod(\gf)$ and
$M(\gfB)$.  The following two algorithms determine, given two trees of
bisets that stabilize the same multicurve, whether they are twists of
one another by mapping classes respecting the multicurve.

The first algorithm expresses, if possible, a sphere tree of bisets as
a left multiple of another one. It relies on the following
observation. Suppose that we want to construct a biprincipal sphere
tree of bisets $\mathfrak T$ over a sphere tree of groups $\gf$, and
that its vertex and edge bisets are already given, so that only the
intertwiners $T_e\to T_{e^-}$ need be specified at edges of
$\mathfrak T$. Consider an edge pair $\{e,\overline e\}$. The bisets
$T_e$ and $T_{e^\pm}$ may be identified with the groups $G_e$ and
$G_{e^\pm}$ respectively; then the intertwiners $T_e\to T_{e^\pm}$ are
defined by $1\mapsto g_\pm$ for some $g_\pm\in G_{e^\pm}$ which
commutes with the image of $G_e$. Since the $G_{e^\pm}$ are free while
$G_e$ is Abelian, the element $g_\pm$ may be chosen arbitrarily in
$(G_e)^\pm$. All resulting choices of maps $T_e\to T_{e^\pm}$ are
called \emph{legal intertwiners}. In fact, writing $g_\pm=(h_\pm)^\pm$
for some $h_\pm\in G_e$, the isomorphism class of $\mathfrak T$
depends only on $h_+(h_-)^{-1}$.
\begin{algo}\label{algo:chech:InMCB:pre}
  \textsc{Given} ${}_\gf\mathfrak B_\gf$ and ${}_\gf\mathfrak C_\gf$
  two sphere $\gf$-trees of bisets,\\
  \textsc{Decide} whether there is an  $\mathfrak M\in \Mod(\gf)$ such that 
  $\mathfrak C\cong \mathfrak M \otimes \mathfrak B $, and if so
  \textsc{construct} $\mathfrak M$ and the isomorphism \textsc{as follows:}\\\upshape
  \begin{enumerate} 
  \item Try to construct an isomorphism of trees
    $h\colon \mathfrak B\to\mathfrak C$ mapping essential vertices
    into essential vertices such that
    $\lambda_\gfB (z)=\lambda_\gfC \circ h(z)$ and
    $\rho_\gfB(z)=\rho_\gfC \circ h(z)$ for every $z\in\gfB$. If $h$
    does not exist, then return \texttt{fail}. Otherwise $h$ is
    uniquely defined.
  \item Choose an essential sphere vertex $v\in \gfB$. Try to find
    $M_{\lambda(v)} \in \Mod(G_{\lambda(v)})$ such that
    $M_{\lambda(v)} \otimes B_v \cong C_{h(v)}$. If such
    $M_{\lambda(v)}$ do not exist, return \texttt{fail}. Otherwise set
    $S\coloneqq \{v\}$ and run over pairs $\{e,\overline e\}\not\subset S$
    but with $e^- \in S$
    Steps~\ref{algo:chech:InMCB:St3}--\ref{algo:chech:InMCB:St6}.
  \item If $\lambda(e)\notin \lambda(S)$, then do the
    following. (Note that in this case $\lambda(e)$ is an edge in
    $\gf$.) Add $e$ and $\overline e$ to $S$, let $M_{\lambda(e)}$ be
    a principal $\Z$-biset, choose any legal intertwiner
    $()^-\colon M_{\lambda(e)}\to M_{\lambda(e)^-}$, and define
    $M_{\overline e}$ similarly.\label{algo:chech:InMCB:St3}
  \item Try to find an isomorphism between
    $M_{\lambda(e)}\otimes B_{e}$ and $C_{h(e)}$ compatible with the
    intertwiner maps. If it does not exist, return \texttt{fail}.
  \item If $e^+$ is not an essential sphere vertex, then do the
    following. If $\lambda(e^+)\notin S$, then (in this case
    $\lambda(e^+)$ is a curve vertex) choose a biprincipal $\Z$-biset
    $M_{\lambda(e^+)}$, choose a legal intertwiner from
    $M_{\lambda(e)}$ to $M_{{\lambda(e^+)}}$, and add $e^+$ to
    $S$. Try to find an isomorphism between
    $M_{\lambda(e^+)}\otimes B_{e^+}$ and $C_{h(e^+)}$ that is
    compatible with the isomorphism between
    $M_{\lambda(e)}\otimes B_{e}$ and $C_{h(e)}$ via the
    intertwiner maps. If it does not exist, return
    \texttt{fail}.
  \item If $e^+$ is an essential sphere vertex, then do the
    following. Try to find $M_{\lambda(z)} \in \Mod(G_{\lambda(z)})$
    such that $M_{\lambda(z)} \otimes B_z \cong C_{h(z)}$. If such
    $M_{\lambda(z)}$ do not exist, return \texttt{fail}. Try to find a
    legal intertwiner $()^-\colon M_{\lambda(e)}\to M_{\lambda(e^+)}$
    such that the isomorphism between $M_{\lambda(e)}\otimes B_{e}$
    and $C_{h(e)}$ is compatible with the isomorphism between
    $M_{\lambda(e)}\otimes B_{e}$ and $C_{h(e)}$ via the
    intertwiner maps. If no such
    intertwiner exists, return
    \texttt{fail}. Add $e^+$ to $S$. \label{algo:chech:InMCB:St6}
  \item Return the principal sphere tree of bisets $\mathfrak M$ and
    the isomorphism between $\gfC$ and $\mathfrak M \otimes \gfB$
    constructed via $h$.
  \end{enumerate}
\end{algo}

\begin{algo}\label{algo:chech:InMCB}
  \textsc{Given} $\subscript \gf\gfB_\gf$ and $\subscript \gf\gfC_\gf$
  two sphere $\gf$-trees of bisets,\\
  \textsc{Decide} whether $\gfC\in M(\gfB)$, and if so
  \textsc{construct} $\mathfrak M, \mathfrak N\in \Mod(\gf)$
  such that $\gfC\cong \mathfrak M \otimes \gfB \otimes
  \mathfrak N$ \textsc{as follows:}\\\upshape
  \begin{enumerate}
  \item Follow Algorithm~\ref{algo:compute M(f,C,A)} to compute a
    basis of the mapping class biset $M(\gfB)$.
  \item For each $\mathfrak N$ in the basis, do the following. Run
    Algorithm~\ref{algo:chech:InMCB:pre} on $\gfC$ and
    $\gfB\otimes\mathfrak N$. If there exists
    $\mathfrak M\in\Mod(\gf)$ with
    $\gfC\cong\mathfrak M\otimes(\mathfrak B\otimes\mathfrak N)$,
    return $(\mathfrak M,\mathfrak N)$.
  \item Return \texttt{fail}.
  \end{enumerate}
\end{algo}

\section{Extensions of bisets}
This section explains how algorithmic problems on bisets, and in
particular, conjugacy and computation of centralizers can be carried
from subbisets to an extension.

We place ourselves in an abstract setting of group actions. Since a
$G$-$H$-biset is nothing more than a $(G\times H)$-set under the rule
$(g,h)\cdot b=g\cdot b\cdot h^{-1}$, we consider left $G$-sets in this
section.

Let $G$ be a group, let $\subscript G B$ be a left $G$-set, and let
$\phi\colon H\to G$ be a group homomorphism. The following two
decision problems for $(\phi,\subscript G B)$ will be of importance in this
section:
\begin{description}
\item[The orbit problem] Given $b_1,b_2\in B$, does there exist $h\in H$
  with $h^\phi\cdot b_1=b_2$? If so, find one.
\item[The stabilizer problem] Given $b\in B$, compute
  $H_b\coloneqq\{h\in H\mid h^\phi\cdot b=b\}$.
\end{description}
Typically, $G$ and $H$ will have fixed generating sets, and $B$ will
be transitive, hence of the form $G b_0$ for some $b_0\in B$ fixed
once and for all.  We make this more precise as follows: we are given
a class $\Omega$ of groups, thought of as ``computable'' groups. They
should be finitely generated, and have solvable word problem, i.e.\
for each $G\in\Omega$ there is an algorithm that, given $g,h\in G$ as
words in a generating set for $G$, determines whether or not $g=h$. We
assume $G$ and $H$ belong to $\Omega$.  A subgroup $K$ of a computable
group $G$ is \emph{computable} if $K$ is finitely generated and has
solvable membership problem (i.e.\ there is an algorithm that, given
$g\in G$, decides whether $g\in K$). A subgroup $L$ of a computable
group $G$ is is \emph{sub-computable} if there is a computable
subgroup $K\le G$ and a computable homomorphism $K\to A$ to an Abelian
group such that $L=\ker(K\to A)$.  (It follows from the definition
that $L$ also has solvable membership problem because it is decidable
if $h\in H$ is in $\ker(H\to A)$.)

We say that the orbit problem is \emph{solvable} if there is an
algorithm that, upon receiving $b_1,b_2\in B$ as input, answers
whether there exists an $h\in H$ with $h^\phi b_1=b_2$, and if so
produces one. The elements $b_1,b_2$ will be given as $g_1b_0,g_2b_0$
for words $g_1,g_2$ in the generating set of $G$, and $h$, if it
exists, will be returned as a word in the generators of $H$.

We say that the stabilizer problem is \emph{solvable}, respectively
\emph{sub-solvable}, if there is an algorithm that, upon receiving
$b\in B$ as input, computes the stabilizer $H_b$ as a computable,
respectively sub-computable subgroup of $H$.

Note that the orbit problem could be reduced to the stabilizer problem
as soon as the class $\Omega$ is closed under taking products and
finite extensions. More precisely, consider an orbit problem $\exists?
h\in H:\,h^\phi\cdot b_1=b_2$ on a $G$-set $B$. Consider the wreath
product $G'\coloneqq G\wr2\perm$ acting naturally on $B\times B$, with
$G\times G$ acting co\"ordinatewise on $B\times B$ and $2\perm$ acting
by permuting the factors of $B\times B$. Define $\phi'\colon
H'\coloneqq H\wr2\perm\to G'$ naturally. If the stabilizer
$H'_{(b_1,b_2)}$ of $(b_1,b_2)$ is computable in $H'$, then original
orbit problem is solvable: $b_1$ and $b_2$ are in the same orbit if
and only if some generator of $H'_{(b_1,b_2)}$ permutes both copies of
$B$; and a witness $h$ to the orbit problem is easily obtained from
this generator. We elected not to make use of this reduction, treating
the orbit problem as an illustrative step.

\subsection{Extensions of \boldmath $G$-sets}
Let $\subscript G B$ be a left $G$-set and let $N$ be a normal subgroup of
$G$. We have a short exact sequence of groups
\begin{equation}
\label{eq:ExtOfGr} 1\longrightarrow N\longrightarrow G \overset\pi\longrightarrow G/N\longrightarrow 1.
\end{equation}

The $N$-set $\subscript NB$ is defined to be $B$ as a set, with the
restricted action of $N$.

Let us denote by $N\backslash(\subscript G B)$ the set of connected
components of $\subscript NB$; namely, its space of $N$-orbits:
\[N\backslash(\subscript G B)\coloneqq B/\{b_1\sim b_2 \text{ if and only if $b_1=n b_2$ for some
}n\in N\}.
\]
We denote by
$\pi\colon\subscript G B \twoheadrightarrow N\backslash(\subscript G
B)$ the induced quotient map. The action of $G/N$ on
$N\backslash(\subscript G B)$ is given by
\[g^\pi\cdot b^\pi=(g b)^\pi.
\]
It is straightforward to verify that $N\backslash(\subscript G B)$ is
a $G/N$-set.  We will call the sequence
\begin{equation}\label{eq:ExtOfG_Sets}
  \subscript NB\hookrightarrow\subscript G B\overset\pi\twoheadrightarrow N\backslash(\subscript G B)
\end{equation}
an \textit{extension of $G$-sets}.

\subsection{Orbit and stabilizer problems in \boldmath $G$-sets}
We consider a class $\Omega$ of ``computable'' groups as above. The
main step in our program to reduce decision problems on extensions
to decision problems on kernel and quotient is contained in the
following
\begin{algo}\label{algo:RedConjAndEqPr}
  Let $\subscript G B$ be a $G$-set, let $N$ be a central subgroup of $G$, let
  \[N\overset\iota\hookrightarrow G \overset\pi\twoheadrightarrow
    G/N\qquad\text{and}\qquad \subscript NB\hookrightarrow\subscript G B
    \overset\pi\twoheadrightarrow N\backslash(\subscript G B)
  \]
  be the corresponding extensions of groups and sets respectively, and
  let $\phi\colon H\to G$ be a group homomorphism. Assume that $H,N,G$
  all belong to $\Omega$, and that the orbit and stabilizer problems
  are solvable in $(\iota,\subscript G B)$.
  
  \textsc{Given} 
  \begin{list}{---}{\leftmargin=0pt \itemindent=1.5em}
  \item elements $b_1,b_2\in B$ such that $b_1^\pi$ and $b_2^\pi$ are
    in the same orbit of $N\backslash(\subscript G B)$;
  \item an element $h'\in H$ such that $b_1^\pi=(h'^\phi b_2)^\pi$;
  \item the stabilizer of $b_1^\pi$ as a computable subgroup of $H$
    that belongs to $\Omega$,
  \end{list}
  \textsc{Decide} whether $b_1,b_2$ belong to the same orbit of
  $\subscript{H^\phi}{B}$, if yes, \textsc{Find} an element $h\in H$ such that $b_1=h^\phi b_2$, and
  \newline \textsc{Compute} their stabilizers as kernels of maps
  $\chi\colon\tilde H\to A$
  \textsc{as follows:}\upshape
  
  \begin{enumerate}
  \item Compute an element $a'\in N$ such that $b_1=a'(h')^\phi b_2$,
    using the assumption that the orbit problem is solvable in
    $(\iota,\subscript G B)$.\label{RCAPR:step:1}
  \item Compute the stabilizer
    $H_{b_1^\pi}=\{h\in H\mid (h^\phi b_1)^\pi=b_1^\pi\}$ of $b_1^\pi$
    and the stabilizer $N_{b_1}=\{a\in N\mid a b_1=b_1\}$ of $b_1$.
  \item Compute the natural homomorphism\label{RCAPR:step:3}
    \[\chi\colon H_{b_1^\pi}\to N/N_{b_1}\cong N G_{b_1}/G_{b_1}.\]
  \item By computing in the abelian group $N/N_{b_1}$, check whether
    there exists $h\in H_{b_1^\pi}$ with $h^\chi=a' N_{b_1}$. If not,
    $b_1$ and $b_2$ are not in the same orbit. If yes, then
    $b_1=(h h')^\phi b_2$ and return $h h'$.
  \item Return the stabilizer of $b_1$ as the kernel of $\chi$. 
  \end{enumerate}.
\end{algo}
\begin{proof}[Correctness of the algorithm]
  At step~\eqref{RCAPR:step:1} the equality
  $b_1^\pi=((h'^\phi)b_2)^\pi$ is lifted to $\subscript G B$. It
  remains to check whether $a'$ may be chosen in $H^\phi$.

  As step~\eqref{RCAPR:step:3}, the map $\chi$ is computable because
  the quotient $N/N_{b_1}$ is computable as a quotient of finitely
  generated abelian groups, and for each $h\in H_{b_1^\pi}$ there is a
  computable $a\in N$ with $h^\phi b_1=ab_1$, for which
  $h^\chi=a N_{b_1}$.
\end{proof}

\begin{cor}
 If the orbit and stabilizer problems are
  solvable in $(\phi\pi,N\backslash(\subscript G B))$ and in $(\iota,\subscript G B)$,
  then they are sub-solvable in $(\phi,B)$.
\end{cor}

We remark that the stabilizer cannot be better written than as kernel
of map to some abelian group; in particular, it can be infinitely
generated, and therefore not directly describable by itself,
see~\S\ref{ss:complicated centralizer}.

\subsection{Extensions of left-free bisets}
A $G_1$-$G_2$-biset is a $(G_1\times G_2)$-set with a fixed
decomposition of $G_1\times G_2$ as a product of $G_1$ and $G_2$.
Namely, if $\subscript{G_1}B_{G_2}$ is a biset, then the \textit{associated
  $(G_1\times G_2)$-set} $\subscript{G_1\times G_2}B$ is defined by
\begin{equation}\label{eq:GsetsVsBisets}
  (g_1,g_2)\cdot b=g_1 b g_2^{-1}.
\end{equation}
In the inverse direction, if $G_1\times G_2$ acts on $B$, then the
$G_1$- and $G_2$-actions on $B$ commute and~\eqref{eq:GsetsVsBisets}
defines a $G_1$-$G_2$-biset structure on $B$.

If $N_1$, $N_2$ are normal subgroups of $G_1$, $G_2$ respectively,
then the $N_1$-$N_2$-biset $\subscript{N_1}B_{N_2}$ and the quotient
$G_1/N_1$-$G_2/N_2$-biset $N_1\backslash(\subscript{G_1}B_{G_2})/N_2$ are
defined as for $G$-sets using~\eqref{eq:GsetsVsBisets}.

\begin{defn}[Extensions of bisets]\label{defn:BisExt}
  Let $\subscript{G_1}B_{G_2}$ be a $G_1$-$G_2$-biset and let $N_1,N_2$ be
  normal subgroups of $G_1$ and $G_2$ respectively, so that for
  $i=1,2$ we have short exact sequences
  \begin{equation}\label{eq:ExtGr12}
    1\longrightarrow N_i\longrightarrow G_i\overset\pi\longrightarrow Q_i\longrightarrow 1.
  \end{equation}
  If the quotient $Q_1$-$Q_2$-biset $N_1\backslash B/N_2$, consisting
  of connected components of $\subscript{N_1}B_{N_2}$, is left-free, then the
  sequence
  \begin{equation}\label{eq:BisetExt}
    \subscript{N_1}B_{N_2}\hookrightarrow\subscript{G_1}B_{G_2}
    \overset\pi\twoheadrightarrow\subscript{Q_1}(N_1\backslash B/N_2)_{Q_2}
  \end{equation}
  is called an \emph{extension of left-free bisets}.
\end{defn}

\noindent The following statement follows directly from
Definition~\ref{defn:BisExt}.
\begin{prop}\label{prop:ExtOfBis}\pushQED{\qed}
  Let $\subscript{G_1}B_{G_2}$ be a left-free $G_1$-$G_2$-biset and
  let $N_1,N_2$ be normal subgroups of $G_1,G_2$ respectively.  Then
  the sequence~\eqref{eq:BisetExt} is an extension of left-free bisets
  if and only if the following property holds:
  \[\text{if $g_1\in G_1$, $n_2\in N_2$, $b\in B$, and $g_1b n_2=b$,
      then $g_1\in N_1$.}\qedhere\]
\end{prop}

\noindent Let $\subscript G B_G$ be a $G$-biset. We are interested in
the following decision problems in $B$:
\begin{description}
\item[The conjugacy problem] Given $b_1,b_2\in B$, does there exist $g\in
  G$ with $g\cdot b_1=b_2\cdot g$? If so, find one.
\item[The centralizer problem] Given $b\in B$, compute
  \[Z(b)\coloneqq\{g\in G\mid g\cdot b=b\cdot g\}.\]
\end{description}
To solve these problems, we remark that conjugacy and centralizer
problems in bisets are special cases of the orbit and stabilizer
problems in $G$-sets, using \eqref{eq:GsetsVsBisets} with $H=G$ and
$\phi\colon H\to G\times G$ given by $g^\phi=(g,g)$. We will use the
same conventions for conjugacy and stabilizer problems in bisets as
for orbit and stabilizer problems in $G$-sets.

\begin{exple}
  A natural source of examples of biset extensions appears as
  follows. Consider for $i=1,2$ path-connected topological bundles
\begin{equation}\label{eq:bundle}
  \begin{tikzcd}
    F_i\arrow[hook]{r} & E_i\arrow[->>]{d}{p_i}\\
    & X_i
  \end{tikzcd}
\end{equation}
with path-connected fibres; fix basepoints $*_i\in F_i\subseteq E_i$,
so $F_i=p_i^{-1}(p_i(*_i))$. Consider the fundamental groups
$N_i=\pi_1(F_i,*_i)$, $G_i=\pi_1(E_i,*_i)$ and
$Q_i=\pi_1(X_i,p_i(*_i))$. Then we have short exact sequences of
groups $N_i\hookrightarrow G_i\twoheadrightarrow Q_i$ provided that the connecting homomorphisms $\beta_i\colon \pi_2(X_i)\to N_i$ are trivial (alternatively we may define $N_i $ as $\pi_1(F_i,*_i)/\beta_i( \pi_2(X_i))$). Let there also
be given a bundle partially-defined covering map
$(f\colon E_1\dashrightarrow E_2,g\colon X_1\dashrightarrow X_2)$,
with $p_2\circ f=g\circ p_1$. Recall (as in~\eqref{eq:Dfn:SphBis})
\begin{equation}\label{eq:biset}
  B(f)=\{\beta\colon[0,1]\to E_1\mid\beta(0)=*_1,\,f(\beta(1))=*_2\}\,/\,{\approx}.
\end{equation}
We then have a short exact sequence of bisets
\[\subscript{N_1}B(f)_{N_2}\hookrightarrow\subscript{G_1}B(f)_{G_2}\twoheadrightarrow{Q_1}(N_1\backslash
  B(f)/N_2)_{Q_2},
\]
which are all left-free by Proposition~\ref{prop:ExtOfBis}, since 
for every loop in a fibre $F_2$ its $f$-lifts are all confined to
fibres. Furthermore,
\begin{lem}Suppose that $p_1^{-1}(y)\cap f^{-1}(p^{-1}_2(f(y)))$ has a single path connected component. Then the natural map $\beta\mapsto p_1\circ\beta$ defines a
  $Q_1$-$Q_2$-biset isomorphism
  \[N_1\backslash B(f)/N_2\overset\cong\longrightarrow B(g).\]
\end{lem}
\begin{proof}
  Clearly, $\beta\mapsto p_1\circ\beta$ defines an epimorphism; we
  need to show that is is also injective.  Consider
  $\beta,\beta'\in B(f)$ with $p_1\circ\beta\approx
  p_1\circ\beta'$. Then $\beta,\beta'$ end at the same fibre $F'$ of
  $p_1\colon E_1 \twoheadrightarrow X_1$. We may choose $n_2\in N_2$
  such that $\beta=\beta' n_2$. Finally there exists an $n_1\in N_1$
  such that $n_1\beta n_2=\beta'$ (because
  $p_1\circ\beta\approx p_1\circ\beta'$).
\end{proof}
\end{exple}

\section{Mapping class groups and mapping class bisets}
First, we recall basic properties of mapping class groups (for more
details, see~\cite{farb-margalit:mcg} and~\cite{hemion:homeos}) and
the solution of the conjugacy problem in mapping class groups.

\subsection{Mapping class groups}
Consider a marked sphere $(S^2,A)$, with $A$ the set of marked points,
and a multicurve $\CC\subset(S^2,A)$, see~\S\ref{ss:multicurves}. Let
\begin{itemize}
\item $\Mod(S^2,A)$ be the pure mapping class group of $(S^2,A)$; it
  is the set of homeomorphisms $h\colon S^2\selfmap$ which
  fix $A$, considered up to isotopy rel $A$;
\item $\Mod(S^2,A,\CC)$ be the maximal subgroup of $\Mod(S^2,A)$
  that fixes each curve of $\CC$ up to isotopy;
\item $\Mod[e](S^2,A,\CC)$ be the subgroup of $\Mod(S^2,A,\CC)$
  generated by Dehn twists around $\CC$;
\item $\Mod[v](S^2,A,\CC)$ be the pure mapping class group of
  $(S^2,A)\setminus \CC$.
\end{itemize}
Their meaning is the following: $\Mod$ is the classical mapping class
group, fixing a marked set and possibly a multicurve.  We remark that
$\Mod(S^2,A)$ could equally well have been defined as \emph{homotopy}
classes of homeomorphisms, see~\cite{epstein:isotopies}*{\S6}.
$\Mod[e]$ is the group of mapping classes that act only on edges of
the decomposition of $S^2$ along $\CC$. Similarly, $\Mod[v]$ is the
group of mapping classes that act disjointedly on vertices of the
decomposition of $S^2$ along $\CC$.

It is known~\cite{farb-margalit:mcg}*{Lemma~3.11} that $\Mod[e](S^2,A,\CC)$
is isomorphic to $\Z^{\#\CC}$ and is in the center of
$\Mod(S^2,A,\CC)$. Furthermore, there is a central exact sequence of groups
\begin{equation}\label{eq:Mp_Gr_seq}
  1\longrightarrow\Mod[e](S^2,A,\CC)\overset \iota\longrightarrow
  \Mod(S^2,A,\CC)\overset\pi\longrightarrow\Mod[v](S^2,A,\CC)\longrightarrow1
\end{equation}
which is easily seen to be split,
see~\cite{farb-margalit:mcg}*{Proposition~3.20}.

\begin{thm}[Nielsen~\cite{nielsen:surfaces}, Thurston~\cite{thurston:surfaces}]\label{thm:nielsen-thurston}
  Every homeomorphism $f\colon(S^2,A)\selfmap$ may be isotoped to a
  homeomorphism $g\colon(S^2,A)\selfmap$ that either
  \begin{itemize}
  \item is periodic;
  \item preserves a multicurve on $(S^2,A)$ (in this case, $g$ is
    called \emph{reducible}); or
  \item is pseudo-Anosov.
  \end{itemize}
  A pseudo-Anosov homeomorphism is never periodic nor reducible.\qed
\end{thm}
Furthermore, there are algorithms, such as ``train track'' technology
(see~\cite{bestvina-h:tt}) that determine to which class $f$ belongs,
and compute the multicurve along which $f$ can be reduced.

Let $\CC_f$ be the minimal multicurve that is invariant by $f$ and
such that all periodic components of
$f\colon(S^2,A)\setminus \CC\selfmap$ are, up to isotopy, either
pseudo-Anosov or of finite order. The multicurve $\CC_f$ is unique, so
$f$ belongs to $\Mod(S^2,A,\CC_f)$ because a mapping class permuting
non-trivially a multicurve must also permute the marked points.

The following algorithm solves the conjugacy problem under
$\Mod(S^2,A)$ for homeomorphisms $f,g\colon(S^2,A)\selfmap$. The
problem of combinatorial equivalence between homeomorphisms
$(S^2,A)\selfmap$ reduces to the pure case because the non-pure mapping class group of $(S^2,A)$ is a finite extension of $\Mod(S^2,A)$. The extension of this algorithm to
non-invertible maps is the main outcome of this series of articles.

\begin{algo}\label{algo:conjmcg}
  \textsc{Given} homeomorphisms $f,g\colon(S^2,A)\selfmap$,\\
  \textsc{Compute} whether they are conjugate by $\Mod(S^2,A)$, and \textsc{compute} the centralizer of $f$ in $\Mod(S^2,A)$ \textsc{as follows:}\upshape
  \begin{enumerate}
  \item Compute the canonical multicurves $\CC_f,\CC_g$. They may be
    constructed e.g.\ using the Bestvina-Handel ``train track''
    algorithm~\cite{bestvina-h:tt}.

  \item Check whether there is an $h\in \Mod(S^2,A)$ such that $h$
    conjugates $g\colon \CC_g\selfmap$ to $f\colon \CC_f\selfmap$. If
    $h$ does not exist, then $f$ and $g$ are not conjugate by
    $\Mod(S^2,A)$. Otherwise choose one such $h$ (all such $h$ differ
    by postcomposition by elements in $\Mod(S^2,A,\CC_f)$) and replace
    $g$ by $g^h$.  Then $\CC_f=\CC_g\eqqcolon\CC$ and the conjugacy
    problem for $f,g$ under $\Mod(S^2,A)$ and under $\Mod(S^2,A,\CC)$
    are equivalent.

  \item Check whether the projections
    $\pi(f),\pi(g)\in\Mod[v](S^2,A,\CC)$ are conjugate under
    $\Mod[v](S^2,A,\CC)$. Every periodic component of $\pi(f)$ and
    $\pi(g)$ is either of finite order or pseudo-Anosov; and in both
    of these cases the conjugacy and centralizer problems are
    solvable, see~\cite{hemion:homeos}.

    If the projections are conjugate, choose an element
    $q\in\Mod(S^2,A,\CC)$ such that $\pi(f)=\pi(g^q)$. Since
    $\Mod[e](S^2,A)$ belongs to the centre of $\Mod(S^2,A,\CC)$, the rotation
    parameter $f^{-1}g^q\in\Mod[e](S^2,A,\CC)$ does not depend on the
    choice of $q$.
  \item The maps $f,g$ are conjugate under $\Mod(S^2,A,\CC)$ if and
    only if $f^{-1}g^q$ is the identity in $\Mod[e](S^2,A,\CC)$, and a
    conjugator is for example $q$.

  \item The centralizer $Z(f)$ is a subgroup of $\Mod(S^2,A,\CC)$, and
    we describe it in terms of~\eqref{eq:Mp_Gr_seq}. Its intersection
    with $\Mod[e](S^2,A,\CC)\cong\Z^\CC$ is easy to compute: it
    consists of those maps $\CC\to\Z$ that are constant on every
    $f$-orbit in $\CC$.

    Consider next the image of $Z(f)$ in $\Mod[v](S^2,A,\CC)$. Let
    $\widehat S_0,\widehat S_1\coloneqq f( \widehat
    S_0),\dots,\widehat S_n= \widehat S_0$ be a cycle of small spheres
    under $\pi(f)$; for simplicity, write
    $f_i\colon \widehat S_i\to \widehat S_{i+1}$, with indices
    computed modulo $n$, for the restrictions of $f$ to the small
    spheres; they are all easily computable.

    Elements of $Z(f)$ must be constant on orbits of small spheres in
    the sense that, if $z_i\colon \widehat S_i\selfmap$ are the
    co\"ordinates of an element centralizing $f$, then
    $z_i^{f_i}=z_{i+1}$, indices again read modulo $n$. Compute the
    return map $f'=f_0f_1\dots f_{n-1}\colon \widehat
    S_0\selfmap$. Then, by our assumption, the map $f'$ is either of
    finite order or pseudo-Anosov, and this can be determined by
    checking if $(f')^e=\one$ for the smallest $e$ such that $(f')^e$
    is pure. Suppose that $f'$ is of finite order. Consider the
    quotient surface $\widehat S_0/f'$; it is also a sphere, with up
    to two extra punctures corresponding to the centres of rotation of
    $f'$. Then $Z(f')$ is isomorphic to a finite-index subgroup of
    $\Mod(\widehat S_0/f')$: every $z_0\in Z(f')$ projects via
    $\widehat S_0\to \widehat S_0/f'$ to an element in $\Mod(\widehat
    S_0/f')$, and conversely every $h\in\Mod(\widehat S_0/f')$ lifts
    to $e$ possibly non-pure mapping classes commuting with $f$,
    namely $\widetilde h,f\widetilde h,\dots,f^{e-1}\widetilde h$, and
    at most one of these is pure.
    
    Let us now briefly consider the case that $f'$ is pseudo-Anosov;
    details will appear in~\cite{bartholdi-dudko:bc5}. Since $f'$
    preserves transverse foliations on $\widehat S_0$, its centralizer
    coincides with its \emph{radical},
    $\sqrt{f'}=\{z\in\Mod(\widehat S_0)\mid \exists k,\ell\neq 0\text{
      with }z^k=(f')^\ell\}$. The map $f'$ admits a \emph{train track
      representation}, namely there exists a $1$-dimensional
    submanifold with forks on $\widehat S_0$, such that $f'$ and in
    fact all elements of $\sqrt{f'}$ may be represented by
    non-negative integer labellings of the train track subject to some
    arithmetic conditions. It follows that $\sqrt{f'}$ is a cyclic
    subgroup, and since the numbers on the train track increase upon
    taking powers, a generator of $\sqrt{f'}$ may be found by
    enumerating all labellings with numbers less than those of $f'$ to
    find a minimal root of $f'$.
  \end{enumerate}
\end{algo}

\subsection{Mapping class bisets}\label{ss:mcb}
We now turn to maps of degree $>1$.  Kameyama introduces
in~\cite{kameyama:thurston}*{\S4} a semigroup containing
$\Mod(S^2,A)$,
\[K(S^2,A)=\{\text{Thurston maps }(S^2,A)\selfmap\}\,/\,{\approx}
\]
and derives some important properties. Clearly $K(S^2,A)$ is a
$\Mod(S^2,A)$-biset, under pre- and post-composition of Thurston
maps. It is a graded semigroup under the map
$\deg\colon K(S^2,A)\to\N^*$, and is finitely generated in every
degree.

We will need the more general situation afforded by non-dynamical
maps, i.e. maps with different domain and range. Let therefore
$f\colon(S^2,C)\to(S^2,A)$ be a sphere map, and let $\DD,\CC$ be
multicurves on $S^2\setminus C,S^2\setminus A$ respectively with
$\DD\subseteq f^{-1}(\CC)$.
\begin{defn}[Mapping class bisets]
\label{def:MapClassBis}
  The $\Mod(S^2,C)$-$\Mod(S^2,A)$-biset $M(f,C,A)$ is defined as
  \[M(f,C,A)=\{m'fm''\mid m'\in\Mod(S^2,C),m''\in\Mod(S^2,A)\}\,/\,{\approx}.\]
  It admits as a subbiset the  $\Mod(S^2,C,\DD)$-$\Mod(S^2,A,\CC)$-biset
  \[M(f,C,A,\DD,\CC)=\{m'fm''\mid m'\in\Mod(S^2,C,\DD),m''\in\Mod(S^2,A,\CC)\}\,/\,{\approx}.\]
  The left and right actions are given by
  $m'fm''=m''\circ f\circ m'$, in keeping with using the algebraic
  order of operations in bisets.
\end{defn}

In case $G=\pi_1(S^2\setminus A,*)$ is a sphere group, we write
$\Mod(G)$ for $\Mod(S^2,A)$; by Theorem~\ref{thm:dehn-nielsen-baer},
it is a subgroup of the outer automorphism group of $G$, and is
defined algebraically in terms of $G$ and its peripheral conjugacy
classes. If $\gf$ be a sphere tree of groups, we denote by $\Mod(\gf)$
its group of pure self-congruences, so that if $\CC$ is a multicurve
on $(S^2,A)$ then $\Mod(\gf)\cong\Mod(S^2,A,\CC)$ for the tree of
groups decomposition $\gf$ of $G$ along $\CC$. Similarly, we write
\[K(G)=\{\text{$G$-$G$-\text{bisets}}\}\,/\,{\cong}
\]
for the $\Mod(G)$-biset isomorphism classes of bisets of Thurston maps
$(S^2,A)\selfmap$, which by Theorem~\ref{thm:dehn-nielsen-baer+} is in
bijection with $K(S^2,A)$.

If $B$ be a sphere $H$-$G$-biset, then by $M(B)$ we denote the
$\Mod(H)$-$\Mod(G)$-biset
\[M(B)=\{B_\psi\otimes B\otimes
  B_\phi\mid\psi\in\Mod(H),\phi\in\Mod(G)\}\,/\,{\cong}.
\]
By Theorem~\ref{thm:dehn-nielsen-baer+}, for a sphere map
$f\colon(S^2,C)\to(S^2,A)$ the bisets $M(f,C,A)$ and $M(B(f))$ are
isomorphic, so $M(f,C,A)$ can also be viewed as the biset of
isomorphism classes of bisets of the form
$B(m')\otimes B(f)\otimes B(m'')$ with
$m'\in\Mod(S^2,C),m''\in\Mod(S^2,A)$,

Similarly, if $\gfB$ be an $\gfY$-$\gf$-tree of bisets,
we define
\[M(\gfB)=\{\mathfrak N' \otimes \gfB \otimes
  \mathfrak N''\mid \mathfrak N'\in\Mod(\gfY),\mathfrak
  N''\in\Mod(\gf)\}\,/\,{\cong},
\]
and note that for a sphere map $f\colon(S^2,C,\DD)\to(S^2,A,\CC)$ the
bisets $M(f,C,A,\DD,\CC)$ and $M(\gfB(f))$ are isomorphic. This
may be seen in a manner that makes more explicit the tree of bisets
decomposition of $f$: let $\subscript\gfY\gfB_\gf$ be the
tree of bisets decomposition of $f$. Every element
$b\in M(f,C,A,\DD,\CC)$ is encoded by a well-defined tree of bisets
$\subscript\gfY\gfC_\gf$ over the same trees of groups
$\gfY,\gf$; so $M(f,C,A,\DD,\CC)$ is the biset of isomorphism
classes of such trees of bisets.

\subsubsection{Properties of Mapping class bisets}
\begin{prop}[An extension of~\cite{kameyama:thurston}*{Proposition~4.1}]\label{prop:MCBfree}
  The bisets $M(f,C,A)$ and $M(f,C,A,\DD,\CC)$ are left-free.
\end{prop}
\begin{proof}
  The second claim obviously follows from the first.  Consider
  $g\in M(f,C,A)$ and $m\in\Mod(S^2,C)$ with $mg=g$. We use the
  same letters $g,m$ to denote a representative Thurston map and
  homeomorphism respectively; so $g$ and $g\circ m$ are isotopic along
  an isotopy fixing $C$. Lifting the isotopy from $g$ to $g\circ m$
  gives $g=g\circ m'$ for some homeomorphism $m'$ isotopic to $m$. We
  wish to show that $m'$ is isotopic to the identity.

  Consider a path $\gamma\colon[0,1]\to S^2\setminus A$ with
  $\gamma(0)\in A$, possibly a critical value; and consider its
  $g$-preimages $\gamma_1,\dots,\gamma_d$. They are permuted by $m'$,
  in the sense that $m'\circ\gamma_i\in\{\gamma_1,\dots,\gamma_d\}$,
  so there exists $e\in\{1,\dots,d\}$ such that
  $(m')^e\circ\gamma_1=\gamma_1$ and in particular
  $(m')^e(\gamma_1(1))=\gamma_1(1)$. This continues to hold as
  $\gamma_1$ is deformed, and $\gamma_1(1)$ is arbitrary, so
  $(m')^e=1$, and $m'$ being of finite order is isotopic to a M\"obius
  transformation by Theorem~\ref{thm:nielsen-thurston}. If $\#C\ge3$,
  we are done since a M\"obius transformation fixing three points is
  trivial, while if $\#C=2$ then $m'$ is isotopic to a rotation on a
  sphere with two marked points, and again is isotopic to the
  identity.
\end{proof}

\begin{exple}
  Let us note that the non-pure version of the mapping class biset
  needs not be left-free; here is the simplest example. Consider the
  sphere map $f(z)=z^2\colon
  (\hC,\{-1,1\})\to(\hC,\{0,1,\infty\})$. The non-pure mapping class
  group $\Mod^*(\hC,\{-1,1\})$ of $(\hC,\{-1,1\})$ consists of two
  elements $\one, z\to -z$. Since $z^2 \circ [ z\to -z]=z^2$ the left
  action of $\Mod^*(\hC,\{-1,1\})$ on $\{f\}$ is trivial.
\end{exple}

We extract, out of a left-free biset, relevant information from which
a basis of mapping class bisets can be constructed.
\begin{defn}[Distillations]\label{defn:distillation}
  Let $G$ be a sphere group with fixed Hurwitz generators
  $\gamma_1,\dots,\gamma_n$. Let $\subscript H B_G$ be a sphere biset,
  and choose a basis $S$ of $B$. Consider the \emph{wreath map}
  associated with $S$,
  see~\cite{bartholdi-dudko:bc1}*{\S\ref{bc1:ss:wreath}}: the map
  $\Phi\colon G\to H\wr S\perm$ such that $s\cdot g=h\cdot t$ if
  $\Phi(g)=\pair{\cdots, h,\cdots}\pi$ with $s^\pi=t$ and the `$h$' in
  position $s$.

  Write $\Phi(\gamma_i)=\pair{h_{i,1},\dots,h_{i,d}}\pi_i$. For each
  cycle $C_{i,j}$ of $\pi_i$, let $\widehat{h_{i,j}}$ be the conjugacy
  class, in $H$, of the product of the $h_{i,k}$ along the cycle. Then
  each $\widehat{h_{i,j}}$ is either trivial or a peripheral conjugacy
  class in $H$. We call the collection
  $\overline\Phi=\{\pi_i,\widehat{h_{i,j}}\}$ the \emph{distillation}
  of $\Phi$.

  Note that we consider distillations up to the diagonal action by
  conjugation by a permutation of $S$, since the set $S$ is not
  assumed ordered. Note also that there are finitely many
  distillations, since there are finitely many choices for the
  permutations $\pi_i$ and the conjugacy classes
  $\widehat{h_{i,j}}$. Furthermore, all peripheral conjugacy classes
  appear exactly once among the $\widehat{h_{i,j}}$, by
  Definition~\ref{dfn:SphBis}\eqref{cond:3:dfn:SphBis}.
\end{defn}

\begin{lem}\label{lem:distillations}
  Let $M(B_0)$ be the mapping class biset of a sphere biset
  $\subscript H{(B_0)}_G$, and consider a biset $B\in M(B_0)$, with
  wreath recursion $\Phi$. Then the distillation $\overline\Phi$
  depends only on $B$, and is called the \emph{distillation $\overline
    B$ of $B$}.

  Let $B,C\in M(B_0)$ be bisets. Then $\overline B=\overline C$ if
  and only if $B,C$ belong to the same left orbit of $M(B_0)$.
\end{lem}
\begin{proof}
  For the first statement, the multiset of lifts, and the monodromy
  actions of the Hurwitz generators, are invariants of $B$.

  For the second claim, let $\Phi,\Psi$ be wreath recursions of $B,C$
  respectively.  If $\overline\Phi=\overline\Psi$ then they have the
  same permutations $\pi_i$, up to relabelling, and therefore the same
  cycles $C_{i,j}$. Let $h'_{i,j}$ be the product of the entries of
  $\Phi(\gamma_i)$ along the cycle $C_{i,j}$ and let $h''_{i,j}$ be
  the product of the entries of $\Psi(\gamma_i)$ along the same cycle
  $C_{i,j}$, in the same order. Then $h'_{i,j}$ and $h''_{i,j}$ are
  conjugate for each $i,j$, and the map $h'_{i,j}\mapsto h''_{i,j}$
  extends to an automorphism of $H$, because the only relation in $H$
  is that the product of the $h'_{i,j}$, in some order, is trivial,
  and the product of the $h''_{i,j}$, in the same order, is also
  trivial. This automorphism is-defined up to inner
  automorphisms. Since it preserves the peripheral conjugacy classes,
  it defines an element $m'\in\Mod(H)$, again by
  Theorem~\ref{thm:dehn-nielsen-baer}, and we have $\Psi=\Phi
  (m'\times\cdots\times m')$, so $B=m C$.

  Conversely, if $B=m C$ for $m\in\Mod(H)$ then obviously $\overline
  B=\overline C$.
\end{proof}

\begin{prop}[Essentially~\cite{koch-pilgrim-selinger:pullback}*{Proposition~3.1}]\label{prop:MCBfinite}
  The biset $M(f,C,A)$ has finitely many left $\Mod(S^2,C)$-orbits,
  and the biset $M(f,C,A,\DD,\CC)$ has finitely many left
  $\Mod(S^2,C,\DD)$-orbits.
\end{prop}
\begin{proof}
  For the first claim, choose a basis $S$ of $B(f)$, write
  $G=\pi_1(S^2\setminus A,*)$ and $H=\pi_1(S^2\setminus C,\dagger)$,
  and $B_0=B(f)$. The claim then follows from
  Lemma~\ref{lem:distillations}.

  The second statement follows from the first: since the right action
  of $\Mod(S^2,A)$ is transitive on $1\otimes M(f,C,A)$, this last set
  is of the form $N\backslash\Mod(S^2,A)$ for a finite-index subgroup
  $N\le\Mod(S^2,A)$. Then $1\otimes M(f,C,A,\DD,\CC)$ is isomorphic,
  qua right $\Mod(S^2,A,\CC)$-set, to
  $(N\cap\Mod(S^2,A,\CC))\backslash\Mod(S^2,A,\CC)$, and
  $[\Mod(S^2,A,\CC):N\cap\Mod(S^2,A,\CC)]\le[\Mod(S^2,A):N]<\infty$.
\end{proof}

The biset $M(f,C,A,\DD,\CC)$ restricts to an
$\Mod[e](S^2,C,\DD)$-$\Mod[e](S^2,A,\CC)$-biset, which furthermore
encodes the Thurston matrix of $f$, namely the map $T_{f,\CC}$
from~\S\ref{ss:invariant mc}.
\begin{prop}\label{prop:MCBedge}
  Assume $\DD=f^{-1}(\CC)$ and consider $g\in M(f,C,A,\DD,\CC)$.  If
  $e\in\Mod[e](S^2,A,\CC)$ and $m\in\Mod(S^2,C,\DD)$ satisfy $g e=m g$, then
  $m\in\Mod[e](S^2,C,\DD)$.

  Furthermore, identify $\Mod[e](S^2,A,\CC)$ with $\Z^\CC$ and
  $\Mod[e](S^2,A,\DD)$ with $\Z^\DD$ by writing these groups as generated
  by Dehn twists. If $g e = m g$, then we have $m=T_{f,\CC}(e)$.
\end{prop}
\begin{proof}
  Suppose $m\notin \Mod[e](S^2,C,\DD)$. Then there is a small sphere $U$
  in $S^2\setminus \DD$ such that the restriction of $m$ to $\widehat U$,
  call it $\widehat m\in \Mod(\widehat U)$, is non-trivial.  As in
  Definition~\ref{lem:SmallMaps}, denote by $\widehat g \colon \widehat U
  \to \widehat S$ the small sphere map induced by $g$. By
  Proposition~\ref{prop:MCBfree} we have $\widehat g\neq \widehat m\widehat
  g$, and this contradicts $\widehat g= \widehat{g e}$.

  Denote by $d$ the degree of $f$. Clearly, $m=T_{f,\CC}(e)$ if and only if
  $m^{d!}=T_{f,\CC}(e^{d!})$, since $T_{f,\CC}$ is a linear operator on
  free abelian groups.

  Consider a curve $\gamma\in\CC$ and its (possibly isotopic) preimages
  $\delta_1,\dots,\delta_m$ mapping to $\gamma$ with degrees
  $d_1,\dots,d_m$ respectively. Then the $d!$-th power of a Dehn twist
  about $\gamma$ lifts to the product, for $i=1,\dots,m$, of $d!/d_i$-th
  powers of Dehn twists about $\delta_i$. For each peripheral or trivial
  curve $\delta_i$, the corresponding Dehn twist is trivial in
  $\Mod[e](S^2,C,\DD)$, while powers add for Dehn twists about isotopic
  curves. We recover precisely~\eqref{eq:thurston matrix}.
\end{proof}

Consider two sphere maps $f\colon (S^2,D)\to (S^2,C)$ and
$g\colon (S^2,C)\to (S^2,A)$. Then there is a natural epimorphism
$M(f,D,C)\otimes M(g,C,A)\twoheadrightarrow M(f g,D,A)$ given by
composition; however this morphism needs not be injective,
see~\S\ref{ss:TensProd:MappClassBis}.

Recall~\cite{bartholdi-dudko:bc1}*{\S2} that to a group homomorphism
$\phi\colon H\to G$ we associate the $H$-$G$-set $B_\phi$, which, qua
right $G$-set, is plainly $G$; the left $H$-action is by
\[h\cdot b=h^\phi b.
\]
Suppose that $f\colon (S^2,C)\to (S^2,A)$ is a homeomorphism. Then it
induces an isomorphism $f_*\colon \Mod(S^2,C)\to \Mod(S^2,A)$ and the
biset $M(f,C,A)$ is biprincipal and is isomorphic to $B_{f_*}$. A bit
more general:
 
\begin{lem}\label{lem:mcb of homeo}
  Consider a sphere map $f\colon (S^2,C)\to (S^2,A)$ of degree
  $1$. Then $f$ induces the forgetful morphism
  $f^*\colon \Mod (S^2,A)\to \Mod(S^2,C)$ so that $f h= f^*(h) f$. The
  biset $M(f,C,A)$ is isomorphic to $B_{f^*}^{\vee}$; it is a left
  principal and right transitive biset. \qed
\end{lem}
\subsection{Computability of mapping class bisets}
Propositions~\ref{prop:MCBfree} and~\ref{prop:MCBfinite} point to the
fact that the mapping class biset $M(f,C,A,\DD,\CC)$ is
computable. Let us make this more precise.

Firstly, the group $G=\pi_1(S^2\setminus A,*)$ is computable, since it
is a finitely generated free group. The conjugacy problem is easy to
solve in $G$ --- conjugacy classes are just freely reduced words up to
cyclic permutations --- so multicurves $\CC$ are computable as
collections of conjugacy classes.

Secondly, the biset $B(f)$ of a sphere map $f\colon(S^2,C)\to(S^2,A)$
is computable, since it is a left-free biset of degree
$d=\deg(f)$. All bisets of the form $B(f)$ will be manipulated
algorithmically by choosing a basis of cardinality $d$ and working
with the wreath map
$\Phi\colon G\to\pi_1(S^2\setminus C,\dagger)\wr d\perm$. The
decomposition of $B(f)$ along a multicurve $\CC$ is computable by
Theorem~\ref{thm:DecompOfBiset}.

Thirdly, the group $\Mod(S^2,A)$ is computable, since by
Theorem~\ref{thm:dehn-nielsen-baer} it is a subgroup of the outer automorphism
group of $G$. Elements of $\Mod(S^2,A)$ will be manipulated
algorithmically as maps $G\selfmap$, by keeping track of
their values on the standard generators of $G$. The condition that a
map on generators $\phi\colon\gamma_i\mapsto w_i$, for $w_i\in G$,
defines an element of $\Mod(S^2,A)$ is easy to check: since
$\gamma_1\cdots\gamma_n=1$ we must have $w_1\cdots w_n=1$, and each
$w_i$ must be conjugate to $\gamma_i$. It is equally easy to check
whether $\phi$ defines an element of a subgroup such as
$\Mod(S^2,A,\CC)$: the corresponding automorphism of $G$ must
preserve the conjugacy classes in $\CC$.

Finally, the biset $M(f,C,A)$ and its subbiset $M(f,C,A,\DD,\CC)$ are
computable. We express this, following
Proposition~\ref{prop:MCBfinite}, as an
\begin{algo}\label{algo:compute M(f,C,A)}
  \textsc{Given} $f$ a sphere map $(S^2,C,\DD)\to(S^2,A,\CC)$,\\
  \textsc{Compute} the biset  $M(f,C,A,\DD,\CC)$ \textsc{as follows:}\upshape
  \begin{enumerate}
  \item Write
    $G=\langle\gamma_1,\dots,\gamma_n\mid\gamma_1\cdots\gamma_n\rangle$
    the fundamental group of $S^2\setminus A$, and identify elements
    of $\CC$ with conjugacy classes in $G$; similarly write
    $H=\langle\delta_1,\dots,\delta_m\mid\delta_1\cdots\delta_m\rangle$.
  \item Choose a basis of $B(f)$, and compute its wreath map $\Phi$ as
    a table with $n$ rows and $d+1$ columns, with $d=\deg(f)$, the
    $0$th one being a permutation of $\{1,\dots,d\}$ and the remaining
    being the entries of $\Phi(\gamma_i)$, written as words in the
    $\delta_j$.
  \item Make a list of all \emph{distillations}, see
    Definition~\ref{defn:distillation}: data structures consisting of
    a list of $n$ permutations $\pi_1,\dots,\pi_n$ whose product is
    $1$, up to the diagonal action by conjugation by a permutation of
    $\{1,\dots,d\}$, and for each cycle of each $\pi_i$ of a conjugacy
    class $\in\{1,\delta_1^H,\dots,\delta_m^H\}$; for a wreath map
    $\Psi$, denote by $\overline\Psi$ its distillation, obtained by
    computing the conjugacy class of product of the entries along each
    cycle.
  \item View the mapping class group $\Mod(S^2,A)$ as a group of
    automorphisms of $G$, considered up to inner automorphisms. It is
    generated by Dehn twists $\tau_{i,j}$ for all $1\le i<j\le n$; the
    map $\tau_{i,j}$ conjugates $\gamma_k$ by $\gamma_i\cdots\gamma_j$
    for all $k\in\{i,\dots,j\}$ and fixes the other generators.
  \item Start with $X=\{\Phi\}$, the wreath map of $f$, and
    $\overline X=\{\overline\Phi\}$. As long as there exists a
    generator $\tau_{i,j}\in\Mod(S^2,A)$ and a wreath map
    $\Psi\in X$ such that
    $\overline{\tau_{i,j}\Psi}\notin\overline X$, add
    $\tau_{i,j}\Psi$ to $X$ and add
    $\overline{\tau_{i,j}\Psi}\notin\overline X$ to $\overline X$.
  \item The set $X$ is now a basis for $M(f,C,A)$, and the wreath map
    of $M(f,C,A)$ may be computed as follows. For each
    $m\in\Mod(S^2,A)$ and each $\Psi\in X$, let $\Psi'\in X$ be such
    that $\overline{\Psi'}=\overline{m\Psi}$. Write $h'_{i,j}$,
    respectively $h''_{i,j}$, for the products of entries along the
    cycles of $m\Psi,\Psi'$ respectively. Let $m'\colon H\selfmap$ be
    the automorphism of $H$ defined by $h'_{i,j}\mapsto h''_{i,j}$; it
    may e.g.\ be normalized by computing the images of the standard
    generators $\delta_k$. Then the biset $M(f,C,A)$ is defined by the
    equations $\Psi\cdot m=m'\cdot\Psi'$.
  \item The biset $M(f,C,A,\DD,\CC)$ is a subbiset of $M(f,C,A)$ in
    the following way: rather than saturating $X$ under the action of
    all $\tau_{i,j}$, only consider those Dehn twists along simple
    closed curves that do not intersect $\CC$.
  \end{enumerate}
\end{algo}

\subsection{Extensions of mapping class bisets}\label{ss:extensionsmcb}
Since we are mainly interested in the dynamical situation of a
Thurston map, let us abbreviate
\[M(f,A)\coloneqq M(f,A,A),\qquad M(f,A,\CC)\coloneqq M(f,A,A,\CC,\CC).
\]
Let us denote by $eM(f,A,\CC)$ the biset $M(f,A,\CC)$ with restricted
actions, namely the $\Mod[e](S^2,A,\CC)$-biset
\[eM(f,A,\CC)\coloneqq \subscript{\Mod[e](S^2,A,\CC)}M(f,A,\CC)_{\Mod[e](S^2,A,\CC)}
\]
and let us define the \emph{small mapping class biset} $vM(f,A,\CC)$
as the $\Mod[v](S^2,A,\CC)$-biset of restrictions to $S^2\setminus\CC$
as in Definition~\ref{lem:SmallMaps}:
\begin{equation}\label{eq:vB}
  vM(f,A,\CC)\coloneqq\{\text{restriction of } g \text{ to }(S^2,A)\setminus\CC\colon g\in M(f,A,\CC)\}.
\end{equation}

\noindent It follows from Propositions~\ref{prop:ExtOfBis},
\ref{prop:MCBfree} and~\ref{prop:MCBedge} that
\begin{equation}\label{eq:Mp_Cl_biset_seq}
  eM(f,A,\CC)\hookrightarrow M(f,A,\CC)
  \overset\pi\twoheadrightarrow \Mod[e](S^2,A,\CC)\backslash M(f,A,\CC)/\Mod[e](S^2,A,\CC)
\end{equation}
is an extension of left-free bisets. We describe it more concretely as
follows.  First, we have the natural morphism
\begin{equation}\label{eq:forgetvM}
  \Mod[e](S^2,A,\CC)\backslash M(f,A,\CC)/\Mod[e](S^2,A,\CC)\to vM(f,A,\CC)
\end{equation}
which maps every $g\in M(f,A,\CC)$ to its restriction to
$S^2\setminus\CC$.  This morphism is finite-to-one, because both
bisets are left-free and the source biset has finitely many left
orbits.

\begin{algo}
\label{algo:from vM to M:eMod}
 \textsc{Given} two $g_1,g_2$  in the source
  of~\eqref{eq:forgetvM} and an \textsc{oracle} solving the conjugacy and centralizer problems for the images of $g_1,g_2$ under~\eqref{eq:forgetvM}, \textsc{Solve} the conjugacy and centralizer problem between $g_1$ and $g_2$ and \textsc{Compute} the centralizer of $g_1$
   \textsc{as follows:}\upshape
   
   \noindent denote by $\pi$ the map~\eqref{eq:forgetvM} and denote by
  $g_1^\pi,g_2^\pi\in vM(f,A,\CC)$ the images of $g_1,g_2$ under $\pi$
\begin{enumerate}
 \item If $g_1^\pi$ and $g_2^\pi$ are not
  conjugate, then neither are $g_1$ and $g_2$.
  \item Choose $h\in\Mod[v](S^2,A,\CC)$ such that
  $h g_1^\pi=g_2^\pi h $. Compute the finite set $F$ of
  $\pi$-preimages of $g_1^\pi$. (Observe that $F\ni g_1,g_2^h$.)
  \item Compute the action of the centralizer $Z(g_1^\pi)$ of
  $g_1^\pi$ on $F$.
  \item Check if $g_1$ and $g_2^h$ are in the same
  orbit under the action of $Z(g_1^\pi)$. If no, then $g_1$ and $g_2$ are not conjugate. If yes, then find $m\in Z(g_1^\pi)$ such that $g_1=(g_2^h)^m$. Return $hm$. 
\item Compute the centralizer of $g_1$: it is the stabilizer of $g_1$ in $F$ under $Z(g_1^\pi)$ (and it has a
  finite index in $Z(g_1^\pi)$, hence it is computable).
\end{enumerate} 
 \end{algo}

Recall that $M(f, A,\CC)$ is naturally isomorphic to $M(\gfB)$ if
$\gfB$ is the sphere tree of bisets of $f$, see
Definition~\ref{defn:gfofbisets}. Similarly, $eM(f, A,\CC)$ and
$vM(f, A,\CC)$ can be identified with $eM(\gfB)$ and $vM(\gfB)$; for
instance,~\eqref{eq:vB} takes form
\begin{equation}\label{eq:vB as grBis}
  vM(\gfB)\coloneqq\{\text{essential bisets in } \gfB\}.
\end{equation}

\subsection{Decomposition of \boldmath $M(f,C,A,\DD,\CC)$}
\label{ss:DecompOfMCB}
The operation implicit in~\eqref{eq:forgetvM} may be made more explicit,
both at the topological and algebraic levels; we return to the general
(non-dynamical) setting of $f\colon (S^2,C,\DD)\to (S^2,A,\CC)$ with
$\DD=f^{-1}(\CC)$ rel $C$. The mapping class group $\Mod$ decomposes into a
direct product of $\Mod[e]$ and $\Mod[v]$,
see~\eqref{eq:Mp_Gr_seq}. Mapping class bisets encode non-invertible maps,
and decompose similarly, as we shall see below. Similarly
to~\eqref{eq:vB},~\eqref {eq:Mp_Cl_biset_seq}, and ~\eqref{eq:forgetvM}
write
\begin{gather}
  vM(f,C,A,\DD,\CC)\coloneqq\{\text{restriction of } g \text{ to }(S^2,A)\setminus\DD\colon g\in M(f,C,A,\DD,\CC)\},\notag\\
  \pi \colon  M(f,C,A,\CC,\DD)\twoheadrightarrow \Mod[e](S^2,C,\DD)\backslash M(f,C,A,\DD,\CC)/\Mod[e](S^2,A,\CC),\notag\\
  \sigma \colon \Mod[e](S^2,C,\DD)\backslash M(f,C,A,\DD,\CC)/\Mod[e](S^2,A,\CC) \to  vM(f,C,A,\DD,\CC).\label{eq:sigma}
\end{gather}

\begin{lem}\label{lem:ExtOFvM}
  With $\sigma$ as in~\eqref{eq:sigma}, there is a finite set $S$
  endowed with a right action of $\Mod[v](S^2,A,\CC)$ and there is a
  biset intertwiner
  \[\tau\colon \Mod[e](S^2,C,\DD)\backslash M(f,C,A,\DD,\CC)/\Mod[e](S^2,A,\CC)\to \subscript{\one}{S} _{\Mod[v](S^2,A,\CC)}
  \]
  such that $\tau$ restricts to a bijection $\sigma^{-1}(b)\to S$ for
  every $b\in vM(f,C,A,\DD,\CC)$. Endow $S$ with the trivial left
  action of $\Mod[v](S^2,C,\DD)$. Then $\sigma\times \tau$ is a
  $\Mod[v](S^2,C,\DD)$-$\Mod[v](S^2,A,\CC)$-biset isomorphism
  \[\Mod[e](S^2,C,\DD)\backslash M(f,C,A,\DD,\CC)/\Mod[e](S^2,A,\CC) \cong vM(f,C,A,\DD,\CC) \times S.\] 
\end{lem}
\begin{proof}
  The set $S$ is the set of restrictions of maps in
  $\Mod[e](S^2,C,\DD)\backslash M(f,A,\CC)/\Mod[e](S^2,A,\CC)$ to all
  non-essential spheres of $(S^2,f^{-1}(A),f^{-1}(\CC))$. More
  precisely, let $\subscript\gfY{\gfB(g)}_\gf$ be the sphere tree of
  bisets of $g\in M(f,C,A,\DD,\CC)$, let $\tau(\gfB(g))$ be $\gfB$
  minus all its essential sphere vertex bisets, and let
  $\sigma (\gfB(g))$ be the set of essential sphere bisets in
  $\gfB(g)$.

  By construction, $\Mod[e](S^2,C,\DD)g=\Mod[e](S^2,C,\DD)g'$ if and
  only if
  $[\tau(\gfB(g)),\sigma(\gfB(g))]=[\tau(\gfB(g')),\sigma(\gfB(g'))]$. We
  define $S$ as the quotient of
  $\{\tau(\gfB(g))\mid g\in M(f,C,A,\DD,\CC)\}$ by the right action of
  $\Mod[e](S^2,A,\CC)$.
\end{proof}

Let
\begin{equation}\label{eq:defn:H_f}
  H_f\coloneqq \{e\in  \Mod[e](S^2,A,\CC) \mid f e\in  \Mod[e](S^2,C,\DD))f\}
\end{equation}
be the subgroup of liftable twists. Denote by $\Lambda$ the biset of the
virtual homomorphism $T_{f,\CC}\colon \Mod[e](S^2,A,\CC) \cong
\Z^\CC\dashrightarrow \Z^\DD\cong \Mod[e](S^2,C,\DD))$ with
$\operatorname{Dom}(T_{f,\CC})= H_f$; the $\Z^\CC$-$\Z^\DD$-biset $\Lambda$
may be viewed as the abelian group
\begin{equation}\label{eq:defn:Lambda}
  \Lambda\coloneqq \Z^\DD\times \Z^\CC/\{(T_{f,\CC}(m), -m)\mid m\in H_f\}
\end{equation}
with natural actions. Since the action of $\Mod[e](S^2,C,\DD)\times
\Mod[e](S^2,A,\CC)$ on $M(f,C,A,\DD,\CC)$ commutes with the actions of
$\Mod[v](S^2,C,\DD)$ and of $\Mod[v](S^2,A,\CC)$, the group $\Lambda$ has a
well defined left action on $M(f,C,A,\DD,\CC)$ defined by $\lambda \cdot b=
m b n$ if $\lambda =[(m,n)]$ with $m\in\Z^\DD\cong\Mod[e](S^2,C,\DD)$ and
$n\in\Z^\CC\cong\Mod[e](S^2,A,\CC)$. We write the action of $\Lambda$ on
$M(f,C,A,\DD,\CC)$ as a left action because of the
\begin{lem}
  The action of $\Lambda\times \Mod[v](S^2,C,\DD)$ on
  $M(f,C,A,\DD,\CC)$ is free.
\end{lem}
\begin{proof}
  The group $\Lambda$ is defined in such a manner that its action on
  $M(f,C,A,\DD,\CC)$ is free. It follows from
  Proposition~\ref{prop:MCBedge} that the combined action of $\Lambda\times
  \Mod[v](S^2,C,\DD)$ on $M(f,C,A,\DD,\CC)$ is free.
\end{proof}

\begin{rem}
  It may happen that $T_{f,\CC}(e)\in\Z^\DD$ for some $e\in\Z^\CC$ that
  does not correspond to a liftable element, i.e.\ $e\notin H_f$. In that
  case, there is no reason for the equality $T_{f,\CC}(e)f=f e$ to
  hold. However, some positive power of $e$ is liftable: if $f$ have degree
  $d$, then $T_{f,\CC}(e^{d!})f=f e^{d!}$.

  It follows that the group $\Lambda$ need not be free Abelian: the element
  $(T_{f,\CC}(e),-e)$ has finite order in $\Lambda$ if
  $e\in\Mod[e](S^2,A,\CC)\setminus H_f$ and $T_{f,\CC}(e)\in\Z^\DD$.

  Examples of maps for which this phenomenon occurs include maps $f$ for
  which $T_{f,\CC}$ is integral but some curves in $\CC$ have lifts of
  degree $>1$. 
\end{rem}

Let
\begin{equation}\label{eq:defn:G_f}
  G_f\coloneqq \{m\in  \Mod[v](S^2,A,\CC) \mid f m\in(\Lambda\times \Mod[v](S^2,C,\DD))f\}
\end{equation}
be the group of liftable elements. Then for every $m\in G_f$ there is
a unique
$(\theta_f(m), \widetilde m)\in \Lambda\times\Mod[v](S^2,C,\DD)$ such
that $f m=(\theta_f(m), \widetilde m)f$ holds in
$M(f,C,A,\DD,\CC)$. The group $\Lambda$ as well as $\theta_f$ are
computable using Algorithm~\ref{algo:compute M(f,C,A)}; this is
essentially done in Proposition~\ref{prop:AbSubBisets}. We can also
easily compute $\theta_f$ on the set of Dehn twists in $G_f$ by
lifting the corresponding curves, see Example in~\S\ref{ss:complicated
  centralizer}.  It is easy to see that $G_f$ is a subgroup of
\[ \{m\in\Mod[v](S^2,A,\CC) \mid \sigma\circ \pi(f)m\in\Mod[v](S^2,C,\DD) \sigma\circ\pi(f)\}.\]

\begin{thm}\label{thm:M xproduct}
  Consider a sphere map $f\colon(S^2,C,\DD) \to (S^2,A,\CC)$, and the
  corresponding mapping class bisets $M(f,C,A,\DD, \CC)$. Let
  $vM_0(f)\coloneqq \Mod[v](S^2,C,\DD) \sigma\circ \pi(f) G_f$ be the
  $ \Mod[v](S^2,C,\DD)$-$G_f$-subbiset of $vM(f,C,A,\DD, \CC)$
  containing $f$; let $G_f\le \Mod[v](S^2,A,\CC)$ be the groups of
  liftable elements as in~\eqref{eq:defn:G_f}; and let $\Lambda$ be
  the $\Mod[e](S^2,C,\DD)$-$\Mod[e](S^2,A,\CC) \times G_f$-Abelian
  biset of
  $T_{f,\CC}\colon \Mod[e](S^2,A,\CC) \cong \Z^\CC\dashrightarrow
  \Z^\DD\cong \Mod[e](S^2,C,\DD))$ as above with the action of $G_f$
  given via $\theta_f$.

  \noindent Then $M(f,C,A,\DD,\CC)$ decomposes as
  \[\subscript{\Mod(S^2,C,\DD)}{(\Lambda\times vM_0(f))}_{ \Mod[e](S^2,A,\CC) \times G_f} \otimes \Mod(S^2,A,\CC)
   \] by the map sending $(m_1,g_1)f(m_2,g_2)$ to
   $([(m_1,m_2)],\sigma(\pi(g_1 f)))\otimes g_2$. The actions on
   $(\Lambda\times vM_0(f))$ are given by
   \[(m_1,g_1)\cdot(\lambda,g)\cdot(m_2,g_2)=((m_1\lambda
     m_2\theta_f(g_2),g_1 g g_2).\]
\end{thm}   
\begin{proof}
  View first $M(f,C,A,\DD,\CC)$ as a
  $(\Lambda\times\Mod[v](S^2,C,\DD))$-$\Mod[v](S^2,A,\CC)$-biset. Since
  $M(f,C,A,\DD,\CC)$ is left-free, we have
  \[M(f,C,A,\DD,\CC)\cong(\Lambda\times \Mod[v](S^2,C,\DD))f G_f \otimes_{G_f}
    \Mod[v](S^2,A,\CC).
  \]
  We then have a decomposition of
  $\Lambda \times \Mod[v](S^2,C,\DD)f G_f$ as
  \[\subscript{\Mod[e](S^2,C,\DD) \times
      \Mod[v](S^2,C,\DD)}(\Lambda\times \Mod[v](S^2,C,\DD)
    \pi(f)G_f)_{\Mod[e](S^2,A,\CC) \times G_f}
  \]
  with $\Mod[v](S^2,C,\DD) \pi(f)G_f$ viewed as a subbiset of
  $\Mod[e](S^2,C,\DD)\backslash M(f,C,A,\DD,\CC)/\Mod[e](S^2,A,\CC)$. The
  biset $\Mod[v](S^2,C,\DD) \pi(f)G_f$ is isomorphic to $vM_0(f)$.
\end{proof}

\subsection{Conjugacy and centralizers in mapping class bisets}
We are ready to apply Algorithm~\ref{algo:RedConjAndEqPr} to solving
centralizer and conjugacy problems in $M(f,A,\CC)$. We consider as
class $\Omega$ of ``computable'' groups all finite direct products
$G_1\times\cdots G_n$ of groups $G_i$ which are all finite-index
subgroups of modular groups $\Mod(S^2,A_i,\CC_i)$.

Note that $\Omega$ contains pure mapping class groups, cyclic groups
(as mapping class groups of a sphere with four marked points separated
by a curve) and $\SL_2(\Z)$ (as the mapping class group of a sphere
with four marked points). Note in particular that all groups in
$\Omega$ are finitely presented and have solvable word and conjugacy
problem.

\begin{prop}\label{prop:AbSubBisets}
  Consider the extension of bisets~\eqref{eq:Mp_Cl_biset_seq} and the natural morphism~\eqref{eq:forgetvM}. Then
  there is an algorithm, with oracle solving the conjugacy and
  centralizer problems in $vM(f,A,\CC)$, by which the conjugacy
  problem is solvable in $M(f,A,\CC)$ and the centralizer problem is
  sub-solvable in $M(f,A,\CC)$.
\end{prop}
\begin{proof}
  To apply Algorithm~\ref{algo:RedConjAndEqPr} we need to verify the
  following two conditions:
  \begin{itemize}
  \item given $b\in eM(f,A,\CC)$, the stabilizer
    \[P_b=\{(a_1,a_2)\in\Mod[e](S^2,A,\CC)\times \Mod[e](S^2,A,\CC)\mid a_1b=b a_2\}\]
    is finitely generated and computable,
  \item given $b_1,b_2\in eM(f,A,\CC)$, it is decidable whether there
    exist $a_1,a_2\in\Mod[e](S^2,A,\CC)$ with $a_1b_1=b_2a_2$.
  \end{itemize}

  For $b\in eM(f,A,\CC)$, consider the subgroup
  \[Q_b\coloneqq\{a_2\in\Mod[e](S^2,A,\CC)\mid b a_2\in\Mod(S^2,A,\CC)b\}.
  \]
  It has finite index in $\Mod[e](S^2,A,\CC)$ so is computable, since
  $M(f,A,\CC)$ has finitely many left orbits by
  Proposition~\ref{prop:MCBfinite}. Furthermore, it follows from
  $f^{-1}(\CC)\subset \CC$ that $Q_b$ has a well-defined image in
  $\Mod[e](S^2,A,\CC)$ under
  \[\phi_b\colon a_2\mapsto \text{the element $a_1\in\Mod[e](S^2,A,\CC)$ such that
  }a_1b=b a_2.
  \]
  Therefore, the stabilizer of $b$ has the following description:
  \[P_b=\{(\phi_b(a_2)^{-1},a_2)\in\Mod[e](S^2,A,\CC)\times Q_b\};\]
  this shows that $P_b$ is computable.

  Consider $b_1,b_2\in eM(f,A,\CC)$. It is decidable whether there
  exists $a_2\in\Mod[e](S^2,A,\CC)$ with $b_2a_2=g_1b_1$ for some
  $g_1\in\Mod(S^2,A,\CC)$. Then $b_1$ and $b_2$ are in the same
  connected component of $eM(f,A,\CC)$ if and only if
  $g_1\in\Mod[e](S^2,A,\CC)$; this last question is decidable.
\end{proof}

\subsection{The vertex mapping class biset of a Thurston map}
We study further the mapping class biset $vM(f,A,\CC)$, decomposing it
as a product.

\begin{defn}[Product bisets]
  Let $(G_i)_{i\in I}$ be a family of groups, let $G=\prod_i G_i$ be
  their product, let $f\colon I\selfmap$ be a map, and let
  $(B_i)_{i\in I}$ be left-free $G_i$-$G_{f(i)}$-bisets.  The
  \emph{product biset} of the bisets $B_i$ is the left-free
  $G$-$G$-biset
  \[\prod_i B_i=\{(b_i)_{i\in I}\mid b_i\in B_i\}
  \]
  with actions
  $h\cdot(b_i)_{i\in I}\cdot g=(h_i b_i g_{f(i)})_{i\in I}$.
\end{defn}

Such a product of bisets appears naturally out of the following
topological data. Let us denote by $(S_i)_{i\in I}$ the connected
components of the disconnection $(S^2,A)\setminus \CC$, and by
$\Mod(S_i)$ the pure mapping class group of $S_i$. Then the group
$\Mod[v](S^2,A,\CC)$ is the direct product $\prod_i\Mod(S_i)$.  For
every $b\in vM(f,A,\CC)$ let $b_i\colon \widehat{U_i}\to
\widehat{S_{f(i)}}$ be the small sphere map induced by $b$ as in
Definition~\ref{lem:SmallMaps} (note that all $b\in vM(f,A,\CC)$
induce the same map on $I$). Denote by $vM(f,A,\CC)_i$ the restriction
of $vM(f,A,\CC)$ to $S_i$:
\[vM(f,A,\CC)_i\coloneqq\subscript{\Mod(S_i)}\{b_i\mid b\in vM(f,A,\CC)\}_{\Mod(S_{f(i)})}.\]

\begin{lem}
  The $\Mod[v](S^2,A,\CC)$-biset $vM(f,A,\CC)$ is the product
  \[vM(f,A,\CC)=\prod_{i} vM(f,A,\CC)_i
  \]
  of the $\Mod(S_i)$-$\Mod(S_{f(i)})$-bisets
  $vM(f,A,\CC)_i$.\qed
\end{lem}

\begin{prop}\label{prop:bisets_gluing}
  Consider a periodic cycle
  \begin{equation}\label{eq:pr_cycle_in_I}
    \Pi=\{i,f(i),\dots, f^{(t)}(i)=i\}
  \end{equation}
  of $f\colon I\selfmap$, and assume that an oracle solves the
  conjugacy and centralizer problems in the $\Mod[v](S_i)$-biset
  \begin{equation}\label{eq:tens_pr_bis}
    vM(f,A,\CC)_i\otimes_{\Mod[v](S_{f(i)})}vM(f,A,\CC)_{f(i)}\otimes\cdots\otimes vM(f,A,\CC)_{f^{t-1}(i)}.
  \end{equation}
  Then the conjugacy and centralizer problems are solvable in the
  $\prod_{j\in\Pi}\Mod[v](S_{j})$-biset
  \begin{equation}\label{eq:pr_bis}
    vM(f,A,\CC)_\Pi\coloneqq\prod_{j\in\Pi}vM(f,A,\CC)_j.
  \end{equation}

  If for every periodic cycle as in~\eqref{eq:pr_cycle_in_I} the
  conjugacy and centralizer problems are solvable
  in~\eqref{eq:pr_bis}, then the conjugacy and centralizer problems
  are solvable in $vM(f,A,\CC)$.
\end{prop}

We will prove the proposition later, in the form of
Algorithm~\ref{algo:smalltobig}; we start by deriving some
consequences. We remark that the bisets defined
by~\eqref{eq:tens_pr_bis} are the first return maps of $f$ on the
system of small spheres $(S_i)_{i\in I}$. We have:
\begin{algo}\label{algo:bisets_decomp}
 \textsc{Given} two $g_1,g_2\in M(f,A,\CC)$ and an \textsc{oracle} solving conjugacy problems between $R(g_1,\CC)$ and $R(g_2,\CC)$ and computing the centralizers of small maps of $g_1$, \textsc{Solve} the conjugacy problem between $g_1,g_2$ and \textsc{Compute} the
  centralizer of $g_1$ as the kernel of a homomorphism from a
  finite-index subgroup of a product of mapping class groups towards a
  finitely generated abelian group
   \textsc{as follows:}\upshape
 \begin{enumerate}
 \item Using Algorithm~\ref{algo:compute M(f,C,A)} compute a recursion
   for $M(f,A,\CC)$. Applying the forgetful map, compute a recursion
   for $vM(f,A,\CC)$ and
   \[\Mod[e](S^2,A,\CC)\backslash M(f,A,\CC)/\Mod[e](S^2,A,\CC).\]
\item Using Algorithm~\ref{algo:smalltobig}, solve the conjugacy and
  centralizer problems for the images of $g_1$ and $g_2$ in
  $vM(f,A,\CC)$.
 \item Using Algorithm~\ref{algo:from vM to M:eMod}, solve the
   conjugacy and centralizer problems for the images of $g_1$ and
   $g_2$ in
   $\Mod[e](S^2,A,\CC)\backslash M(f,A,\CC)/\Mod[e](S^2,A,\CC)$.
 \item View the conjugacy and centralizer problems for
   $g_1,g_2\in M(f,A,\CC)$ as orbit and stabilizer problems in a
   $G$-set, using~\eqref{eq:GsetsVsBisets}; solve these problems using
   Algorithm~\ref{algo:RedConjAndEqPr}.
 \end{enumerate}
\end{algo}
\begin{proof}[Correctness of the algorithm]
  This follows directly from Propositions~\ref{prop:AbSubBisets}
  and~\ref{prop:bisets_gluing}.
\end{proof}

We recall again that $g_1,g_2\in M(f,A,\CC)$ are typically given in
terms of their bisets. By Theorem~\ref{thm:DecompOfBiset} the sphere
tree of bisets decomposing $B(g_1)$ and $B(g_2)$ relatively to $\CC$
are computable; so the projections of $g_1$ and $g_2$ in
$vM(f,A,\CC)$, see~\eqref{eq:vB as grBis}, and the bisets of small
Thurston maps of $g_1$ and $g_2$, see~\eqref{eq:def:SmallThBis}, are
computable.

Note that the centralizer of a homeomorphism is usually not very
complicated. For example, if $h\colon(S^2,A,\CC)\selfmap$ is a
homeomorphism, then its centralizer is, up to finite index, the direct
product of the centralizers of the small Thurston maps. The situation
is much more intricate for non-invertible maps,
see~\S\ref{ss:complicated centralizer}.

\begin{algo}\label{algo:smalltobig}
  \textsc{Given} a multicurve $\CC$ and a solution, for each periodic
  cycle of small spheres, of the conjugacy and centralizer problems in
  the first return biset,\\
  \textsc{Solve} the conjugacy and centralizer problems in $vM(f,A,\CC)$
  \textsc{as follows:}\upshape

  Consider two elements $b,c\in vM(f,A,\CC)$ for which we wish to know
  whether they are conjugate and for which we wish to compute $Z(b)$.
  \begin{enumerate}
  \item We use the notation $b_i$ for the projection of
    $b\in vM(f,A,\CC)$ to $vM(f,A,\CC)_i$, etc.

    Let $\Pi_1,\Pi_2,\dots,\Pi_n$ be all the periodic cycles of
    $f\colon I\selfmap$, viewed as subsets of $I$. To avoid
    hard-to-read subscripts, we introduce the notation
    $b[\Pi_1],\dots,b[\Pi_n]$ for the images of $b$ in
    $vM(f,A,\CC)_{\Pi_1},\dots,vM(f,A,\CC)_{\Pi_n}$ respectively, and
    define $c[\Pi_1],\dots,c[\Pi_n]$ similarly.
  \item If for some cycle $\Pi_j$ the element $b[\Pi_j]$ is not
    conjugate in $vM(f,A,\CC)_{\Pi_j}$ to $c[\Pi_j]$, then $b$ and $c$
    are not conjugate, and we have solved the conjugacy problem for
    $b,c$ in the negative.
  \item For every $X\subset I$ with $f(X)\subset X$, let us denote by
    \[\Mod_X\coloneqq\prod_{i\in X}\Mod(\widehat S_i)
    \]
    the pure mapping class group of $\bigsqcup_{i\in X}\widehat S_i$, by
    \[vM(f,A,\CC)_X\coloneqq\prod_{i\in X}vM(f,A,\CC)_i
    \]
    the mapping class biset over $\bigsqcup_{i\in X}\widehat S_i$, and by
    $b[X]$ and $c[X]$ the images of $b$ and $c$ in $vM(f,A,\CC)_X$
    respectively. We denote also by $Z(b[X])$ the centralizer of
    $b[X]$ in $vM(f,A,\CC)_X$, and by $h[X]$ an element of $\Mod_X$
    that conjugates $b[X]$ to $c[X]$, namely satisfies
    $b[X]=c[X]^{h[X]}$; or, if $b[X],c[X]$ are not conjugate, we write
    $h[X]=\texttt{fail}$.
  \item Let $X_0=\bigsqcup_j\Pi_j$ be the union of the cycles in $I$.
    From the assumptions of the proposition, $h[X_0]$ and $Z(b[X_0])$
    are computable, and we show now how to compute the variables
    $h[X]$ and $Z(b[X])$ for all $X\subseteq I$.

    After all $h[X]$ and $Z(b[X])$ have been computed, we will in
    particular know whether $b$ and $c$ are conjugate: if
    $h[I]\neq\texttt{fail}$, then it conjugates $b$ into $c$. The
    centralizer of $b$ in $Z(b[I])$.
  \item Assume that $h[X]$ and $Z(b[X])$ have been computed, and
    consider $i\in I\setminus X$ with $f(i)\in X$. This is how to
    compute $h[X\cup\{i\}]$ and $Z(b[X\cup\{i\}])$.

    If $h[X]=\texttt{fail}$ then set
    $h[X\cup\{i\}]\coloneqq\texttt{fail}$. Otherwise, compute the
    action of $Z(b[X])[f(i)]\le\Mod(\widehat S_{f(i)})$ on
    $1\otimes c_i h[f(i)]\in 1\otimes vM(f,A,\CC)_i$. Since $Z(b[X])$
    is computable and in particular finitely generated and it acts on
    a finite set, the orbit of $1\otimes c_i h[f(i)]$ is
    computable. If this orbit does not contain $1\otimes b_i$, then
    set $h[X\cup\{i\}]\coloneqq\texttt{fail}$. Otherwise, let
    $m\in Z(b[f(i)])$ and $h[i]\in\Mod(\widehat S_i)$ be such that
    $h[i]b_i=c_i h[f(i)]m[f(i)]$; by replacing $h[X]$ with $h[X]m[X]$
    and defining $h[i]$ as above, we extend $h[X]$ to $h[X\cup\{i\}]$.

    Again we use the fact that the action of $Z(b[X])$ on the finite
    set $1\otimes vM(f,A,\CC)_i$ is computable; let $H$ be the
    finite-index subgroup of $Z(b[X])$ that stabilizes $1\otimes
    b_i$.
    Furthermore, let $\phi\colon H\to\Mod(\widehat S_i)$ satisfy
    $a^\phi b_i=b_i a[f(i)]$; it is well-defined by
    Proposition~\ref{prop:MCBfree}. Let then $Z(b[X\cup\{i\}])$ be the
    diagonal image $a\mapsto(a,a^\phi)$ of $H$ in
    $Z(b[X])\cup\Mod(\widehat S_i)\le \Mod_{X\cup\{i\}}$.
  \end{enumerate}
\end{algo}

\begin{proof}[Proof of Proposition~\ref{prop:bisets_gluing}]
  The second part is proven in Algorithm~\ref{algo:smalltobig}.

  For the first part, consider a new set
  $\Pi'=\{i,f(i),\dots,f^{(t-1)}(i),f^{(t)}(i)=i'=f(i')\}$. Set $S_{i'}\coloneqq S_i$ and extend $f$ to $S_{i'}$ as
  $f^{(t)}$. Then the biset $vM(f,A,\CC)_{i'}$ is given
  by~\eqref{eq:tens_pr_bis}, which by assumption has solvable
  conjugacy and centralizer problems. Since this is the only cycle of
  $\Pi'$, we may apply the second part and deduce that the conjugacy
  and centralizer problems are solvable for $\Pi'$.

  Consider now two elements $b,c\in vM(f,A,\CC)_\Pi$ for which we wish
  to know whether they are conjugate and for which we wish to compute
  $Z(b)$.

  Write $b=(b_i,b_{f(i)},\dots,b_{f^{(t-1)}(i)})$ and
  $c=(c_i,c_{f(i)},\dots,c_{f^{(t-1)}(i)})$. Extend them to elements
  $b',c'\in vM(f,A,\CC)_{\Pi'}$ by setting
  $b_{i'}=b_i\otimes\cdots\otimes b_{f^{(t-1)}(i)}$ and
  $c_{i'}=c_i\otimes\cdots\otimes c_{f^{(t-1)}(i)}$.

  Then $b,c$ are conjugate if and only if $b',c'$ are conjugate in
  $vM(f,A,\CC)_{\Pi'}$. If $h'=(h_i,h_{f(i)},\dots,h_{i'})$ conjugates
  $b'$ to $c'$, then $h=(h_i,\dots,h_{f^{(t-1)}(i)})$ conjugates $b$
  to $c$. The centralizer $Z(b)$ coincides with the centralizer
  $Z(b')$. Therefore, both problems are solvable.
\end{proof}

\section{Decidability of Thurston equivalence}
\label{s:CombEquiv}
We prove our main result in this section, as a consequence of
Algorithm~\ref{algo:bisets_decomp} and results
from~\cites{bartholdi-dudko:bc3,bartholdi-dudko:bc4}:
\begin{thm}\label{thm:A:new}
  It is decidable whether or not two Thurston maps are combinatorially
  equivalent.

  Furthermore, the centralizer of a Thurston map $f$ (i.e.\ the set of
  homeomorphisms that commute with $f$ up to isotopy) is computable.
\end{thm}

For a sphere map $f\colon (S^2,A)\selfmap$, denote by
$A^\infty\subseteq A$ the forward orbit of the periodic critical
points of $f$. The map $f\colon (S^2,A)\selfmap $ with $\deg f>2$ is
\emph{geometric} if $f$ is either
\begin{itemize}
\item[\Exp] \emph{B\"ottcher expanding}: there is a metric on
  $S^2\setminus A^\infty$ that is expanded by $f$, and such that at
  all $a\in A^\infty$ the first return map of $f$ is locally conjugate
  to $z\mapsto z^{\deg_a(f^n)}$; or
\item[\Tor] a quotient of a torus endomorphism
  $z\mapsto M z+q\colon\R^2/\Z^2\selfmap$ by the involution
  $z\mapsto-z$, for a $2\times2$ matrix $M$ whose eigenvalues are
  different from $\pm 1$.
\end{itemize}

For a sphere map $f\colon (S^2,A)\selfmap$, denote by $P_f$ its
postcritical set and denote by $f\colon (S^2,P_f,\ord_f)\selfmap $ the
minimal orbisphere map induced by $f$.  Let $X$ be a basis of $M(f)$
and let $N$ be a finite generating set of $\Mod(S^2,A)$; so
$M(f)=\bigcup_{n\ge0} N^n X$. We call an algorithm with input in
$M(f)\times M(f)$ \emph{efficient} if for $f,g\in N^n X$ the running
time of the algorithm is bounded by a polynomial in $n$.

\begin{thm}[\cite{bartholdi-dudko:bc3}*{Corollary~\ref{bc3:cor:conjZpb} and Theorem~\ref{bc3:thm:RedConjCentrProb}}]
\label{thm:removing extra marked pnts}
 There is an efficient algorithm with oracle that, given two
  orbisphere maps $f,g\colon (S^2,A)\selfmap$ by their bisets and such that $f$ is
  geometric, decides whether $f,g$ are conjugate, and computes the
  centralizer of $f$.

  The oracle must answer, given two geometric orbisphere maps $f,g$ on
  their minimal orbisphere $(S^2,P_f,\ord_f)$ respectively
  $(S^2,P_g,\ord_g)$, whether they are conjugate and what the
  centralizer of $f$ is.
  
  The centralizer of $f\colon (S^2,A)\selfmap$ is a isomorphic to a finite-index subgroup of the centralizer of  $f\colon (S^2,P_f,\ord_f)\selfmap$. 
\end{thm}

An obstruction to being geometric is the existence of a
\emph{Levy cycle}. This is an essential simple closed curve on
$S^2\setminus A$ that is isotopic to some iterated preimage of
itself. This is the only obstruction:
\begin{thm}
  Suppose $f\colon (S^2,A)\selfmap $ is a Thurston map with degree at
  least $2$ such that $f$ admits no Levy obstruction. Then either $f$
  is isotopic to a map expanding a metric on $S^2\setminus A^\infty$, or $f$
  is isotopic to the quotient by the involution $z\mapsto -z$ of an
  affine map on $\R^2/\Z^2$ whose eigenvalues are different from
  $\pm1$.
\end{thm} 
\noindent The torus case is proven
in~\cite{selinger-yampolsky:geometrization}*{Main Theorem
  \MakeUppercase{\romannumeral 2}}, and the non-torus case is proven
in~\cite{bartholdi-dudko:bc4}*{Theorem~\ref{bc4:thm:main}}.

A sphere map $f\colon (S^2,A)\selfmap$ induces a pullback map
$\sigma_f$ on the Teichm\"uller space $\mathscr T_A$ of complex
structures on $(S^2,A)$,
see~\S\ref{ss:examples}. Pilgrim's~\cite{pilgrim:combinations}
\emph{canonical obstruction} $\CC_f$ consists of all essential simple
closed curves $\gamma$ such that for some (and equivalently all)
$\eta\in \mathscr T_A$ the length of $\gamma$ with respect to
$(\sigma_f)^n\eta$ shrinks to $0$ as $n\to +\infty$.

\begin{algo}\label{algo:ratmaps}
  \textsc{Given} two sphere bisets $\subscript{ G }B_G$ and $\subscript{ G }B'_G$,\\
  \textsc{Decide} whether $B$ and $B'$ are bisets of rational non-\Tor~maps;\\
  in that case, \textsc{Decide} whether $B$ and $B'$ are conjugate by $\Mod(G)$, and if so \textsc{Compute} a conjugator \textsc{ as follows:}\\
  \begin{enumerate}
  \item
    Using~\cite{bartholdi-dudko:bc4}*{Algorithms~\ref{bc4:algo:is2cover}},
    check whether $B$ is the biset of a map double covered by a torus
    endomorphism. If so, then $B$ is not the biset of a non-\Tor\
    rational map.
  \item Compute a complete list $\mathscr R$ of conjugacy classes of
    $G$-$G$-bisets of non-\Tor\ rational maps of degree $\deg(B)$ as
    follows:
    \begin{enumerate}
    \item Choose arbitrary complex co\"ordinates for the position of a
      marked set $C$ of $\rank (G)$ points in $\hC$. Every rational
      map $g\colon(\hC,C)\selfmap$ with $\deg(g)=\deg(B)$ is given as
      a quotient of two polynomials of degree $\deg(B)$, namely by
      $2\deg(B)+1$ complex co\"efficients.
    \item These co\"efficients satisfy a system of algebraic
      equations, and by Thurston's rigidity theorem this system has
      finitely many non-\Tor\ solutions.
    \item Every non-\Tor\ solution can then be approximated to any
      desired floating-point precision. This only requires standard
      algorithms for root finding and root isolation.
    \item By homotopy lifting of paths in $\hC$ away from $C$, the
      biset of each solution $g$ may be computed. Create the list
      $\mathscr R$ containing bisets $B(g)$ of all the rational
      non-\Tor\ maps $g\colon (\hC,C)\selfmap$.
    \end{enumerate}
  \item Perform in parallel the following steps:
    \begin{enumerate}
    \item Enumerate all multicurves on the sphere $(S^2,A)$. For every
      multicurve $\CC$, check whether $\CC$ is an obstruction for
      $B$. If so, $B$ is not equivalent to a rational map.
      \label{st:7:algo:ratmaps}
    \item Enumerate isomorphisms of sphere groups
      $G\to\pi_1(\hC,C)$. For every isomorphism $m$ and every
      $C\in \mathscr R$, check whether $m$ conjugates $B$ into $C$,
      namely $B^m\cong C$. If so, $B$ is the biset of a rational map.
      \label{st:8:algo:ratmaps}
    \end{enumerate}
  \item Repeat the previous steps for $B'$.
  \item The bisets $B$ and $B'$ are conjugate if and only if $B$ and
    $B'$ are conjugate to the same biset $C\in \mathscr R$. If $m$ and
    $m'$ are the conjugators found in Step~\eqref{st:8:algo:ratmaps},
    then $B^m\cong C$ and $(B')^{m'}\cong C$ so $m(m')^{-1}$ is a
    conjugator from $B$ to $B'$.
  \end{enumerate}
\end{algo}
We remark that this is the only place, in our text, where
floating-point calculations are invoked. In terms of performance, it
is faster to run Thurston's algorithm instead of computing the list
$\mathscr R$; see~\cite{bonnot-braverman-yampolsky:thurstondecidable}.

The computability of the \emph{canonical decomposition} (namely the
decomposition with respect to $\CC_f$) of a sphere map was shown
in~\cite{selinger-yampolsky:geometrization}*{Main Theorem
  \MakeUppercase{\romannumeral 1}}. If a map $f$ is Levy free
non-torus, then $\CC_f$ is a unique minimal obstruction such that all
maps in $R(f,\CC_f)$ are rational
\cite{bartholdi-dudko:bc4}*{Lemma~\ref{bc4:lem:ExpThmFirstObserv}}. We
have the following algorithm:

\begin{algo}\label{algo:canonical}
  \textsc{Given} a Levy-free non-\Tor\ sphere biset $\subscript G B_G$,\\
  \textsc{Compute} the canonical decomposition $\mathfrak B$ of $B$ \textsc{ as follows:}\\
  \begin{enumerate}
  \item Enumerate in increasing order all multicurves $\CC$ on the
    sphere $(S^2,A)$; for every multicurve $\CC$ perform the following
    steps:
  \item If $\CC$ is not fully $B$-invariant, discard it;
  \item Using Theorem~\ref{thm:mc to gog}, compute the decomposition
    $\mathfrak B$ of $B$ with respect to $\CC$;
  \item Using Algorithm~\ref{algo:ratmaps}, check whether all the
    bisets in $R(B,\CC)$ are the bisets of rational non-\Tor~maps. If
    so, return $\mathfrak B$.\label{st:5:algo:canonical}
  \end{enumerate}
\end{algo}

The following algorithm solves the conjugacy problem for Levy free
non-torus maps:
\begin{algo}\label{algo:DecidExp}
  \textsc{Given} two geometric sphere bisets $\subscript {G}B_{G}$ and  $\subscript {H}C_{H}$ that are not \Tor,\\
  \textsc{Decide} whether $B$ and $C$ are conjugate, and \textsc{compute} the centralizer $Z(B)$ \textsc{as
    follows:}\\\upshape
 \begin{enumerate}
 \item Using Theorem~\ref{thm:removing extra marked pnts}, reduce the problem to the case when the peripheral conjugacy classes in $B$ and $C$ correspond to the postcritical set.
 \item Using Algorithm~\ref{algo:canonical}, compute the canonical 
   decompositions $\subscript \gf\gfB_\gf$ and
   $\subscript \gfY\gfC_\gfY$ of $B$ and $C$
   respectively.
 \item Let $X$ and $Y$ be the set of distinguished conjugacy classes
   of $\gf$ and $\gfY$ respectively, see~\S\ref{ss:distinguished
     cc}. Enumerate all possible bijections $h\colon X\to Y$.
 \item For every $h\colon X\to Y$ try to do the following
   steps. Return \texttt{fail} if no $h$ is successful.
 \item Using Algorithm~\ref{algo:MarkCl:Prom} try to promote
   $h\colon X\to Y$ into a biprincipal sphere $\gf$-$\gfY$-tree of
   bisets $\gfI$. Discard $h$ if there is no promotion.
 \item Check by Algorithm~\ref{algo:chech:InMCB} whether
   $\gfC^\gfI \in M(\gfB)$; if not, discard
   $h$.
 \item For every $\gfI$ check using Algorithm~\ref{algo:ratmaps}
   whether the bisets in $R(\gfB)$ are conjugate to the associated
   bisets in $R(\gfC^\gfI)$. If not, discard $h$.
 \item For every $\gfI$ using Algorithm~\ref{algo:bisets_decomp}
   check whether the conjugacies between $R(\gfB)$ and
   $R(\gfC^\gfI)$ promote into a conjugacy between
   $\gfB$ and $\gfC$. 
 \item Write $\gfB=B(f\colon (S^2,A,\CC)\selfmap)$. Compute the
   centralizer $Z(B)$ of $B$: since all bisets in $R(\gfB)$ have
   trivial centralizers the centralizer $Z(B)$ consists of Dehn twists
   $\gfI\in \Mod[e](S^2,A,\CC)$ commuting with $\gfB$; this is
   computable by Algorithm~\ref{algo:compute M(f,C,A)}.
 \end{enumerate}
\end{algo}

\begin{algo}\label{algo:conj:LevyFree}
  \textsc{Given} two sphere bisets $\subscript {G}B_{G}$ and  $\subscript {H}C_{H}$ that are either degree-$1$ or Levy-free,\\
  \textsc{Decide} whether $B$ and $C$ are conjugate, and \textsc{compute} the centralizer $Z(B)$ as a finitely generated subgroup of a product of pure mapping class groups \textsc{as follows:}\\\upshape
  \begin{enumerate}
  \item If $B,C$ are degree-$1$ then apply Algorithm~\ref{algo:conjmcg};
  \item If $B,C$ are geometric and not \Tor\ then apply Algorithm~\ref{algo:DecidExp}.
  \item If $B,C$ are \Tor\ bisets, then using
    Theorem~\ref{thm:removing extra marked pnts} reduce the problem to
    the case when $B$ and $C$ are minimal, and apply
    \cite{bartholdi-dudko:bc3}*{Algorithm~\ref{bc3:algo:minimaltor}}.
  \end{enumerate}
\end{algo}

The \emph{canonical Levy obstruction} $\CC_{\text{Levy}}$ of a
Thurston map $f\colon(S^2,A)\selfmap$ is the minimal $f$-invariant
multicurve all of whose small Thurston maps are either homeomorphisms
or admit no Levy cycle. It is unique by
\cite{bartholdi-dudko:bc4}*{Proposition~\ref{bc4:prop:HypLevyMCurvInter}}. The
\emph{Levy decomposition} of $f$ (and equivalently of its biset) is
its decomposition (as a graph of bisets) along the canonical Levy
obstruction.

\begin{algo}[\cite{bartholdi-dudko:bc4}*{Algorithm~\ref{bc4:algo:levy}}]\label{algo:decidelevy}
  \textsc{Given} a Thurston map $f\colon(S^2,A)\selfmap$ by its biset,\\
  \textsc{Compute} the Levy decomposition of $f$.
\end{algo}

\noindent We are ready to prove Theorem~\ref{thm:A:new}:
\begin{algo}\label{algo:Decid}
  \textsc{Given} two sphere bisets $\subscript G B_G$ and  $\subscript H C_H$,\\
  \textsc{Decide} whether $B$ and $C$ are conjugate, and \textsc{compute} the centralizer $Z(B)$ \textsc{as follows:}\\\upshape
  \begin{enumerate}
  \item Using Algorithm~\ref{algo:decidelevy}, compute the Levy
    decompositions $\subscript \gf\gfB_\gf$ and
    $\subscript \gfY\gfC_\gfY$ of $B$ and $C$
    respectively.
  \item Let $X$ and $Y$ be the set of distinguished conjugacy classes
    of $\gf$ and $\gfY$ respectively, see~\S\ref{ss:distinguished
      cc}. Enumerate all possible bijections $h\colon X\to Y$.
  \item For every $h\colon X\to Y$ try to do the following
    steps. Return \texttt{fail} if there is no success.
  \item Using Algorithm~\ref{algo:MarkCl:Prom} try to promote
    $h\colon X\to Y$ into a biprincipal sphere $\gf$-$\gfY$-tree of
    bisets $\gfI$. Discard $h$ if there is no promotion.
  \item Check by Algorithm~\ref{algo:chech:InMCB} whether
    $\gfC^\gfI \in M(\gfB)$; if not, discard
    $\gfI$.
  \item For every $\gfI$ check, using
    Algorithm~\ref{algo:conj:LevyFree}, whether bisets in $R(\gfB)$
    and $R(\gfC^\gfI)$ are conjugate; if not discard
    $\gfI$.
  \item Using Algorithm~\ref{algo:conj:LevyFree} compute the centralizer
    of bisets in $R(\gfB)$.
  \item For every $\gfI$ using Algorithm~\ref{algo:bisets_decomp}
    check whether the conjugacies between $R(\gfB)$ and
    $R(\gfC^\gfI)$ promote into a conjugacy between
    $\gfB$ and $\gfC$.
  \item Using Algorithm~\ref{algo:bisets_decomp} compute the centralizer
    $Z(B)$ of $B$.
  \end{enumerate}
\end{algo}

\subsection{Remarks on the Thurston iteration and its symbolic version}\label{ss:rem:Thurston iter}
Recall that every sphere map $f\colon (S^2,C)\to(S^2,A)$ induces a
pullback map $\sigma_f\colon\mathscr T_A\to\mathscr T_C$ between the
Teichm\"uller spaces of complex structures on the respective spheres,
see~\S\ref{ss:examples}. Selinger~\cite{selinger:augts} shows that
$\sigma_f$ extends continuously to a self-map of the augmented
Teichmüller $\widehat{\mathscr T}_A$. Suppose for simplicity that $f$
is not double covered by a torus endomorphism. If the map $f$ is
obstructed, for any $\eta\in \mathscr T_A$ the sequence
$(\sigma_f)^n(\eta)$ converges to the stratum of
$\partial\mathscr T_A$ associated with Pilgrim's canonical obstruction
$\CC_f$.  This yields a method of computing $\CC_f$.

The pullback map $\sigma_f\colon \mathscr T_A\selfmap$ projects to a
modular correspondence
\[
\mathscr M_A
\overset{\overline{\sigma_f}}\longleftarrow
\mathscr W_f \overset i\longrightarrow  \mathscr M_A
\]
of finite degree, see~\eqref{eq:thurston correspondence}. The
attracting fixed point or stratum of $\sigma_f$ projects to a
repelling fixed point or stratum of
$i\circ (\overline{\sigma_f})^{-1}$. By Proposition~\ref{prop:modular
  correspondence} we have $M(f)\cong B(\overline{\sigma_f},i)$; fixing
a basepoint $*\in\mathscr M_A$, every every $g\in M(f)$ can be viewed
as an element of $B((\overline{\sigma_f},i),*)$, namely as a path
$b_g=b_{g,0}\colon [0,1]\to \mathscr M_A$ with $b_g(0)=*$ and
$b_g(1)=\overline{\sigma_f}(x)$ for some $x\in i^{-1}(*)$. We can lift
$b_g$ to $\mathscr W_f$ and then project the lift back to
$\mathscr M_A$ to construct a path $b_{g,1}$ starting at the endpoint
of $b_{g,0}$. Continuing the process, we construct the sequence of
paths $b_{g,n}$ converging to the attracting fixed point or stratum
associated with $g$. In this manner we have implemented the Thurston
iteration at the level of the modular correspondence. Furthermore,
this process can be performed symbolically as follows.

Choose a basis $X$ for $M(f)$, so every $g\in M(f)$ can be uniquely
written as $m x$ for some $m=m_{g,0}\in \Mod(S^2,A)$ and
$x=x_{g,0}\in X$. Then $g=m x$ is conjugate to $x m$, which can
uniquely be rewritten as $m_{g,1} x_{g,1}$. We obtain the
\emph{symbolic pullback iteration}:
\begin{equation}\label{eq:SymbIterat}
  g=m x \sim x_{g,0}m_{g,0} = m_{g,1}x_{g,1}\sim x_{g,1}m_{g,1} = m_{g,2}x_{g,2}\sim x_{g,2}m_{g,2} = m_{g,3} x_{g,3}\sim\cdots
\end{equation}
and we are interested in the attractor of this process, namely in the
minimal subset of $M(f)$ to which $m_{g,n}x_{g,n}$ gets attracted. We
remark that the issue of convergence of $m_{g,n}x_{g,n}$ is more
delicate than the convergence of $b_{g,n}$ because at each step there
is an ``additive correction'' depending on the choice of $X$, which
may dominate the contraction under lifting.

For a sphere map $g\colon (S^2,A)\selfmap$, let $\CC_{\text{Levy}}$
be its canonical Levy obstruction (recall that it is the minimal
$g$-invariant multicurve all of whose small Thurston maps are either
homeomorphisms or admit no Levy cycle). If a small sphere $S'$ of
$(S^2,A,\CC_{\text{Levy}})$ is periodic and belongs to a $\deg >1$
cycle of small sphere maps, then let $H(S')$ be the trivial subgroup
of $\Mod(S')$; otherwise set $H(S')\coloneqq \Mod(S')$. In all cases,
view $H(S')$ as a subgroup of $\Mod(S^2,A)$. Denote by
$\Z^{\CC_{\text{Levy}}}$ the subgroup of $\Mod(S^2,A)$ generated by
Dehn twists along curves in $\CC_{\text{Levy}}$. Set then
\[H_g\coloneqq \Z^{\CC_{\text{Levy}}} \times \prod_{S'\in (S^2,A,\CC_{\text{Levy}})} H(S'),
\]
where the product is taken over all small spheres of
$(S^2,A,\CC_{\text{Levy}})$. Observe that all maps in $H_g g$ have
the same Levy decomposition.

Let us choose a finite generating set for $\Mod(S^2,A)$; it induces
the word metric ``$|\cdot|$'' on $\Mod(S^2,A)$. Let us also choose a
basis $X$ of $M(f)$, and for $g = m x\in M(f)$ let us set $|g|=|m|$.

\begin{conj}\label{conj:GenralNucl}
  There is a finite set $\mathscr N\subseteq M(f)$ such that, for
  every $g\in M(f)$, the symbolic pullback
  iteration~\eqref{eq:SymbIterat} run on $g$ converges to
  $\bigcup_{h\in\mathscr N}H_h h$ in $\mathcal O(\log |g|)$ steps.
\end{conj}
If Conjecture~\ref{conj:GenralNucl} is true, then the symbolic
pullback iteration~\eqref{eq:SymbIterat} allows to replace $g$ with a
new $\tilde g \in H_h h$ for some $h\in\mathscr N$, in such a manner
that $|\tilde g|$ is bounded by a polynomial in $|g|$. In this manner,
the conjugacy and centralizer problems in $M(f)$ would be reduced in
polynomial time to the conjugacy and centralizer problems in a modular
group $H_h$. These problems are expected to be solvable in polynomial
time; for recent progress see~\cite{bell-webb:poly,calvez:NT}. A
consequence of Conjecture~\ref{conj:GenralNucl} and such results would
be the
\begin{conj}\label{conj:p}
  For every Thurston map $f$, the conjugacy and centralizer problems
  are solvable in $M(f)$ in polynomial time.
\end{conj}

\section{Orbispheres}\label{ss:orbispheres}
All the previous considerations, on marked spheres, apply equally well
to a slightly more general situation, that of
\emph{orbispheres}. Consider a marked sphere $(S^2,A)$, and let there
also be given a function $\ord\colon A\to\{2,3,\dots,\infty\}$,
assigning a positive or infinite degree to each marked point. This
describes an \emph{orbispace} structure: if $\ord(a)=\infty$, then the
sphere is punctured at $a\in A$, while if $\ord(a)=n$ then the space
has a cone-type singularity of angle $2\pi/n$ at $a$.  This can be
thought of as specifying a complex structure on $S^2$: a small enough
neighbourhood of a point $a\in A$ with $\ord(a)=\infty$ is identified
with a neighbourhood of $0$ in $\C^*$; while if $\ord(a)\in\N$, a
neighbourhood of $a$ is modelled on a neighbourhood of $0$ in
$\C/\langle e^{2\pi i/\ord(a)}\rangle$. The parabolic class $\Gamma_i$
consists in elements of order $\ord(\Gamma_i)\coloneqq\ord(a_i)$,
see~\eqref{eq:orbispheregp}.  For more details,
consult~\cite{mcmullen:renormalization}*{Appendix~A}.

Orbispheres are denoted $(S^2,A,\ord)$, or just $(S^2,A)$ if the
degree function is implicit. To avoid degenerate cases, assume
$\#A\neq1$; and if $\#A=2$, assume furthermore that $\ord$ is constant
on $A$. The Euler characteristic of $(S^2,A,\ord)$ is defined as
\[\chi(S^2,A,\ord)=2-\sum_{a\in A}\left(1-\frac1{\ord(a)}\right).\]
Accordingly, $(S^2,A,\ord)$ is called \emph{spherical},
\emph{euclidean} or \emph{hyperbolic} if its Euler characteristic is
$>0$, $=0$ or $<0$. If we endow $S^2$ with a complex
structure, then the \emph{universal cover} $\widetilde{(S^2,A,\ord)}$
of $(S^2,A,\ord)$ is respectively $\hC$, $\C$ or $\mathbb D(0,1)$ with
its usual complex
structure~\cite{mcmullen:renormalization}*{Theorem~A.2}; in the sense
that $(S^2,A,\ord)=\widetilde{(S^2,A,\ord)}/G$ for a group $G$ of
isometries acting with finite stabilizers. Those points with
non-trivial stabilizers project then precisely to the marked points
$A$; because of the ambient complex structure, these stabilizers are
perforce cyclic groups, and the order of $G_{\widetilde a}$ is
$\ord(a)$.

The same loops $\gamma_i$ as before generate $G$; but now,
$\gamma_i^{\ord(a_i)}=1$ whenever $\ord(a_i)<\infty$. We have
\begin{equation}\label{eq:orbispheregp}
  G=\pi_1(S^2,A,\ord,*)=\langle \gamma_1,\dots,\gamma_n\mid \gamma_1^{\ord(a_1)},\dots,\gamma_n^{\ord(a_n)}, \gamma_1\cdots\gamma_n\rangle.
\end{equation}
Indeed $G$ is a group of isometries of $\widetilde{(S^2,A,\ord)}$, with
fundamental domain a $\#A$-sided polygon with angles $2\pi/\ord(a_i)$;
so the presentation~\eqref{eq:orbispheregp} follows from Poincar\'e's
Theorem, see e.g.~\cite{harpe:ggt}*{Theorem~V.B.40}.

\begin{defn}[Orbisphere groups]
  An \emph{orbisphere group} is a tuple $(G,\Gamma_1,\dots,\Gamma_n)$
  consisting of a group and $n\neq 1$ conjugacy classes $\Gamma_i$ in
  $G$, such that $G$ admits a presentation as
  in~\eqref{eq:orbispheregp} for some choice of
  $\gamma_i\in\Gamma_i$. The \emph{Euler characteristic} of $G$ is
  \[\chi(G)=2-\sum_{i=1}^n\left(1-\frac1{\ord(a_i)}\right).\] 
  
  Suppose that $G_1=\pi_1(S^2,A,\ord_1,*)$ and
  $G_2=\pi_1(S^2,A,\ord_2,*)$ are two orbisphere groups and suppose
  that $\ord_2(a)\mid \ord_1(a)$ for all $a\in A$. Then the natural
  homomorphism $G_1\to G_2$ is called an \emph{inessential forgetful
    map}.
\end{defn}

It is convenient to extend `$\ord$' to $S^2$ so that
$\ord(p)=1\Leftrightarrow p\notin A$.  Note again that marked spheres
are subsumed in the definition of orbispheres; namely, as those for
which $\ord(a)=\infty$ for all $a\in A$.

\subsection{Pure homeomorphisms between orbispheres}
\begin{defn}[Orbisphere maps]
  An \emph{orbisphere map} $f\colon(S^2,C,\ord_C)\to(S^2,A,\ord_A)$
  between orbispheres is a branched covering between the underlying
  spheres, with $f(C)\cup\{\text{critical values of }f\}\subseteq A$,
  that is locally modelled at $p\in S^2$ in oriented complex charts by
  $z\mapsto z^{\deg_p(f)}$ for some integer $\deg_p(f)\ge1$, so that
  $\ord_C(p)\deg_p(f)\mid\ord_A(f(p))$ for all $p\in S^2$. A
  \emph{covering} between orbispheres is an orbisphere map $f$ for
  which one has $\ord_C(p)\deg_p(f)=\ord_A(f(p))$ for all $p\in S^2$.
  
  An \emph{anti-orbisphere map}
  $f\colon(S^2,C,\ord_C)\to(S^2,A,\ord_A)$ between orbispheres is a
  branched covering satisfying the same condition as above except that
  $f$ is locally modelled in oriented charts at $p\in S^2$ by
  $z\mapsto {\bar z}^{\deg_p(f)}$ for some integer $\deg_p(f)\ge1$.
\end{defn}
In particular, if $f$ is a covering then
$A=f(C\cup\{\text{critical points of }f\})$.

For an orbisphere $(S^2,A,\ord)$ we define the \emph{modular group}
consisting of orbisphere and anti-orbisphere pure homeomorphisms as
\[\Mod^{\pm}(S^2,A,\ord) =\{ h\colon (S^2,A,\ord)\selfmap\mid \deg(h)=1, h|_{A}=\one\},\]
and we denote by $\Mod(S^2,A,\ord)$ the subgroup of order-preserving
homeomorphisms in $\Mod^{\pm}(S^2,A,\ord)$.  Clearly,
$\Mod(S^2,A,\ord)$ and $ \Mod(S^2,A)$ are isomorphic.

Choose $*\in S^2\setminus A$, write $G=\pi_1(S^2,A,\ord,*)$, and set
\[\Mod^{\pm}(G)=\{\phi\in\Out(G)\mid \Gamma_i^\phi=\Gamma^{\pm1 }_i\;\forall i=1,\dots,n\}.
\]
We have the following generalization of
Theorem~\ref{thm:dehn-nielsen-baer}.

\begin{thm}[Dehn-Nielsen-Baer-Zieschang-Vogt-Coldewey]
  \label{thm:dehn-nielsen-baer-zieschang}
  Let $(S^2,A,\ord)$ be an orbisphere with non-positive Euler
  characteristic. Then the natural map
  $\Mod^{\pm}(S^2,A,\ord)\to \Mod^{\pm}(\pi_1(S^2,A,\ord))$ is an
  isomorphism.
\end{thm}
\begin{proof} 
  This is essentially a direct consequence of
  \cite{zieschang-vogt-coldewey:spdg}*{Theorems~5.8.3
    and~5.14.1}. Since these theorems deal with groups that have only
  ``finite order peripheral classes'' (such a set is always preserved
  by automorphisms), we need to make slight adjustments.

  Let us a assume $\#A\ge 4$; if not, $\Mod^{\pm }(S^2,A,\ord)$ is a
  group of order two and the claim is easy to verify.

  Define a new map $\ord'\colon A\to\{2,3,\dots\}$ by
  \[\ord'(a)=\begin{cases} \ord(a)&\text{if } \ord(a)<\infty,\\
      2016 & \text{if } \ord(a)=\infty.
    \end{cases}
\]
We still have $\chi(S^2,A,\ord)\le 0$ because $\#A\ge 4$. Since the
natural map
$\Mod^{\pm}(S^2,A,\ord) =\Mod^{\pm}(S^2,A,\ord')\to
\Mod^{\pm}(\pi_1(S^2,A,\ord'))$ factors as
\[\Mod^{\pm}(S^2,A,\ord)\to \Mod^{\pm}(\pi_1(S^2,A,\ord))\to
  \Mod^{\pm}(\pi_1(S^2,A,\ord')),
\]
it suffices to prove the theorem for the orbisphere $(S^2,A,\ord')$.

Consider $\phi\in \Mod^{\pm}(\pi_1(S^2,A,\ord'))$. By
\cite{zieschang-vogt-coldewey:spdg}*{Theorem~5.8.3} there is a
homeomorphism $h_{\phi}\colon (S^2\setminus A)\selfmap$ inducing
$\phi$. By \cite{zieschang-vogt-coldewey:spdg}*{Theorem~5.14.1} it is
unique up to isotopy rel $A$. Indeed, if $h'_{\phi}$ is another such
homeomorphism, then the lift of $h_{\phi}\circ h^{-1}_{\phi}$ to the
universal cover $\subscript{\pi_1(S^2,A,\ord')}U$ commutes with the action of
$\pi_1(S^2,A,\ord')$; and by
\cite{zieschang-vogt-coldewey:spdg}*{Theorem~5.14.1}
$h_{\phi}\circ h^{-1}_{\phi}$ is isotopic to the identity relatively
to the action of $\pi_1(S^2,A,\ord')$. We get a homomorphism
$\Mod^{\pm}(\pi_1(S^2,A,\ord'))\to \Mod^{\pm}(S^2,A,\ord')$ that is
inverse to $\Mod^{\pm}(S^2,A, \ord') \to \Mod(\pi_1(S^2,A,\ord'))$.
\end{proof}

For $G=\pi_1(S^2,A,\ord)$, define $\Mod(G)$ to be the image in
$\Mod^{\pm}(G)$ of $\Mod(S^2,A,\ord)$.

\begin{cor}\label{cor:ModPresUnderInessMap}
  Let $G$ be an orbisphere group and let $\widetilde G$ be a
  sphere group together with an inessential forgetful map
  $\widetilde G\to G$. Then the natural map $\widetilde G\to G$
  induces an isomorphism $\Mod(\widetilde G)\to\Mod(G)$.
\end{cor}
\begin{proof}
  If $\chi(G)>0$, then $\Mod(\widetilde G)$ and $\Mod( G)$ are trivial
  groups. The case $\chi(G)\le 0$ follows from
  Theorem~\ref{thm:dehn-nielsen-baer-zieschang}.
\end{proof}
\noindent Since $\Mod(G)$ is countable and its elements can be
explicitly enumerated we also have
\begin{cor}\label{cor:can lift groups}
  Let $\widetilde G\to G$ be as in
  Corollary~\ref{cor:ModPresUnderInessMap}. There is an algorithm
  that, given $h\in \Mod(G)$ computes its preimage under the
  forgetful map $\Mod(\widetilde G)\to\Mod(G)$.\qed
\end{cor}

\begin{rem}
  A more effective algorithm is to decompose $h\in \Mod(G)$ into a
  product of Dehn twists, and then to lift each Dehn twist to
  $\Mod(\widetilde G)$. This will be the developed
  in~\cite{bartholdi-dudko:bc5}.
\end{rem}

\begin{prop}\label{prop:OrbiIndByHomeo}
  Let $G$ be an orbisphere group. 

  If $\chi(G)\le0$, then $\Mod(G)$ has index two in $\Mod^\pm(G)$, and
  there is an algorithm that, given $\phi\in\Mod^\pm(G)$, determines
  whether $\phi$ belongs to $\Mod(G)$.

  If $G$ has at least one peripheral conjugacy class $\Gamma$ with
  $\ord(\Gamma)>2$, then $\phi\in\Mod(G)$ if and only if
  $\Gamma^\phi=\Gamma$.
\end{prop}
\begin{proof}
  There exist anti-orbisphere maps, so $\Mod(G)\neq\Mod^\pm(G)$. This
  can be seen algebraically, by considering the map $\phi$ given on
  generators $\gamma_1,\dots,\gamma_n$ with $\gamma_1\cdots\gamma_n=1$
  by $\gamma_i^\phi=\gamma_i^{-\gamma_{i-1}^{-1}\cdots\gamma_1^{-1}}$.
  By Theorem~\ref{thm:dehn-nielsen-baer-zieschang} we have
  $[\Mod^\pm(G):\Mod(G)]=2$ if $\chi(G)\le0$.

  If $\phi$ is induced by an anti-orbisphere map, then
  $\phi(\Gamma)=\Gamma^{-1}\neq\Gamma$; this proves the last claim,
  and also gives an algorithm in case at least one peripheral
  conjugacy class has order $>2$. 

  Suppose now that all peripheral conjugacy classes in $G$ have order
  $2$, and let us give an algorithm deciding whether
  $\phi\in\Mod^{\pm}(G)$ is in $\Mod(G)$. Denote by $\Gamma_i$ all
  peripheral conjugacy classes in $G$, so $G$ has presentation
  $\langle\gamma_1,\dots,\gamma_n\mid\gamma_i^2,\gamma_1\cdots\gamma_n\rangle$.
  Consider the map $G\twoheadrightarrow\Z/2$ mapping generators
  $\gamma_1,\dots,\gamma_4$ to $1\pmod 2$ and the others to $0\pmod2$,
  and let $H$ denote the kernel of this map. Setting
  $\alpha=\gamma_2\gamma_1$ and $\beta=\gamma_1\gamma_3$ and
  $\delta_i=\gamma_i^{\gamma_1\gamma_2\gamma_3}$ for $n=5,\dots,n$, we
  obtain by the Reidemeister-Schreier theorem (see
  e.g.~\cite{magnus-k-s:cgt}*{\S2.3}) the presentation
  \[H=\langle\alpha,\beta,\gamma_5,\delta_5,\dots,\gamma_n,\delta_n\mid[\alpha,\beta]\delta_5\cdots\delta_n\gamma_5\cdots\gamma_n,\gamma_i^2,\delta_i^2\rangle.\]
  Consider $\phi\in\Mod^{\pm}(G)$. Since $\phi$ preserves every
  $\Gamma_i$, it also preserves $H$, and therefore induces an
  automorphism $\phi_*$ of $H/[H,H]\otimes\Q\cong\Q^2$, whose action
  on generators can be explicitly computed. We have
  \[\Mod(G)=\{\psi\in\Mod^{\pm}(G)\mid\det(\psi_*)=1\},\]
  so to determine whether $\phi\in\Mod(G)$ it suffices to compute the
  action of $\phi_*$ on $H/[H,H]\otimes\Q$. Note that, for the
  orientation-reversing map $\phi$ given at the beginning of the
  proof, we have
  $\phi_*=(\begin{smallmatrix}-1&2\\0&1\end{smallmatrix})$ in basis
  $\{\alpha,\beta\}$.
\end{proof}

\begin{rem}
  In the above proof, the homomorphism $G\twoheadrightarrow\Z/2$
  constructs a genus-$1$ cover of the punctured sphere; it is on the
  cover that the orientation of a mapping class can be read. More
  conceptually, if $n$ is even then we may consider
  $H=\ker(\gamma_i\mapsto1\pmod2)$, corresponding to the hyperelliptic
  cover of the $n$-punctured sphere, which is then a surface of genus
  $n/2-1$ without punctures; then $\phi_*$ is the action of $\phi$ on
  the homology of this cover.
\end{rem}

\subsection{Orbisphere bisets}
Let $H,G$ be orbisphere groups. An \emph{orbisphere $H$-$G$-biset} $B$
is defined exactly as in Definition~\ref{dfn:SphBis} --- the only
difference is that the acting groups are orbisphere groups. As in
Definition~\ref{dfn:SphBis} the biset $B$ induces a map
$B_*\colon C\to A$ from peripheral conjugacy classes of $G$ indexed as
$(\Gamma_a)_{a\in A}$ to peripheral conjugacy classes of $H$ indexed
as $(\Delta_c)_{c\in C}$. In the dynamical case $G=H$, then we have a
self-map $B_*\colon A\selfmap$ called the \emph{portrait} of $B$.

The same argument as in Lemma~\ref{lem:SphBisOfSphMap} shows that the
biset of an orbisphere map is an orbisphere biset. Moreover, if
$\ord_C(C)= \{2\}$, then the biset of an anti-orbisphere map is an
orbisphere biset. We are now ready to generalize
Theorem~\ref{thm:dehn-nielsen-baer+}. We omit the case
$\ord_C(C)= \{2\}$ and $\chi(G)> 0$ from the theorem --- this case
will be illustrated in Example~\ref{ex:positive orbisphere}.

\begin{thm}\label{thm:dehn-nielsen-baer++}
  Suppose that $(S^2,C,\ord_C)$ and $(S^2,A,\ord_A)$ are marked
  orbispheres with $\#C\ge 2$; write
  $H\coloneqq\pi_1(S^2,C,\ord_C,\dagger)$ and
  $G\coloneqq\pi_1(S^2,A,\ord_A,*)$ for choices of $\dagger\in
  S^2\setminus C$ and $*\in S^2\setminus A$.

  Suppose first $\ord_C(C)\neq \{2\}$. Then orbisphere maps
  $f_0,f_1\colon(S^2,C,\ord_C)\to(S^2,A,\ord_A)$ are isotopic
  if and only if $B(f_0)\cong B(f_1)$. Conversely, for every
  orbisphere $H$-$G$-biset $B$ there exists an orbisphere map
  $f\colon(S^2,C)\to(S^2,A)$, unique up to isotopy, such that
  $B\cong B(f)$.
  
  Suppose that $\ord_C(C)= \{2\}$ but $\chi(G)\le 0$. Then orbisphere
  or anti-orbisphere maps
  $f_0,f_1\colon(S^2,C,\ord_C)\to(S^2,A,\ord_A)$ are isotopic
  if and only if $B(f_0)\cong B(f_1)$. Conversely, for every
  orbisphere $H$-$G$-biset $B$ there exists an orbisphere or
  anti-orbisphere map $f\colon(S^2,C)\to(S^2,A)$, unique up to isotopy, such that $B\cong B(f)$.
\end{thm}
\noindent The proof is essentially the same as that of
Theorem~\ref{thm:dehn-nielsen-baer+}; therefore we only sketch the
argument, underlining the differences.
\begin{proof}
  Recall that the proof of Theorem~\ref{thm:dehn-nielsen-baer+} is
  based on Theorem~\ref{thm:dehn-nielsen-baer}, on
  Decomposition~\eqref{eq:DecOfSphBis} and on
  Lemmas~\ref{lem:CorrOfSphBis} and~\ref{lem:CorrOfSphBis2}.
  Theorem~\ref{thm:dehn-nielsen-baer} is generalized by
  Theorem~\ref{thm:dehn-nielsen-baer-zieschang} --- see
  Corollary~\ref{cor:ModPresUnderInessMap}.
  Decomposition~\eqref{eq:DecOfSphBis} holds for orbisphere bisets,
  while Lemma~\ref{lem:CorrOfSphBis} is adjusted as follows, with a
  completely analogous proof:

  \begin{lem}\label{lem:CorrOfSphBis:orb}
    Suppose that $\subscript H B_G$ is the biset of an orbisphere map
    $f\colon (S^2,C, \ord_C)\to (S^2,A,\ord_A)$ with
    $G=\pi_1(S^2\setminus A,*)$ and $H=\pi_1(S^2\setminus
    C,\dagger)$. If $\ord_C(C)=\{2\}$, then allow $f$ to be an
    anti-orbisphere map.
  
    For $p\in S^2$ define
    $\ord_{\widetilde A}(p)\coloneqq \ord_A(f(p))/\deg_p(f)$ and set
    $\widetilde A\coloneqq \{p\in S^2\mid \ord_{\widetilde A}(p)>1 \}$
    so that
    $f\colon (S^2,\widetilde A,\ord_{\widetilde A})\to (S^2,A,\ord_A)$
    is a covering of orbispheres. Consider $b\in B$ and let $*'$ be
    the endpoint of $b$.

    Then $\pi_1(S^2,\widetilde A,\ord_{\widetilde A}, *')$ is
    identified via
    $f_*\colon \pi_1(S^2,\widetilde A,\ord_{\widetilde A}, *')\to
    \pi_1(S^2,A, *)$ with $G_b$, and via this identification the
    $\pi_1(S^2,\widetilde A,\ord_{\widetilde A},*')$-$G$-biset of
    $f\colon (S^2,f^{-1}(A))\to (S^2,A)$ is isomorphic to
    $\subscript{G_b}G_G$ while the $H$-$\pi_1(S^2,f^{-1}(A),*')$-biset of
    $(S^2,C)\overset{\one}{\rightarrow} (S^2,f^{-1}(A))$ is isomorphic
    to $H b G_b$.
 
    Moreover, via the identification of $G_b$ with
    $\pi_1(S^2,\widetilde A,\ord_{\widetilde A},*')$ the peripheral
    conjugacy classes of $G_b$ are $(\Xi_{i,j})_{i,j}$ constructed as
    follows. Let $\Gamma_1,\dots,\Gamma_n$ be the peripheral conjugacy
    classes of $G$. Then for each $\Gamma_i$ there is a unique
    decomposition
    \begin{equation}\label{eq:SplGamToXi:orb}
      (\Gamma_i^+\setminus \{1\}) \cap G_b = (\Xi^{+}_{i,1}\setminus \{1\})\sqcup (\Xi^{+}_{i,2}\setminus \{1\})\sqcup\dots \sqcup (\Xi^{+}_{i,s}\setminus \{1\}) 
    \end{equation}
    such that every $\Xi_{i,j}$ is a non-trivial conjugacy class of
    $G_b$. Assuming $\Xi_{i,j}$ is generated by
    $\gamma_{i,j}^{d(i,j)}$ with $\gamma_{i,j}\in \Gamma_i$ and with
    minimal possible $d(i,j)\ge 1$, we let $\{(d(i,j), \Xi_{i,j})\}$
    be the multiset of lifts of $\Gamma_i$ via $\subscript{G_b}G_{G}$.\qed
  \end{lem}
  \noindent Lemma~\ref{lem:CorrOfSphBis2}, as well as its proof, holds
  for orbisphere bisets, as long as references to
  Lemma~\ref{lem:CorrOfSphBis} are replaced by references to
  Lemma~\ref{lem:CorrOfSphBis:orb}.

  Now the proof copies that of Theorem~\ref{thm:dehn-nielsen-baer+}
  with the exception that in the case $\ord_C(C)=\{2\}$
  anti-orbisphere maps are indistinguishable from orbisphere maps.
\end{proof}

\begin{cor}\label{cor:can lift bisets}
  Let $\widetilde G\to G$ and $\widetilde H\to H$ be as in
  Corollary~\ref{cor:ModPresUnderInessMap}. There is an algorithm
  that, given $B$ an $H$-$G$-biset, computes its preimage under the
  forgetful map.
\end{cor}
\begin{proof}
Anti-sphere bisets are defined in complete analogy with sphere bisets, see Definition~\ref{dfn:SphBis}, except that the multisets of all lifts of $\Gamma_i$ contain $\Delta^{-1}_j$ instead of $\Delta_j$. Enumerate then all sphere and
  anti-sphere $\widetilde H$-$\widetilde G$-bisets $\tilde B$, and return the
  first one that maps to $B$ under the inessential forgetful maps.
\end{proof}

\subsubsection{Orbisphere bisets and orientation}
By Lemma~\ref{lem:sphere to sphere group}, to every non-trivial sphere group there
is a uniquely associated oriented sphere. This is also true for
orbisphere groups, except if all peripheral conjugacy classes have
order $2$; in that case, it is not possible to recover the orientation
of the topological orbisphere: if $A\subset\R$, the map
$z\mapsto\overline z\colon(\hC,A,\ord)\selfmap$ has biset isomorphic
to that of $z\mapsto z$. Thus, in Corollary~\ref{cor:can lift bisets},
the choices of $\widetilde G$ and $\widetilde H$ amount to choices of
orientations above the orbispheres of $G,H$ respectively.

\begin{prop}
  Suppose that in an orbisphere group $H$ all peripheral conjugacy
  classes have order $2$ and $\chi(H)\le0$.

  Then there is an algorithm that decides whether the bisets
  $\subscript H B_G$, $\subscript H C_G$ have the same or different
  orientation. In particular, there is an algorithm deciding whether
  an orbisphere biset $\subscript H B_H$ corresponds to an orientation
  preserving map.
\end{prop}
\begin{proof}
  Choose by Corollary~\ref{cor:can lift groups} sphere groups
  $\widetilde G,\widetilde H$ above $G,H$ respectively.  Using
  Corollary~\ref{cor:can lift bisets}, find $\widetilde
  H$-$\widetilde G$-bisets $\widetilde B,\widetilde C$ above $B,C$
  respectively. Then $B,C$ have the same orientation if and only if
  $\widetilde B,\widetilde C$ are either both sphere or both
  anti-sphere bisets.

  The second statement follows from the first, by taking $G=H$ and
  $C=\subscript G G_G$.
\end{proof}

Note that if $H$ has four peripheral classes, so $\chi(H)=0$, there is
a more efficient algorithm: consider indeed an orbisphere biset
$\subscript H B_H$. Since $H$ has a unique subgroup of index $2$
isomorphic to $\Z^2$, the restriction of $B$ to $\Z^2$ yields a
$2\times2$-integer matrix $M_B$, and $B$ is orientation-preserving
precisely when $\det(M_B)>0$. Details will be given
in~\cite{bartholdi-dudko:bc3}.

\begin{exple}[Different orbisphere bisets of the same map]\label{ex:positive orbisphere}
  Consider the following post-critically finite rational map
  $f(z)=\left(\frac{z^2-1}{z^2+1}\right)^2$. It has order-$2$ critical
  points at $\pm i, \pm 1, 0, \infty$ with
  $f(\pm i)=\infty, f(\pm 1)=0, f(0)=f(\infty)=1$. Let
  $A\coloneqq \{0,1, \infty\}$ be the post-critical set of $f$ and
  write $G\coloneqq \pi_1(\hC,A,*)=\langle a, b, c \mid a b c \rangle$
  so that $a$ and $c$ are counterclockwise loops around $0$ and
  $\infty$ respectively intersecting once $\R_{\ge0}$ while $b$ is a
  counterclockwise loop around $1$ intersecting twice $\R_{\ge0}$.
  
  We compute $B(f)$ viewed as a sphere map. Suppose $*\notin
  \R_{\ge0}$ and let $X\coloneqq\{1,2,3,4\}$ be the basis of $B(f)$
  normalized so that $1,2,3,4$ are unique curves in $\C\setminus
  \R_{\ge0}$ connecting $*$ to its unique preimage in $\{z\mid
  \operatorname{Re}(z)>0,\operatorname{Im}(z)>0\}$, $\{z\mid
  \operatorname{Re}(z)>0,\operatorname{Im}(z)<0\}$, $\{z\mid
  \operatorname{Re}(z)<0,\operatorname{Im}(z)<0\}$, $\{z\mid
  \operatorname{Re}(z)>0,\operatorname{Im}(z)<0\}$ respectively. Then
  $B(f)$ is described by the recursion
  \begin{equation}
    \begin{aligned}
      a &\mapsto \pair{a^{-1},1,1,ab}(1,4)(2,3)\\
      b &\mapsto \pair{c,1,1,a}(1,3)(2,4)\\
      c &\mapsto \pair{1,1,1,,1}(1,2)(3,4).
    \end{aligned}
  \end{equation}

  View next $f$ as an orbisphere map $\hC\to (\hC, A,\ord)$ with
  $\ord(A)=\{2\}$. Then $f\colon \hC\to (\hC, A,\ord)$ is a covering
  of orbifolds with
  $\hC/{\langle -z, \frac{1}{z}\rangle}\overset{f}{\approx}
  (\hC,A,\ord)$. Write
  $G'\coloneqq \pi_1(\hC, A,\ord, *)=\langle a, b, c \mid
  a b c,a^2,b^2,c^2\rangle \cong\langle -z, \frac{1}{z}\rangle$. Then
  $G'$ is isomorphic to $(\Z/2)^{2}$ and the biset of
  $f\colon \hC\to (\hC, A,\ord)$ is
  \[\one \otimes_{G} B(f) \otimes_G  G' \cong \subscript{\one}{(\Z/2)^{2}}_{(\Z/2)^{2}}.\]

  Finally view $f$ as an orbisphere map
  $ (\hC, A,\ord)\to (\hC, A,\ord^\circ)$ with
  $\ord^\circ(a)=\ord^\circ(b)=4$ and $\ord^\circ(c)=2$.  Write
  $G^{\circ}\coloneqq \pi_1(\hC, A,\ord^\circ, *)=\langle a, b, c \mid
  a b c,a^4,b^4,c^2\rangle$. Then $\chi(G^\circ)=0$ and $G^\circ$ has
  an index $3$-subgroup isomorphic to $\Z^2$. The biset of
  $(\hC, A,\ord)\to (\hC, A,\ord^\circ)$ is
  \[\subscript{G'}B^\circ_{G^\circ} = G' \otimes_{G} B(f) \otimes_G  G^\circ;\]
  as a set it has $16=\#G'\cdot\#X$ elements transitively permuted by
  $G^{\circ}$. We note that $\subscript{G'}B^\circ_{G^\circ}$ uniquely
  determines $f$ among all (orientation preserving) orbisphere maps
  from $ (\hC, A,\ord)$ to $ (\hC, A,\ord^\circ)$; however the biset
  of anti-orbisphere map
  $f\circ \bar z \colon (\hC, A,\ord)\to (\hC, A,\ord^\circ)$ is also
  $\subscript{G'}B^\circ_{G^\circ}$.
\end{exple}

\subsection{Inessential forgetful morphisms}\label{ss:inessential}
Consider a Thurston map $f\colon(S^2,A)\selfmap$.  There exist various
orbispace structures on $(S^2,A)$, namely functions
$\ord\colon A\to\{2,3,\dots,\infty\}$, that turn $f$ into an
orbisphere map. The minimal one is given by
\begin{equation}\label{eq:minorbispace}
  \ord_f(a)=\lcm\big\{\deg_z(f^k)\mid k\in\N,\,z\in f^{-k}(a)\big\}\in\N\cup\{\infty\},
\end{equation}
and is denoted by $(S^2,P_f,\deg_f)$ with $P_f\subseteq A$ the
post-critical set of $f$. The maximal one is given by
$\deg_\infty(a)=\infty$ for all $a\in A$, and is denoted by
$(S^2,A,\deg_\infty)$. The former is the usual one to consider when
one is just given a map $f$; the latter has the advantage of being
defined on a punctured sphere rather than on an orbispace; but
intermediate orbispace structures will be required, firstly because
sometimes the map $f$ to consider is given by a previous construction
that specified its orbispace structure; secondly because periodic
non-postcritical cycles are invisible in the fundamental group of the
minimal orbispace structure (they correspond to generators of order
$1$); and thirdly because the maximal structure will not satisfy the
required ``contraction'' properties in the next
article~\cite{bartholdi-dudko:bc3} of the series. If periodic cycles
were marked is $A$, then $\deg_f$ can be extended to $A\setminus P_f$
by taking a constant value, for example $2$, on non-post-critical
points.

These notions have algebraic counterparts. Consider an orbisphere
biset $\subscript G B_G$. It has a \emph{portrait}
$B_*\colon A\selfmap$ induced by the map on peripheral conjugacy
classes, and a \emph{local degree} $\deg_a(B)$. There is a
\emph{minimal orbisphere quotient} of $G$ associated with $B$, which
is the quotient $\overline G$ of $G$ by the additional relations
$\Gamma_a^{\ord_B(a)}=1$, for
\begin{equation}\label{eq:ordB}
  \ord_B(a)=\lcm\{d\mid n\ge0\text{ and $\Gamma_a$ has a lift of degree $d$ under }B^{\otimes n}\},
\end{equation}
and clearly $\ord_{B}(a) \mid \ord_{G}(a)$. We call the quotient biset
$\overline G\otimes_G B\otimes_G\overline G$ the \emph{minimal
  orbisphere biset} of $\subscript G B_G$.

There is an even smaller quotient of $G$ than $\overline G$,
informally ``the smallest for which a quotient biset of $B$ can be
defined''. Let $\subscript G B_G$ be an orbisphere biset; we then
naturally get a right action of $G$ on
\[T(B)\coloneqq\bigsqcup_{n\ge0}\{\cdot\}\otimes_G B^{\otimes n}.
\]
If $B$ is left-free of degree $d$ then $T(B)$ naturally has the
structure of a $d$-regular rooted tree: if $S$ is a basis of $B$, then
$T(B)$ is in bijection with the set of words $S^*$, which forms a
$\#S$-regular tree if one puts an edge between $s_1\dots s_n$ and
$s_1\dots s_{n+1}$ for all $s_i\in S$. The action of $G$ on $T(B)$
need not be free; following~\cite{nekrashevych:ssg}*{5.1.1} we denote
by $\IMG_B(G)$ the quotient of $G$ by the kernel of this action.

Let $\widetilde G$ and $G$ be two orbisphere groups with a forgetful
morphism $\iota\colon\widetilde G\to G$ between them. Let
$\subscript{\widetilde G}{\widetilde B}_{\widetilde G}$ be an
orbisphere biset.  If the kernel of $\iota\colon\widetilde G\to G$ is
contained in the kernel of
$\widetilde G\to\IMG_{\widetilde B}(\widetilde G)$, then the biset
\begin{equation}\label{eq:tildeB to B}
  B\coloneqq G\otimes_{\widetilde G}\widetilde B_{\widetilde G}\otimes G
\end{equation}
is an orbisphere $G$-$G$-biset. It is easy to see that the kernel of
$G\to \overline G$ is contained in the kernel of
$G\to \overline \IMG(G)$.

\begin{cor}
  Suppose that $\widetilde G\to G$ is an inessential forgetful
  morphism such that
  $B\coloneqq G\otimes_{\widetilde G}\widetilde B_{\widetilde
    G}\otimes G$ as in~\eqref{eq:tildeB to B} is an orbisphere biset.
  Then the natural map
  $\widetilde b\mapsto 1\otimes\widetilde b\otimes 1$ induces an
  isomorphism between the mapping class bisets $M(\widetilde B)$ and
  $M(B)$.
\end{cor}
\begin{proof}
  Follows immediately from
  Theorems~\ref{thm:dehn-nielsen-baer-zieschang}
  and~\ref{thm:dehn-nielsen-baer++}.
\end{proof}

\subsection{Freely-oriented mapping class bisets}
Consider $f\colon (S^2,C)\to (S^2,A)$ and as in~\eqref{eq:intro:m}
define the \emph{freely-oriented mapping class biset}
\[M^{\pm}(f,C,A)=\{m'fm''\mid m'\in\Mod^{\pm}(S^2,C),m''\in\Mod^{\pm}(S^2,A)\}.
\] 
Then $M^{\pm}(f,C,A)$ is a \emph{crossed product} of $\Z/2$ and $M(f,C,A)$ in
the following way. On the level of mapping class groups we have short
exact sequences
\begin{equation}\label{eq:Modpm:split}
  \begin{aligned}
1\longrightarrow \Mod(S^2,C)\longrightarrow &\Mod^{\pm}(S^2,C)\longrightarrow  \Z/2\longrightarrow 1,\\
1\longrightarrow \Mod(S^2,A)\longrightarrow &\Mod^{\pm}(S^2,A)\longrightarrow  \Z/2\longrightarrow 1.
 \end{aligned}
\end{equation}
In particular, we have a short exact sequence of bisets, see
Definition~\ref{defn:BisExt},
\begin{equation}\label{eq:SES:NonOrientMCB}
  \subscript{\Mod(S^2,C)}M^{\pm}(f,C,A)_{\Mod(S^2,A)}\hookrightarrow \subscript{\Mod^{\pm}(S^2,C)}M^{\pm}(f,C,A)_{\Mod^{\pm}(S^2,A)} \twoheadrightarrow  \Z/2.
\end{equation}
The sequences in~\eqref{eq:Modpm:split} split; fix splittings
$\sigma \colon \Z/2\to \Mod^{\pm}(S^2,C)$ and
$\sigma \colon \Z/2\to \Mod^{\pm}(S^2,A)$. For example, we may choose
$\sigma(1)=\bar z$ after identifying $S^2$ with $\widehat \C$ and
realizing $C$ and $A$ as subsets of $\R$. Then $\Z/2$ acts on
$m\in\Mod(S^2,C)$ or $m\in\Mod(S^2,A)$ by
$(m)^{a}= \sigma(a^{-1}) m\sigma(a)$ and on $b\in M(f,C,A)$ by
$(b)^{a}= \sigma(a^{-1}) b\sigma(a)$. Using these splitting we have
semidirect products
$\Mod^{\pm}(S^2,C)=\Z/2 \ltimes_\sigma \Mod(S^2,C)$ and
$\Mod^{\pm}(S^2,A)=\Z/2 \ltimes_\sigma \Mod(S^2,A)$. Define
$\Z/2 \ltimes_\sigma M(f,C,A)$ to be the
$(\Z/2 \ltimes_\sigma
\Mod(S^2,C))$-$(\Z/2 \ltimes_\sigma \Mod(S^2,A))$-biset that is
$\Z/2\times M(f,C,A)$ as a set, endowed with the actions
\begin{equation}
\label{eq:NonOrientMCB:CrossProd}
(a', m') \cdot (a, g)\cdot (a'', m'')=  (a' a a'', (m')^{a a''} (g)^{a''} m'').
\end{equation}
\begin{lem}[Cross product for~\eqref{eq:SES:NonOrientMCB}]
  The biset of $ M^{\pm}(f,C,A)$ is isomorphic to the cross product
  $\Z/2 \ltimes M(f,C,A)$ by an isomorphism mapping
  $\sigma(a) g\in M^{\pm}(f,C,A)$ to $(a,g)\in \Z/2\times M(f,C,A)$ with
  $g\in M(f,C,A)$.  \qed
\end{lem} 
  
In the dynamical setting $(S^2,C)=(S^2,A)$ we may
interpret~\eqref{eq:NonOrientMCB:CrossProd} as a semidirect product of
bisets. Denote by $K^{\pm}(S^2,A)$ the set of isotopy classes of
sphere and anti-sphere maps $f\colon (S^2,A)\selfmap$. It is naturally
a semigroup under composition, and we have a short exact sequence of
semigroups
\[1\longrightarrow K(S^2,A)\longrightarrow K^{\pm}(S^2,A)\longrightarrow  \Z/2\longrightarrow 1
\]
that splits. Fixing a splitting
$\sigma \colon \Z/2 \to K^{\pm}(S^2,A) $ we get
$K^{\pm}(S^2,A)\cong \Z/2 \ltimes_\sigma K^{\pm}(S^2,A)$.  Of course
$M^{\pm}(f,C,A)$ is isomorphic to the algebraic associated
freely-oriented mapping class biset $M^{\pm}(B(f))$ consisting of
sphere and anti-sphere bisets; thus we also have
$M^{\pm}(B(f))=\Z/2\ltimes_{\sigma} M(B(f))$ once a splitting is
fixed.

\section{Examples}\label{ss:examples}
We briefly recall the application of Teichm\"uller theory to Thurston
maps; see~\cite{bartholdi-dudko:bc0}*{\S\ref{bc0:ss:canonical decomposition}} and~\cite{selinger:canonical}.

The \emph{Teichm\"uller space} of $(S^2,A)$ is the set $\mathscr T_A$
of complex structures on $(S^2,A)$, or equivalently Riemannian
structures on $S^2\setminus A$ of curvature $-1$. We view it as the
set of homeomorphisms $(S^2,A)\to\hC$ up to isotopy rel $A$ and
post-composition by M\"obius transformations. \emph{Moduli space}
$\mathscr M_A$ is the set of embeddings $A\hookrightarrow\hC$ up to
post-composition by M\"obius transformations. Teichm\"uller space is
contractible, while the fundamental group of $\mathscr M_A$ is
$\Mod(S^2,A)$; there is a natural map
$\mathscr T_A\to\mathscr M_A$, given by restricting homeomorphisms
$(S^2,A)\to\hC$ to $A$, which is a universal covering map with group
$\Mod(S^2,A)$.

Thurston associates with $f\colon(S^2,C)\to(S^2,A)$ a map
$\sigma_f\colon\mathscr T_A\to\mathscr T_C$ defined by pulling back
complex structures through $f$, see~\cite{douady-h:thurston}: given
$h\in\mathscr T_A$ represented by a homeomorphism
$h\colon(S^2,A)\to\hC$, the composition $h\circ f$ defines a
holomorphic atlas on $S^2\setminus C$, with erasable singularities at
$C$; so there exists a homeomorphism $h'\colon S^2\to\hC$ and a
rational map $f_h$ making the following diagram commutative:
\begin{equation}\label{eq:h'}
  \begin{tikzcd}
    (S^2,C)\arrow[dashed]{rr}{h'}\dar{f} & & \hC\arrow[dashed]{d}{f_h}\\
    (S^2,A)\arrow{rr}{h} & & \hC.
  \end{tikzcd}
\end{equation}
Furthermore, $h'$ is uniquely defined up to post-composition by a
M\"obius transformation, and its class in $\mathscr T_C$ depends only
on the class of $h$ in $\mathscr T_A$. The map $\sigma_f$ is defined
by $\sigma_f(h)=h'$. It is a contravariant functor from the category
of spheres $(S^2,A)$ with sphere maps as morphisms to the category of
Teichm\"uller spaces $\mathscr T_A$ with analytic maps as morphisms.

Unless $f$ is a homeomorphism, the map
$\sigma_f\colon\mathscr T_A\to\mathscr T_C$ does not descend to a map
$\mathscr M_A\to\mathscr M_C$. However, define $H_f\le\Mod(S^2,A)$,
the subgroup of \emph{liftable classes}, as
\begin{equation}\label{eq:liftable}
  H_f=\{h\in\Mod(S^2,A)\mid \exists h'\in\Mod(S^2,C)\text{ with
}h\circ f=f\circ h'\}.
\end{equation}
By Proposition~\ref{prop:MCBfinite}
or~\cite{koch-pilgrim-selinger:pullback}*{Proposition 3.1}, the
subgroup $H_f$ has finite index in $\Mod(S^2,A)$, and
Proposition~\ref{prop:MCBfree} implies that the $h'$
in~\eqref{eq:liftable} is unique; see
also~\cite{bartholdi-buff:compactification}. Define then
\[\mathscr W_f=\mathscr T_A/H_f.\]
Then $\mathscr W_f$ is a finite covering of $\mathscr M_A$, say with
covering map $i$, and the map $\sigma_f$ descends to a map
$\overline{\sigma_f}\colon\mathscr W_f\to\mathscr M_C$. We therefore
have a correspondence $(\overline{\sigma_f},i)$, which we call
\emph{modular correspondence}, see~\cite{koch:teichmuller-cfe},
\begin{equation}\label{eq:thurston correspondence}
\mathscr M_C
\overset{\overline{\sigma_f}}\longleftarrow
\mathscr W_f \overset i\longrightarrow  \mathscr M_A
\end{equation}
with the associated biset
(see~\cite{bartholdi-dudko:bc1}*{\S\ref{bc1:ss:biset of
    correspondence} and Equation~\eqref{bc1:eq:Bfi:lmm:FibrCorr}})
\begin{equation}
\label{eq:def:bis:MapClassCorr}
  B((\overline{\sigma_f},i), \dagger,*)=\{(\delta\colon[0,1]\to \mathscr M_C,z\in \mathscr W_f) \mid\delta(0)=\dagger, \delta(1)=\overline{\sigma_f}(z), i(z)=*\}\,/\,{\approx}.
\end{equation}

Choose $\widetilde \dagger\in \mathscr T_C$ and
$\widetilde *\in \mathscr T_ A$ above $\dagger ,*$ respectively; by
definition they are maps $\widetilde \dagger\colon (S^2,C)\to \hC$ and
$\widetilde *\colon (S^2,A)\to \hC$. Using these choices we have
identifications $\pi_1(\mathscr M_C,\dagger)\cong \Mod(S^2,C)$ and
$\pi_1(\mathscr M_A,*)\cong \Mod(S^2,A)$. Indeed, every
$\gamma\in \pi_1(\mathscr M_C,\dagger)$ determines an element of
$\Mod(S^2,C)$ as the composition
\begin{equation}\label{eq:ident:Mod with pi1:C}
C\hookrightarrow(S^2,C)\overset{\widetilde \dagger }{\longrightarrow} \dagger \overset{\widetilde\gamma}{\longrightarrow}\dagger \overset{\widetilde \dagger ^{-1}}{\longrightarrow}(S^2,C)
\end{equation}
with $\dagger$ viewed as a marked Riemann sphere and
$ \dagger \overset{\widetilde\gamma}{\longrightarrow}\dagger$ viewed
as an isotopy of marked Riemann spheres. Similarly, every
$\gamma\in \pi_1(\mathscr M_A,*)$ determines an element of
$\Mod(S^2,A)$ as the composition
\begin{equation}\label{eq:ident:Mod with pi1:A}
  A \hookrightarrow (S^2,A)\overset{\widetilde * }{\longrightarrow} * \overset{\widetilde\gamma}{\longrightarrow}* \overset{\widetilde * ^{-1}}{\longrightarrow}(S^2,A).
\end{equation}

The following result identifies biset elements
$b=(\delta, z)\in B((\overline{\sigma_f},i), \dagger,*)$ with branched
coverings $b\colon (S^2,C)\to (S^2,A)$:
\begin{prop}\label{prop:modular correspondence}
  Suppose that $\widetilde \dagger\in \mathscr T_ C$ and
  $\widetilde *\in \mathscr T_ A$ are fixed as above and suppose that
  $\pi_1(\mathscr M_C,\dagger)\cong \Mod(S^2,C)$ and
  $\pi_1(\mathscr M_A,*)\cong \Mod(S^2,A)$ are identified
  by~\eqref{eq:ident:Mod with pi1:C} and~\eqref{eq:ident:Mod with
    pi1:A}.

  Then the biset of the modular correspondence~\eqref{eq:thurston
    correspondence} is isomorphic to the mapping class biset
  $M(f,C,A)$, see Definition~\ref{def:MapClassBis}, by an isomorphism
  mapping $(\delta,z)\in B((\overline{\sigma_f},i), \dagger,*)$ given
  in the form of~\eqref{eq:def:bis:MapClassCorr} to the following
  sphere map~\eqref{eq:def:b as a S map}. Choose any
  $\widetilde z\in \mathscr T_A$ above $z$. Then the sphere map
  \begin{equation}\label{eq:def:b as a S map}
    b\colon (S^2,C)\overset{\widetilde\dagger }\longrightarrow \dagger \overset{\widetilde\delta}{\longrightarrow}\delta(1)
    \overset{\overline{\sigma_f} (\widetilde z)^{-1}}{\longrightarrow}(S^2,C)
    \overset{f}{\longrightarrow}(S^2,A)
    \overset{\widetilde z}{\longrightarrow}
    * \overset{\widetilde *^{-1}}\longrightarrow (S^2,A)
  \end{equation}
  is independent up to isotopy of the choice of $\widetilde z$. Here
  the map $\widetilde\delta\colon[0,1]\to\mathscr T_C$ is the lift of
  $\delta$ that ends at $\overline{\sigma_f} (\widetilde z)(C))$, and
  represents the continuous deformation of
  $(\hC,\widetilde\dagger(C))$ into
  $(\hC,\overline{\sigma_f} (\widetilde z)(C))$.
\end{prop}
\begin{proof}
  Suppose that $\widetilde z' \in \mathscr T_A$ is a different choice
  of a point above $z$. Then $\widetilde z'= m\widetilde z$ for some
  $m\in H_f$, see~\eqref{eq:liftable}, and $f m=m' f$ for some
  $m'\in \Mod(S^2,C)$. By definition of $\sigma_f$, we have
  $\overline{\sigma_f}(m \widetilde z)=m' \sigma(\widetilde z)$;
  therefore substitution of $\widetilde z$ by $\widetilde z'$ does not
  change the isotopy type of~\eqref{eq:liftable}.

  The map in question is a morphism of left-free transitive bisets; it
  is an isomorphism because the group of the bisets have the same
  group of liftable elements.
\end{proof}

\subsection{Twisted rabbits}
We give now a few examples of calculations of the structure of mapping
class bisets, following~\cite{bartholdi-n:thurston}. It was used in
that article to solve ``twisted rabbit'' problems, of the following
kind: consider the ``rabbit'' polynomial $f_R(z)=z^2+c$ for
$c\approx-0.12256+0.74486i$, with post-critical set
$A=\{0,c,c^2+c,\infty\}$. Let $t$ denote a Dehn twist about a simple
closed curve surrounding $c$ and $c^2+c$ (the ``ears'' of the
rabbit). By Thurston's Theorem~\ref{thm:thurston} characterizing
rational maps, the ``twisted rabbit'' $t^n\circ f_R$ is
combinatorially equivalent to a degree-two polynomial, which may then
be normalized in the form $z^2+c$ with $(c^2+c)^2+c=0$; so it is
either $f_R$, $f_C\coloneqq z^2+\overline c$ (the ``corabbit''), or
$f_A\approx z^2-1.7549$ (the ``airplane''). The problem is to
determine which one.

The reformulation of this question is the following. The biset
$M(f_R,A)$ is a left-free $\Mod(S^2,A)$-biset of degree $2$, and can be
computed explicitly by Algorithm~\ref{algo:compute M(f,C,A)}. By
Thurston's Theorem~\ref{thm:thurston}, there are precisely $3$
conjugacy classes (in the sense
of~\cite{bartholdi-dudko:bc1}*{\S\ref{bc1:ss:cc in bisets}}) in
$M(f_R,A)$, containing respectively the rabbit, corabbit and airplane
polynomials. As we shall see in~\cite{bartholdi-dudko:bc4}, the
conjugacy problem is solvable in $M(f_R,A)$.

The answer to the concrete question above is: to determine the
conjugacy class of $t^n\circ f_R$, write $n$ is base $4$, as
$n=\sum_{i\ge0}n_i4^i$ with $n_i\in\{0,1,2,3\}$; if $n\ge0$ then
almost all $n_i=0$ while if $n<0$ then almost all
$n_i=3$. Then~\cite{bartholdi-n:thurston}*{Proposition~4.3}
\begin{equation}\label{eq:rabbit}
  t^n\circ f_R\sim\begin{cases}f_R & \text{ if $n\ge0$ and all }n_i\in\{0,3\},\\
    f_C & \text{ if $n<0$ and all }n_i\in\{0,3\},\\
    f_A & \text{ if some }n_i\in\{1,2\}.\end{cases}
\end{equation}

The equalities in~\eqref{eq:rabbit} follow from an explicit
computation of the structure of $M(f_R,A)$. In this case, it is more
direct to obtain it using Teichm\"uller theory, as it is done
in~\cite{bartholdi-n:thurston}*{\S5}, than to run our general
algorithm.  Moduli space $\mathscr M_A$ may be described as
$\hC\setminus\{0,1,\infty\}$ by identifying $A=\{0,c,c^2+c,\infty\}$
with its cross-ratio $c+1$. The correspondence
$\mathscr M_A\leftarrow\mathscr W_{f_R}\to\mathscr M_A$
from~\eqref{eq:thurston correspondence} descends in this case to the
inverse of a rational map $g\colon\mathscr M_A\leftarrow\mathscr
M_A$. Indeed, consider $h'\colon(S^2,A)\to\hC$ in Teichm\"uller space,
and its image $\sigma_{f_R}(h')=h$. Let $\{0,1,w,\infty\}$ be the
image of $h'$ in moduli space. Then in~\eqref{eq:h'} the rational
function $f_h$ has degree $2$, has critical points at $0,\infty$, and
satisfies
\[f_h(\infty)=\infty,\quad f_h(0)=1,\quad f_h(w)=0
\]
so $f_h(z)=1-z^2/w^2$; and the image $\{0,1,w',\infty\}$ of $h'$ in
moduli space is given by $w'=f_h(1)$ so we have
\[g(w)=1-\frac1{w^2}.\]

The post-critical set of $g$ is $\{0,1,\infty\}$, and $M(S^2,A)$ is
naturally identified with $\pi_1(\hC\setminus\{0,1,\infty\},w)$, by
identifying the loop traced by $w$ in $\hC\setminus\{0,1,\infty\}$
with the mapping class of $S^2\setminus A$ in which the non-critical
preimage of the critical point is dragged along the loop, keeping all
other points in $A$ fixed. We write
$M(S^2,A)=\langle s,t,u\mid u t s\rangle$, with $s$ representing
the positive Dehn twist about a curve surrounding $c^2+c$ and $0$.

\begin{figure}
  \begin{center}\begin{tikzpicture}[x=3.5cm,y=3.5cm]
  \node (0) at (0,0) {$0$};
  \node (1) at (1,0) {$1$};
  \node (m1) at (-1,0) {$-1$};
  \node (*1) at (0.8774388331,0.74486176) [label=left:$f_R$] {$*_1$};
  \node (*2) at (-0.8774388331,-0.74486176) {$*_2$};
  \lasso{(*1)}{4mm}{(1)}{node[right] {$s$}};
  \lasso{(*1)}{4mm}{(0)}{node[right=2pt] {$t$}};
  \lasso[dotted]{(*2)}{4mm}{(m1)}{}
  \draw[dotted] (*2) -- ($(0)!4mm!(*2)$);
  \draw[->] (*1) .. controls +($(0)!0.5!160:(*1)$) and +($(0)!0.5!20:(*1)$) .. node[left] {$f_R\cdot t$} (*2);
  \draw (0,0) ellipse (6cm and 4cm);
  \draw[->,dash pattern=on 0pt off 1cm] (6cm,1cm) -- +(0,-1cm) node [left] {$u$};
  \clip (0,0) ellipse (6cm and 4cm);
  \draw (*1) -- +(1,2);
  \draw[dotted] (*2) -- +(-1,-2);
  \end{tikzpicture}\end{center}
  \caption{Computing the biset of the map $1-w^{-2}$}\label{fig:bnmap}
\end{figure}
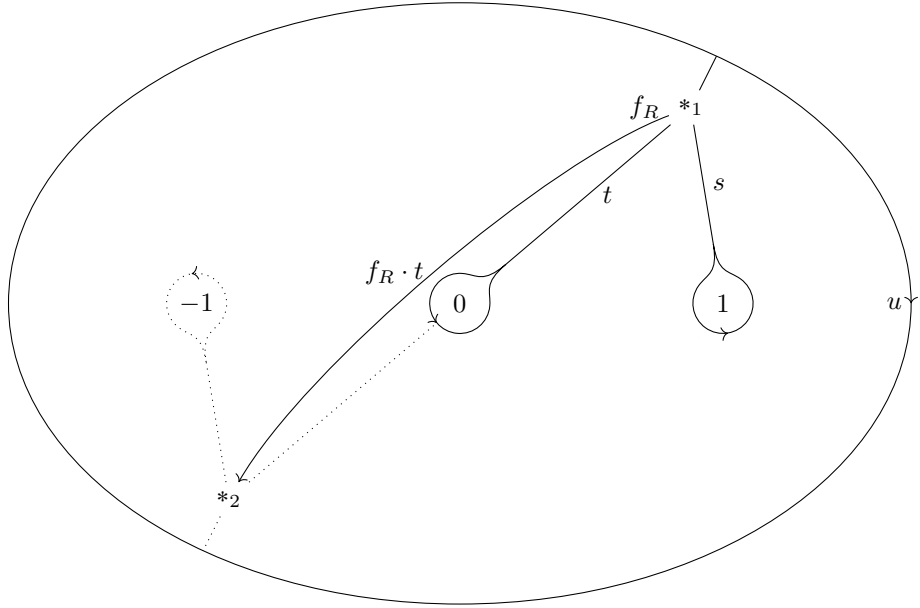

We have $M(f_R,A)\cong B(g)$, and may write its structure as follows,
in basis $\{f_R,f_R t\}$, see Figure~\ref{fig:bnmap}:
\begin{equation}\label{eq:rabbit mcb}
  \begin{alignedat}{3}
    f_R\cdot t&=f_R t,\qquad & f_R\cdot s&=t\cdot f_R,\qquad & f_R\cdot u&=s\cdot f_R t,\\
    f_R t\cdot t&=u\cdot f_R,\qquad & f_R t\cdot s&=f_R t,\qquad & f_R t\cdot u&=f_R.
  \end{alignedat}
\end{equation}

The claims in~\eqref{eq:rabbit} then follow from the initial cases
\[t^{-1}\cdot f_R\sim f_R t^{-1}\sim f_C,\quad t\cdot f_R\sim f_R t\sim f_A\]
which can be checked by hand, and the equations
\begin{alignat*}{2}
  t^{4k}f_R &\sim f_R t^{4k}=u^{2k}f_R\sim f_R u^{2k}=s^k f_R\sim f_R s^k&&=t^k f_R,\\
  t^{4k+1}f_R &\sim f_R t^{4k+1}=u^{2k}f_R t\sim f_R t u^{2k}=s^k f_R t\sim f_R t s^k&&=f_R t,\\
  t^{4k+2}f_R &\sim f_R t^{4k+2}=u^{2k+1}f_R\sim f_R u^{2k+1}=s^{k+1} f_R t\sim f_R t s^{k+1}&&=f_R t,\\
  t^{4k+3}f_R &\sim f_R t^{4k+3}=u^{2k+1}f_R t\sim f_R t u^{2k+1}=s^k f_R\sim f_R s^k&&=t^k f_R
\end{alignat*}
which directly follow from~\eqref{eq:rabbit mcb}.

Similar calculations were performed
in~\cite{bartholdi-n:thurston}*{\S6} to study
$M(z^2+i,A,A)$ with $A=\{i,i-1,-i,\infty\}$. Again this biset is
left-free of degree $2$; and coincides with the biset of the rational
map $g(w)=(1-2/w)^2$ after $\mathscr M_A$ has been identified with
$\hC\setminus\{0,1,\infty\}$. There are infinitely many conjugacy
classes in $M(z^2+i,A,A)$,
see~\cite{bartholdi-n:thurston}*{Corollary~6.11}: two
classes corresponding to rational maps $z^2+i$ and $z^2-i$, and a
$\Z$'s worth of conjugacy classes corresponding to obstructed maps
$f_n$ with $n\in\Z$, as described
in~\cite{bartholdi-dudko:bc0}*{\S\ref{bc0:ex:z2pi}}.

\subsection{Belyi maps}\label{ss:BelyaMaps}
We now describe a class of maps $f$ for which the modular
correspondence~\eqref{eq:thurston correspondence}, and the mapping
class biset, can be quite explicitly understood: sphere maps
$f\colon(S^2,C)\to(S^2,A)$ with $\#A=4$ and such that only three
values in $A$ are critical.  Even though the data are purely
topological, we write $A=\{0,1,\infty,v\}$ by abuse of notation. There
is an underlying map
$f'\colon(S^2,f^{-1}\{0,1,\infty\})\to(S^2,\{0,1,\infty\})$, and $f$
may be considered as a `decorated' version of $f'$ with the extra
marked point $v$.

In~\cite{bartholdi-dudko:bc3} we shall consider the more general
situation of a map $f\colon(S^2,C)\to(S^2,A)$ and its restriction
$f'\colon(S^2,C')\to(S^2,A')$ for $C'\subset C, A'\subset A$; the
philosophy being that $M(f,C,A)$ is an extension of $M(f',C',A')$ via
the natural maps $\Mod(S^2,C)\twoheadrightarrow\Mod(S^2,C')$,
$\Mod(S^2,A)\twoheadrightarrow\Mod(S^2,A')$ and
$M(f,C,A)\twoheadrightarrow M(f',C',A')$, using which $M(f,C,A)$ can
be efficiently encoded using $M(f',C',A')$ and $B(f')$. In our case
$M(f',C',A')$ is trivial and the considerations are much simplified.

We normalize all maps $(S^2,A)\to\hC$ so that
$0\mapsto0,1\mapsto1,\infty\mapsto\infty$, and in this manner we
identify $\mathscr M_A$ with the image of $v$ in
$\hC\setminus\{0,1,\infty\}$. We also identify $f'$ with a
rational map; such maps with critical values in $\{0,1,\infty\}$ are
called \emph{Belyi maps}.

Let us fix a basepoint $*\in S^2\setminus A$, and write
$\pi=\pi_1(S^2\setminus A,*)$. It admits a presentation as a sphere group
\[\pi=\langle a,b,c,d\mid d c b a\rangle\]
generated by loops around the punctures $0,1,\infty,v$
respectively. Then the covering
$f\colon S^2\setminus f^{-1}(A)\to S^2\setminus A$ is described by an
index-$\deg(f)$ subgroup $\pi_f\le\pi$: fix a preimage $*_0$ of $*$,
and let $\pi_f$ denote those loops in $\pi$ that lift to a loop at
$*_0$. Equivalently, $S^2\setminus A\cong\pi\backslash\mathbb D$ and
$S^2\setminus f^{-1}(A)\cong\pi_f\backslash\mathbb D$ with
$f\colon\pi_f\backslash\mathbb D\to\pi\backslash\mathbb D$ the natural
map.

There is a natural map
$\pi=\pi_1(S^2\setminus
A,*)\twoheadrightarrow\pi_1(S^2\setminus\{0,1,\infty\},*)\eqqcolon G$;
let $H$ denote the image of $\pi_f$ in $G$. We have
$G=\pi/\langle d\rangle^\pi$; and $\pi_f\ge\langle d\rangle^\pi$ since
$d$ lifts to $d$ loops at every preimage of $*$; so
$[G:H]=[\pi:\pi_f]=\deg(f)$, and we may also represent $f'$ as the
quotient map $H\backslash\mathbb D\to G\backslash\mathbb D$.

By choosing a path from $*$ to $v$, we may also identify $G$ with
$\pi_1(S^2\setminus\{0,1,\infty\},v)$. Let $\{c_1,\dots,c_d\}\subset
f^{-1}(A)$ be the preimages of $v$. Then $G$ acts by monodromy on
$\{c_1,\dots,c_d\}$, so that we have a transitive permutation
representation $\rho\colon G\to d\perm$. A choice of $c_1\in
f^{-1}(A)$ determines a choice of $H$ as the stabilizer of $c_1$ under
$\rho$.  The pure modular group $\Mod(S^2,A)$ is free of rank $2$, and
is identified with $G$: the element of $G$ represented by a simple
loop of $v$ turning respectively around $0,1,\infty$ corresponds to
the Dehn twist in $\Mod(S^2,A)$ in which $v$ is dragged respectively
around $0,1,\infty$.

We may also present $G$ as a sphere group,
\begin{equation}
\label{eq:def:ModGr:G}
G=\Mod(S^2,A)=\langle s,t,u\mid u t s\rangle.
\end{equation}
It is the group of outer automorphisms of $\pi$ that fixes conjugacy
classes of generators. The outer automorphisms $s,t,u$ are, modulo
interior automorphisms,
\[\begin{array}{cllll}
  s: & a\mapsto a^{ad}, & b\mapsto b, & c\mapsto c, & d\mapsto d^{ad},\\
  t: & a\mapsto a, & b\mapsto b^{db}, & c\mapsto c^{[d,b]}, & d\mapsto d^{db},\\
  u: & a\mapsto a, & b\mapsto b, & c\mapsto c^{dc}, & d\mapsto d^{dc}.
\end{array}\]
These generators are respectively Dehn twists about the curves
$ad,db,dc$.

\begin{lem}\label{lem:Numb left orb of Belyi maps}
  The group of liftable elements $H_f\le G$ is the kernel of
  $\rho\colon G\to d\perm$.

  More precisely, denote by $\Mod^*(S^2,C)$ the group of not
  necessarily pure mapping classes of $(S^2,C)$ that fix
  $0,1,\infty$. Then $\Mod^*(S^2,C)$ extends the left action of
  $\Mod(S^2,C)$ on $M(f,C,A)$ and the new extended action is free and
  transitive. For every $h\in G$ there is a unique $\widetilde h\in
  \Mod^*(S^2,C)$ such that $\widetilde h f= f h$. The element $h$ is
  liftable if and only if $\widetilde h\in \Mod(S^2,C)$.
\end{lem}
\begin{proof}
  We show that the action of $\Mod^*(S^2,C)$ on $M(f,C,A)$ is free;
  the other claims then follow easily. Suppose that $hf=f$ in
  $M(f,C,A)$ for $h\in \Mod^*(S^2,C)$. There is an $e\ge 1$ such that
  $h^e\in \Mod(S^2,C)$. We have $h^e f=f$, so by
  Proposition~\ref{prop:MCBfree} we have $h^e=\one$, so $h$ is
  isotopic to a M\"obius transformation. Since $h$ fixes $0,1,\infty$
  we get $h=\one$ in $\Mod^*(S^2,C)$.
\end{proof}

Fix $f'$-preimages $0',1',\infty'$ of $0,1,\infty$ respectively, and
write $\mathscr T'=\mathscr T_{\{0',1',\infty',c_1,\dots,c_d\}}$ and
$\mathscr M'=\mathscr M_{\{0',1',\infty',c_1,\dots,c_d\}}$. Once
$0,1,\infty$ and the points above them are fixed, the map $f_h$
in~\eqref{eq:h'} is uniquely determined, so we may consider
$\mathscr T',\mathscr M'$ rather than the larger
$\mathscr T_{f^{-1}(A)},\mathscr M_{f^{-1}(A)}$.  In the following
diagram, vertical arrows are coverings:
\begin{equation}\label{eq:diagr:ModCorr}
\begin{tikzcd}
    \mathscr T_C\ar{dd} & \mathscr T'\ar{l}\ar{d} & & \mathscr T_A\ar{ll}\ar{dl}\ar{dd}\\
    & \mathscr M'\ar{dl} & \mathscr W_f\ar[hook]{l}\ar{dr}\\
    \mathscr M_C & & & \mathscr M_A
  \end{tikzcd}
\end{equation}
The space $\mathscr M'$ is naturally a subset of affine space:
\begin{equation}
\label{eq:deef:M'}
\mathscr M'=\{(c_1,\dots,c_d)\in\hC^d\mid c_i\neq 0',1',\infty';c_i\neq c_j\},\end{equation}
and the embedding of $\mathscr W_f$ in $\mathscr M'$ is onto the curve
\[\mathscr W_f=\{(c_1,\dots,c_d)\in\hC^d\mid f(c_1)=\cdots=f(c_d)\neq0,1,\infty\}.
\]
The correspondence in then determined by the map
$\mathscr M'\to\mathscr M_C$, which encodes how $f^{-1}(A)$
corresponds to $C$.

\subsubsection{Products of mapping class bisets}\label{ss:TensProd:MappClassBis}
Write $A=\{0,1,\infty,v\}$, consider the map
$f(z)=z^5\colon(\hC,\{0,1,\infty,f^{-1}(v)\})\to(\hC,A)$, set
$C=\{0,1,\infty,f^{-1}(v)\}=f^{-1}(A)$ and denote by
$\{c_1,\dots,c_5\}$ the five $f$-preimages of $v$. Choose a
homeomorphism $i\colon(\hC,A)\to(\hC,\{c_1,c_2,c_3,c_4\})$. Denote
also by $k=\one\colon(\hC,\{c_1,c_2,c_3,c_4\})\to(\hC,C)$ the map
marking $0,1,\infty,c_5$, and set $j=i k$. Denote also $\Mod(\hC,A)$
by $G$. We shall compute $M(f,C,A)$, $M(j,A,C)$ and $M(j f,A,A)$ and
note that the natural map
$M(j,A,C)\otimes M(f,C,A)\twoheadrightarrow M(j f,A,A)$ is not an
isomorphism of $G$-$G$-bisets.

Since $f$ is a Belyi map, we determine $M(f,C,A)$ using
Lemma~\ref{lem:Numb left orb of Belyi maps}. Since $\rho(G)$ is cyclic
of order $5$, the biset $M(f,C,A)$ is left-free of degree $5$. Also,
every mapping class $n\in G$ is liftable by $f$, but possibly to an
impure class: there exists $n'\in\Mod^*(\hC,\{0,1,\infty,f^{-1}(v)\})$
with $n' f=f n$, and $n'$ acts on $\{c_1,\dots,c_5\}$ by a power of
the cycle $(1,2,3,4,5)$.

The maps $i$ and $k$ are homeomorphisms, so $M(j,A,C)$ is the dual of
the forgetful homomorphism $j^*$ induced on mapping class groups, see
Lemma~\ref{lem:mcb of homeo}. In particular, $M(j,A,C)$ is
left-principal.  The product $M(j,A,C)\otimes M(f,C,A)$ is therefore
left-free of degree $5$.

We claim that the biset $M(j f,A,A)$ is left-principal. Indeed, we may
write $j f=g\ell$ with
$g=z^5\circ i\colon(\hC,A)\to(\hC,\{0,\infty,v\})$ and
$\ell=\one\colon(\hC,\{0,\infty,v\})\to(\hC,A)$. Consider $n\in G$;
then we have $j f n=g \ell n = g \ell = j f$, since
$\Mod(\hC,\{0,\infty,v\})$ is trivial and $\ell^{-1}$ is an erasing
map.

We also remark that the pullback map $\sigma_{j f}=\sigma_{g\ell}$ is
constant, because $\sigma_\ell$ is constant; so the modular
correspondence $\mathscr W_{j f}$ coincides with $\mathscr M_A$, with
induced map $\overline{\sigma_{j f}}$ constant. Its image point is the
cross-ratio of four fifth roots of unity. The
diagram~\eqref{eq:diagr:ModCorr} becomes
\[
  \begin{tikzcd}
    \mathscr T_A\ar{dd} & \mathscr T'\ar{l}{\sigma_j}\ar{d} & & \mathscr T_A\ar{ll}{\sigma_f}\ar{dl}\ar{dd}\\
    & \mathscr M'\ar{dl} & \mathscr W_f\ar[hook]{l}\ar{dr}{5:1}\ar{d}\\
    \mathscr M_A & & \mathscr W_{j f} \ar{ll}{\overline{\sigma_{j f}}}\ar[equal]{r} & \mathscr M_A.
  \end{tikzcd}
\]

\begin{figure}
\begin{center}\begin{tikzpicture}
    \begin{scope}
      \fill[gray!20] (0,3) rectangle +(3,3);
      \fill[gray!20] (3,0) rectangle +(3,3);
      \def\bendota{(6,3) .. controls (8,4) and (10,5) .. (12,3)}
      \def\bendotb{(9,0) .. controls (10,2) and (11,4) .. (9,6)}
      \def\bendtta{(6,3) .. controls (9,3) and (9,3) .. (9,0)}
      \def\bendttb{(9,0) .. controls (7.5,1) and (7,1.5) .. (6,3)}
      \fill[gray!20] \bendtta \bendttb;
      \begin{scope}
        \clip \bendotb -- (6,6) -- (6,0) -- cycle;
        \fill[gray!20] \bendota -- (12,6) -- (6,6) -- cycle;
      \end{scope}
      \begin{scope}
        \clip \bendotb -- (12,6) -- (12,0) -- cycle;
        \fill[gray!20] \bendota -- (12,0) -- (6,0) -- cycle;
      \end{scope}

      \draw[gray,very thick] (0,0) -- (0,6) -- (12,6) -- (12,0) -- cycle;
      \draw[gray,thin] (6,0) -- (6,6);
      \draw[gray,thin] (3,0) -- (3,6);
      \draw[gray,thin] (0,3) -- (6,3);
      \draw[gray,thin,name path=p12a] \bendota;
      \draw[gray,thin,name path=p12b] \bendotb;
      \draw[gray,thin,name path=p23a] \bendtta;
      \draw[gray,thin,name path=p23b] \bendttb;
      \node (A1) at (11,5) {A1};
      \node (A2) at (9.3,3.3) {A2};
      \node (A3) at (6.7,0.7) {A3};
      \node (A4) at (4.5,4.5) {A4};
      \node (A5) at (1.5,1.5) {A5};
      \path[name path=diag] (A2) -- (A3);
      \draw[red] (12,3) ++(110:0.3) -- (A1);
      \draw[red,->] (12,3.3) arc (90:270:0.3);
      \draw[red] (6,3) ++(15:0.3) .. controls (8,3.5) .. (A2);
      \draw[red] (6,3) ++ (-80:0.3) .. controls (6.3,1.5) .. (A3);
      \draw[red,->] (6.3,3) arc (0:360:0.3);
      \draw[red] (6,3) ++ (135:0.3) -- (A4);
      \draw[red] (0,3) ++ (-45:0.3) -- (A5);
      \draw[red,->] (0,2.7) arc (-90:90:0.3);
      \draw[green,name intersections={of=p12a and p12b,by=x12}] (x12) ++(45:0.3) -- (A1);
      \draw[green] (x12) ++(225:0.3) -- (A2);
      \draw[green,->] (x12) ++(0.3,0) arc (0:360:0.3);
      \draw[green,name intersections={of=p23b and diag,by=y23}] (y23) ++(225:0.3) -- (A3);
      \draw[green,->] (y23) ++(0.3,0) arc (0:360:0.3);
      \draw[green] (3,3) ++(45:0.3) -- (A4);
      \draw[green] (3,3) ++(225:0.3) -- (A5);
      \draw[green,->] (3.3,3) arc (0:360:0.3);
      \draw[blue] (9,6) ++(-30:0.3) -- (A1);
      \draw[blue,->] (8.7,6) arc (180:360:0.3);
      \draw[blue] (3,6) ++(-45:0.3) -- (A4);
      \draw[blue] (9,0) ++(75:0.3) .. controls (9.5,2) .. (A2);
      \draw[blue] (9,0) ++ (170:0.3) .. controls (7.5,0.3) .. (A3);
      \draw[blue,>-] (9.3,0) arc (0:180:0.3);
      \draw[blue] (3,0) ++(135:0.3) -- (A5);
      \draw[blue,->] (3.3,0) arc (0:180:0.3);
      \draw[blue,->] (2.7,6) arc (180:360:0.3);
      \draw[black] (12,6) ++(225:0.3) -- (A1);
      \draw[black,->] (11.7,6) arc (180:270:0.3);
      \draw[black,->] (0,5.7) arc (-90:0:0.3);
      \draw[black,->] (12,0.3) arc (90:180:0.3);
      \draw[black,name intersections={of=p23a and diag,by=x23}] (x23) ++(45:0.3) -- (A2);
      \draw[black,->] (x23) ++(0.3,0) arc (0:360:0.3);
      \draw[black] (6,0) ++(45:0.3) -- (A3);
      \draw[black,->] (6.3,0) arc (0:180:0.3);
      \draw[black] (6,6) ++(225:0.3) -- (A4);
      \draw[black,->] (5.7,6) arc (180:360:0.3);
      \draw[black] (0,0) ++(45:0.3) -- (A5);
      \draw[black,->] (0.3,0) arc (0:90:0.3);
    \end{scope}
    \begin{scope}[xshift=6cm,yshift=-3cm,scale=0.4]
      \fill[gray!20] (0,0) rectangle +(6,6);
      \draw[gray,very thick] (0,0) rectangle +(12,6);
      \draw[gray,thin] (6,0) -- (6,6);
      \node (A) at (9,3) {\large A};
      \node[red] at (6,6) [above] {\large $a$};
      \draw[red,thick] (6,6) ++(-45:0.5) -- (A);
      \draw[red,thick,->] (5.5,6) arc (180:360:0.5);
      \node[green] at (12,6) [above right] {\large $c$};
      \draw[green,thick] (12,6) ++(225:0.5) -- (A);
      \draw[green,thick,->] (11.5,6) arc (180:270:0.5);
      \draw[green,thick,->] (0,5.5) arc (-90:0:0.5);
      \node[blue] at (12,0) [below right] {\large $b$};
      \draw[blue,thick] (12,0) ++ (135:0.5) -- (A);
      \draw[blue,thick,->] (12,0.5) arc (90:180:0.5);
      \draw[blue,thick,->] (0.5,0) arc (0:90:0.5);
      \node[black] at (6,0) [below] {\large $d$};
      \draw[black,thick] (6,0) ++ (45:0.5) -- (A);
    \draw[black,thick,->] (6.5,0) arc (0:180:0.5);
    \end{scope}
    \draw[thick,->] (3,-0.5) -- node[below] {$i$} (5.5,-2.5);
    \draw[thick,->] (11.5,-0.5) -- node[right] {$f$} (11,-2.5);
\end{tikzpicture}\end{center}
\caption{Pilgrim's ``blowing up an arc'' subdivision rule}\label{fig:pilgrimmap}
\end{figure}
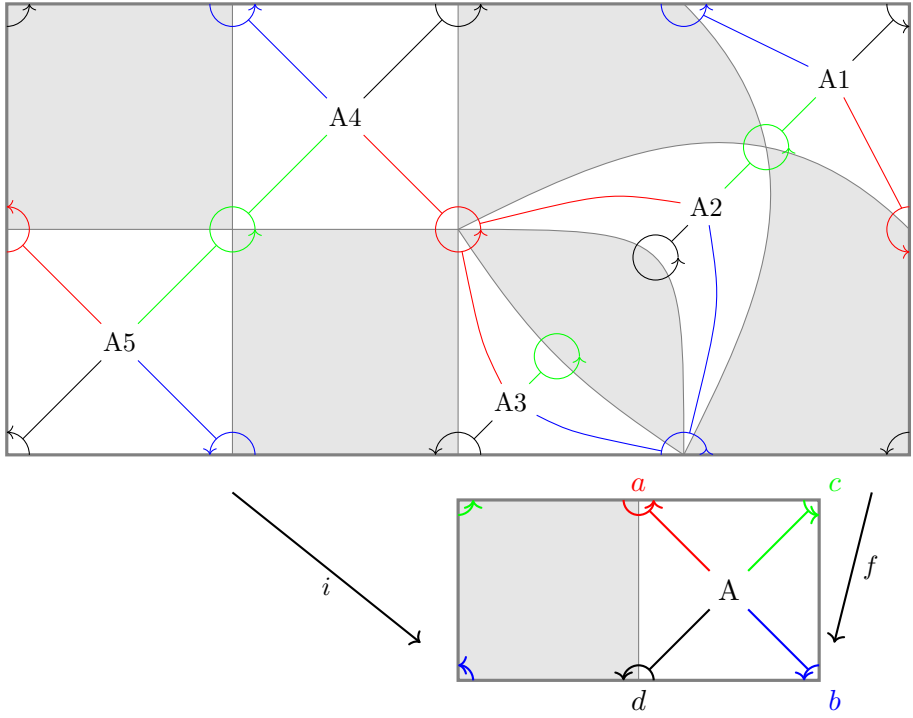

\subsubsection{Kevin Pilgrim's ``blowing up an arc'' map}\label{ss:pilgrim}
We consider another concrete example illustrating
\S\ref{ss:BelyaMaps}: a degree-$5$ self-map $g$, due to Kevin Pilgrim,
which is obtained from a torus endomorphism by blowing up an arc;
see~\cite{bartholdi-dudko:bc0}*{\S\ref{bc0:ex:pilgrim}}, repeated for
convenience in Figure~\ref{fig:pilgrimmap} as the correspondence
$(i,f)$. Note that we exchanged $b$ and $c$ compared
to~\cite{bartholdi-dudko:bc0}*{Figure~\ref{bc0:fig:pilgrimmap}}.

The map $g$ is non-dynamically modeled on the map
$f(z)=z^3((4z+5)/(5z+4))^2$ viewed as a Belyi map $(\hC,C)\to (\hC,A)$
with $A=\{0,1,\infty,v\}$ and $C=\{0,1,\infty, f^{-1}(v)\}$; with such
choices the map $\mathscr M'$ and $\mathscr M_C$ are identical via the
map in Diagram~\eqref{eq:diagr:ModCorr}. Recall that Pilgrim's map is
the map $g\colon(S^2,A)\selfmap$ obtained (topologically) from
$f\colon(\hC,C)\to(\hC,A)$ by identifying $A$ with four preimages
$c_1,\dots,c_4$ of $v$. We thus have a map $\sigma'\colon\mathscr
M_C\approx \mathscr M'\to \mathscr M_A$ given by evaluating the
cross-ratio of the first four coordinates of $\{c_1,\dots,c_5\}\in
\mathscr M'$, see Diagram~\eqref{eq:diagram pilgrim}.

Let $\mathscr M_A\overset{\overline{\sigma_g} }\longleftarrow\mathscr
W_f\overset{i }\longrightarrow\mathscr M_A$ be the correspondence
obtained from $\mathscr M_C\leftarrow\mathscr W_f\to\mathscr M_A$ in
Diagram~\eqref{eq:diagr:ModCorr} by precomposing with $\mathscr
M_A\overset{\sigma'}{\longleftarrow }\mathscr M_C$. The correspondence
we obtain is explicitly given by the pair of maps
\begin{align}
  i(c_1,\dots,c_5) &=f(c_1),\notag\\
  \sigma'(c_1,\dots,c_5) &= \frac{(c_1-c_3)(c_2-c_4)}{(c_1-c_4)(c_2-c_3)}\eqqcolon x;\label{eq:sigma'}
\end{align}
We shall see that it is the modular correspondence of $g$. On the
other hand, we shall also see that $\mathscr W_g$ is not the minimal
correspondence above $\mathscr M_A$ covered by the map $\sigma_g$.
Thus diagram~\eqref{eq:diagr:ModCorr} will become
\begin{equation}\label{eq:diagram pilgrim}
  \begin{tikzcd}
  \mathscr T_A\ar{ddd} & \mathscr T_C\ar{l}\ar{d} & & \mathscr T_A\ar{ll}{\sigma_{f}}\ar{dl}\ar{ddd}\\
  & \mathscr M_C\ar{ddl}[swap]{\sigma'} & \mathscr W_f=\mathscr W_g\ar[hook]{l}\ar{ddr}{5!:1}\ar{ddll}{\overline{\sigma_g}}\ar{d}\\
  & & \mathscr W\ar{dll}\ar{dr}[swap]{30:1}\\
  \mathscr M_A & & & \mathscr M_A.
\end{tikzcd}
\end{equation}

Let us compute the biset of the map $g$. With $\pi=\pi_1(\hC\setminus
A)$, it is the following $\pi$-$\pi$-biset $B(g)$, in the basis
$X=\{1,2,3,4,5\}$:
\begin{equation}
  \begin{aligned}
    a &\mapsto \pair{a,1,1,1,a^{-1}}(1,3,5)(2,4)\\
    b &\mapsto \pair{b^{-1},1,1,b,1}(1,4)(2,5,3)\\
    c &\mapsto \pair{1,1,c,c^{-1},1}(1,2)(3,4)\\
    d &\mapsto \pair{a,b,d,c,1}.
  \end{aligned}\label{eq:Fbiset}
\end{equation}

Let us write $G=\Mod(\hC,A)$. To compute the monodromy representation
$\rho$ of $f$, we simply consider the permutations in the biset
$B(g)$. We present $G$ as a sphere group
$\langle s,t,u\mid u t s\rangle$; the twists $s,t,u$ correspond to a
motion of $v$ around $0,\infty,1$ respectively. Therefore,
$\rho\colon G\to 5\perm$ is given by
\[s\mapsto(1,3,5)(2,4),\qquad t\mapsto(1,4)(2,5,3),\qquad
u\mapsto(1,2)(3,4).
\]
In particular, $\rho$ is onto $5\perm$ so the subgroup of liftables
$H_f=\ker(\rho)$ from~\eqref{eq:def:ModGr:G} has index $120$, and both
the mapping class biset $M(f,C,A)$ and the modular correspondence
$\mathscr W_f$ have degree $120$.

Let us consider the $G$-$G$-biset $M(g,A)\coloneqq\{B(\phi
g\psi)\}/{\sim}$ of isomorphism classes of twists of the biset
$B(g)$. This biset is left-free, and a basis may be found by
considering its distillations, see
Definition~\ref{defn:distillation}. Since $\rho(G)=5\perm$, its
centralizer in $5\perm$ is trivial, so the permutations associated
with the generators $a,b,c$ in the biset $B(\phi g\psi)$ can in a
unique manner be brought respectively to
$(1,3,5)(2,4),(1,4)(2,5,3),(1,2)(3,4)$. The conjugacy classes on the
five entries of $d$ in $B(\phi g\psi)$ are then $\one,a,b,c,d$ in some
order, and this order determines uniquely an element of $5\perm$. All
such orderings may appear, so $M(g,A)$ is also left-free of degree
$120$.

Using a choice of distillations, Algorithm~\ref{algo:compute M(f,C,A)}
lets us compute a presentation for $M(g,A)$, and in this manner
describe the correspondence. We shall not describe $M(g,A)$ in such a
basis, but will content ourselves with giving, for illustration, the
results of the calculation in the form of the lifts of the conjugacy
classes $s^G,t^G,u^G$.

Since $\mathscr M_A$ is a sphere punctured at $\{0,1,\infty\}$, the
modular correspondence is a sphere correspondence. Recall
from~\S\ref{ss:sphere bisets} that to each of the peripheral conjugacy
classes $s^G,t^G,u^G$ can be assigned the multiset $\{(d_i,h_i^G)\mid
i=1,\dots,\ell\}$ of its lifts.

Let us start with $u^G$. The degrees $d_i$ are all the same, and equal
to the order of $\rho(u)$, so there are $60$ cycles of length $2$,
given by the regular action of $\rho(u)$ on $5\perm$. The lifts are
$16\times s^G$, $16\times t^G$, $16\times u^G$, $4\times (s^2)^G$,
$4\times (t^2)^G$, $4\times (u^2)^G$. Similarly, the degrees of the
lifts of $s^G$ and $t^G$ are all $6$ since $\rho(s)$ and $\rho(t)$
have order $6$, and both $s^G$ and $t^G$ lift to $8\times 1^G$,
$4\times (s^5)^G$, $4\times(t^5)^G$ and $4\times (u^5)^G$. We deduce
that the degree of the map $\overline{\sigma_g}\colon\mathscr
W_f\to\mathscr M_A$ is $16\times1+4\times2+4\times5+4\times5=64$,
counting the number with degree and multiplicity of `$(s^n)^G$' that
appear as lifts of $s^G,t^G,u^G$. In particular, the correspondence is
not constant, and moreover is given by a pair of branched coverings
$\mathscr W_f\rightrightarrows\mathscr M_A$.

In this example, there exists a correspondence $\mathscr
M_A\leftarrow\mathscr W\rightarrow\mathscr M_A$ which is a quotient of
$\mathscr M_A\leftarrow\mathscr W_f\rightarrow\mathscr M_A$, and which
also covers $\sigma_g\colon\mathscr T_A\selfmap$. In fact, we shall
give the minimal such $\mathscr W$, and show that it has left-degree
$30$.

The variables $c_2,\dots,c_5$ may be eliminated from $\mathscr W_f$,
yielding a planar projection of the correspondence. It is given by a
polynomial equation $P(v,x)=0$ of degree $30$ in $x$ and $16$ in $v$,
see~\eqref{eq:sigma'}; so that the degree of the map towards
$v=f(c_1)$ is $30$, and the degree of the map towards
$x=[c_1,c_2;c_3,c_4]$ is $16$. The polynomial $P(v,x)$ may be computed
using Gr\"obner bases as follows, in \textsc{Maple} using the package
\textsc{FGb}\footnote{available from \texttt{http://www-polsys.lip6.fr/\char126jcf/FGb/}}:
\begin{verbatim}
with(FGb);
sys := [(c[1]-c[3])*(c[2]-c[4])-x*(c[1]-c[4])*(c[2]-c[3]),
  coeff(16*product(z-c[i],i=1..5)-z^3*(4*z+5)^2+v*(5*z+4)^2,z,i)$i=0..4]:
fgb_gbasis_elim(sys,0,[c[i]$i=1..5], [x,v],{"verb"=3,"index"=2000000}):
P := simplify(%[1]/v):
lprint(P);
\end{verbatim}
One can easily check that $P$ is irreducible, so
$\mathscr M_A\leftarrow\mathscr W\rightarrow\mathscr M_A$ is indeed
the minimal correspondence covered by $\sigma_g$.

Here is the explanation of the existence of a smaller cover $\mathscr
W$, using group theory. The minimality of $\mathscr W$ follows from
the fact that the induced map on conjugacy classes admits an order-$4$
symmetry and no larger one. The cross-ratio map $\sigma'$
from~\eqref{eq:sigma'} admits symmetries:
$\sigma'(c_1,c_2,c_3,c_4,c_5)=\sigma'(c_2,c_1,c_4,c_3,c_5)=\sigma'(c_3,c_4,c_1,c_2,c_5)=\sigma'(c_4,c_3,c_2,c_1,c_5)$. We
consider
\[V=\{1,(1,2)(3,4),(1,3)(2,4),(1,4)(2,3)\}
\]
the Klein group of order $4$, so that $\sigma'$ is invariant under
permutation of its arguments by $V$. We define $\mathscr W=\mathscr
T_A/\rho^{-1}(V)$; then the map $\overline{\sigma_g}\colon\mathscr
W_g\to\mathscr M_A$ descends to a map $\mathscr W\to\mathscr
M_A$. Thus $V$ acts on $\mathscr W_g$, and the quotient is $\mathscr
W$.

The Klein group $V$ acts regularly on the set of distillations of
bisets in $M(g,A)$; this action combines with the left action of $G$
on $M(g,A)$ to give a left action of $V\ltimes G$ on $M(g,A)$, by
impure mapping classes, which is still free, but now of degree
$30$. The quotient biset $V\backslash M(g,A)$ is a left-free
$G$-$G$-biset of degree $30$, and coincides with the biset $B(\mathscr
M_A\leftarrow\mathscr W\rightarrow\mathscr M_A)$ of the quotient
correspondence.

Recall that the map $g\colon(S^2,A)\selfmap$ was obtained from the
doubling map $z\mapsto 2z$ on the torus $\C/\Z+\Z i$, by blowing up an
arc. The Klein group $V$ acts on the torus by $\langle z\mapsto
z+\frac12,z\mapsto z+\frac i2\rangle$, and elements of the biset
$V\backslash M(g,A)$ can be identified with orbits of $V$ under this
action.

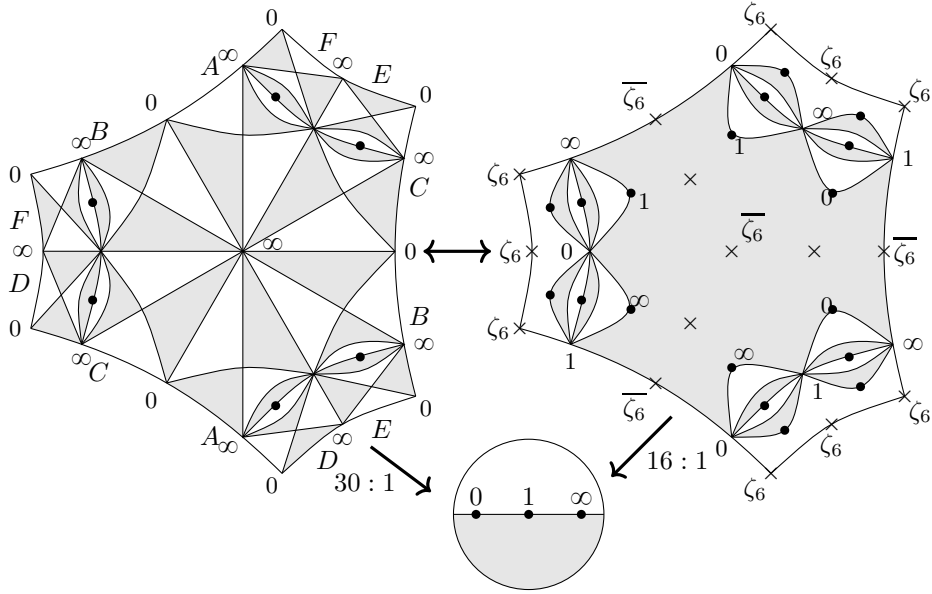
\begin{figure}
\begin{center}
\begin{tikzpicture}
  \gdef\pathSR{(160:3) .. controls (180:2.6) .. (200:3)}
  \gdef\pathRS{(-40:3) .. controls (-15:2) and (15:2) .. (40:3)}
  \gdef\surfacelabels{
    \foreach \a/\x/\y/\z/\t in {0/B/C/F/D,120/A/B/D/E,240/C/A/E/F} {
      \draw[rotate=\a] node at (-20:2.5) {$\x$};
      \draw[rotate=\a] node at (20:2.5) {$\y$};
      \draw[rotate=\a] node at (172:3.0) {$\z$};
      \draw[rotate=\a] node at (188:3.0) {$\t$};
    }
  }
  \begin{scope}
  \surfacelabels
  \gdef\surfacecoords{
  \foreach \a/\i in {0/0,120/1,240/2} {
    \begin{scope}[rotate=\a]
      \path[name path=v0] \pathSR coordinate (R\i);
      \path[name path=w0] (0,0) -- (-3,0);

      \path[name path=v1] \pathRS coordinate (S\i);
      \path[name path=w1] (0,0) -- (3,0);
      \path[name intersections={of=v0 and w0,by=p0},name intersections={of=v1 and w1,by=p1}] (p0) coordinate (P\i) -- (p1) coordinate (Q\i);

      \path[name path=w2] (0,0) -- (30:3);
      \path[name intersections={of=v1 and w2,by=p2}] (p2) coordinate (T\i);
      \path[name path=w2] (0,0) -- (-30:3);
      \path[name intersections={of=v1 and w2,by=p2}] (p2) coordinate (U\i);
    \end{scope}
  }
  \foreach \a/\i/\j in {0/0/1,120/1/2,240/2/0} {
    \begin{scope}[rotate=\a]
      \gdef\pathTU{(T\i) .. controls (60:1.8) .. (U\j)}
      \path[name path=v0] \pathTU;
      \path[name path=w0] (0,0) -- (60:3);
      \path[name intersections={of=v0 and w0,by=p0}] (p0) coordinate (V\i);
    \end{scope}
  }  
  }
  \surfacecoords
  \gdef\surfaceleft{
  \foreach \a/\i/\j/\k in {0/0/1/2,120/1/2/0,240/2/0/1} {
    \begin{scope}[rotate=\a]
      \gdef\pathQV{(Q\i) .. controls (30:1.5) .. (V\i)}
      \clip (0,0) -- \pathQV -- cycle;
      \fill[gray!20] (0,0) -- (T\i) -- (V\i) -- cycle;
    \end{scope}
    \begin{scope}[rotate=\a]
      \gdef\pathVQ{(V\i) .. controls (90:1.5) .. (Q\j)}
      \clip (0,0) -- \pathVQ -- cycle;
      \fill[gray!20] (0,0) -- (U\j) -- (Q\j) -- cycle;
    \end{scope}
    \begin{scope}[rotate=\a]
      \clip (0,0) -- \pathRS -- cycle;
      \clip (S\i) -- \pathQV -- cycle;
      \fill[gray!20] (0,0) -- (T\i) -- (Q\i) -- cycle;
    \end{scope}
    \begin{scope}[rotate=\a]
      \gdef\pathVTa{(V\i) .. controls (50:1.8) and (40:1.7) .. (T\i)}
      \clip (0,0) -- \pathTU -- cycle;
      \fill[gray!20] (S\i) -- \pathVTa -- cycle;
    \end{scope}
    \begin{scope}[rotate=\a]
      \gdef\pathVUa{(V\i) .. controls (70:1.8) and (80:1.7) .. (U\j)}
      \clip (0,0) -- \pathVUa -- cycle;
      \fill[gray!20] \pathVQ -- (U\j) -- cycle;
    \end{scope}
    \begin{scope}[rotate=\a]
      \gdef\pathVUb{(V\i) .. controls (75:2.3) .. (U\j)}
      \clip (0,0) -- \pathVUb -- cycle;
      \fill[gray!20] (P\k) -- \pathTU -- cycle;
    \end{scope}
    \begin{scope}[rotate=\a]
      \gdef\pathVTb{(V\i) .. controls (45:2.3) .. (T\i)}
      \clip (P\k) -- \pathVTb -- cycle;
      \fill[gray!20] (0,0) -- (V\i) -- (S\i) -- (Q\i) -- cycle;
    \end{scope}
    \begin{scope}[rotate=\a]
      \clip (V\j) -- \pathSR -- cycle;
      \fill[gray!20] (T\j) -- (P\i) -- (S\j) -- cycle;
    \end{scope}
    \begin{scope}[rotate=\a]
      \clip (V\j) -- \pathSR -- cycle;
      \fill[gray!20] (V\j) -- (P\i) -- (U\k) -- cycle;
    \end{scope}
    \begin{scope}[rotate=\a]
      \clip (V\k) -- \pathRS -- cycle;
      \fill[gray!20] (U\i) -- (P\j) -- (R\j) -- cycle;
    \end{scope}
  }
  \foreach \a/\i/\j/\k/\c in {0/0/1/2/,120/1/2/0/,240/2/0/1/} {
    \begin{scope}[rotate=\a]
      \draw \pathRS;
      \draw \pathSR;
      \draw (P\i) -- (Q\i);
      \draw (0,0) -- \pathTU -- cycle;
      \draw \pathVUa;
      \draw \pathVUb;
      \draw \pathVTa;
      \draw \pathVTb;
      \draw \pathVQ;
      \draw \pathQV;
      \draw (V\i) -- (R\k);
      \draw (V\i) -- (S\i);
      \draw (U\j) -- (P\k);
      \draw (T\i) -- (P\k);
      \fill (42:2.1) circle (0.06);
      \fill [\c] (78:2.1) circle (0.06);
    \end{scope}
  }
  \node at (0.4,0.1) {\small $\infty$};

  \node[right] at (U0) {\small $\infty$};
  \node[right] at (Q0) {\small $0$};
  \node[right] at (T0) {\small $\infty$};
  \node[above right=-1mm] at (S0) {\small $0$};
  \node[above] at (P2) {\small $\infty$};
  \node[above left=-1mm] at (R2) {\small $0$};
  \node[above left=-1mm] at (U1) {\small $\infty$};
  \node[above left] at (Q1) {\small $0$};
  \node[above] at (T1) {\small $\infty$};
  \node[left] at (S1) {\small $0$};
  \node[left] at (P0) {\small $\infty$};
  \node[left] at (R0) {\small $0$};
  \node[below] at (U2) {\small $\infty$};
  \node[below left] at (Q2) {\small $0$};
  \node[below left=-1mm] at (T2) {\small $\infty$};
  \node[below left=-1mm] at (S2) {\small $0$};
  \node[below] at (P1) {\small $\infty$};
  \node[below right=-1mm] at (R1) {\small $0$};
  }
  \surfaceleft
\end{scope}

\begin{scope}[shift={(6.5,0)}]
  \surfacecoords
  \foreach \a/\i/\j/\k in {0/0/1/2,120/1/2/0,240/2/0/1} {
    \begin{scope}[rotate=\a]
      \clip (0,0) -- \pathTU -- cycle;
      \fill[gray!20] (S\i) -- \pathVTa -- cycle;
    \end{scope}
    \begin{scope}[rotate=\a]
      \clip (0,0) -- \pathTU -- cycle;
      \fill[gray!20] (R\k) -- \pathVUa -- cycle;
    \end{scope}
    \begin{scope}[rotate=\a]
      \gdef\pathVTd{(V\i) .. controls (48:2.6) .. (T\i)}
      \clip (0,0) -- \pathVTd -- cycle;
      \fill[gray!20] (P\k) -- \pathVTb -- (S\i) -- cycle;
    \end{scope}
    \begin{scope}[rotate=\a]
      \gdef\pathVUd{(V\i) .. controls (72:2.6) .. (U\j)}
      \clip (0,0) -- \pathVUd -- cycle;
      \fill[gray!20] (P\k) -- \pathVUb -- (R\k) -- cycle;
    \end{scope}
  }
  \gdef\pathVUc{(V\i) .. controls (95:1.4) .. (U\j)}
  \gdef\pathVTc{(V\i) .. controls (25:1.4) .. (T\i)}
  \begin{scope}
    \foreach \a/\i/\j/\k in {0/0/1/2,120/1/2/0,240/2/0/1} {
      \clip[rotate=\a] \pathVUc -- (T\j) -- (U\k) -- (T\k) -- (U\i) -- (T\i) -- cycle;
      \clip[rotate=\a] \pathVTc -- (U\i) -- (T\k) -- (U\k) -- (T\j) -- (U\j) -- cycle;
      \clip[rotate=\a] \pathRS -- (R\k) -- (S\j) -- (R\i) -- (S\k) -- cycle;
    }
    \fill[gray!20] (0,0) circle (3);
  \end{scope}
  \foreach \a/\i/\j/\k/\c in {0/0/1/2/,120/1/2/0/,240/2/0/1/} {
    \begin{scope}[rotate=\a]
      \draw \pathRS;
      \draw \pathSR;
      \draw \pathTU;
      \draw \pathVUa;
      \draw \pathVUb;
      \draw \pathVUc;
      \draw \pathVUd;
      \draw \pathVTa;
      \draw \pathVTb;
      \draw \pathVTc;
      \draw \pathVTd;
      \fill (42:2.1) circle (0.06);
      \fill [\c] (78:2.1) circle (0.06);
      \path[name path=v0] \pathVQ; \path[name path=w0] (0,0) -- (U\j);
      \fill[name intersections={of=v0 and w0,by=p0}] (p0) circle (0.06);
      \path[name path=v0] \pathQV; \path[name path=w0] (0,0) -- (T\i);
      \fill[name intersections={of=v0 and w0,by=p0}] (p0) circle (0.06);
      \path[name path=v0] (U\j) -- (P\k); \path[name path=w0] (R\k) -- (V\i);
      \fill[name intersections={of=v0 and w0,by=p0}] (p0) circle (0.06);
      \path[name path=v0] (T\i) -- (P\k); \path[name path=w0] (S\i) -- (V\i);
      \fill[name intersections={of=v0 and w0,by=p0}] (p0) circle (0.06);
      \node at (Q\i) {$\times$};
      \node at (P\i) {$\times$};
      \node at (S\i) {$\times$};
      \node at (R\i) {$\times$};
      \node at (0:1.1) {$\times$};
    \end{scope}
  }
  \node at (0,0) {$\times$};
  \node[above right] at (0,0) {$\overline{\zeta_6}$};

  \node[right] at (U0) {\small $\infty$};
  \node[right] at (Q0) {$\overline{\zeta_6}$};
  \node[right] at (T0) {\small $1$};
  \node[above right=-1mm] at (S0) {$\zeta_6$};
  \node[above] at (P2) {$\zeta_6$};
  \node[above left=-1mm] at (R2) {$\zeta_6$};
  \node[above left=-1mm] at (U1) {\small $0$};
  \node[above left] at (Q1) {$\overline{\zeta_6}$};
  \node[above] at (T1) {\small $\infty$};
  \node[left] at (S1) {$\zeta_6$};
  \node[left] at (P0) {$\zeta_6$};
  \node[left] at (R0) {$\zeta_6$};
  \node[below] at (U2) {\small $1$};
  \node[below left] at (Q2) {$\overline{\zeta_6}$};
  \node[below left=-1mm] at (T2) {\small $0$};
  \node[below left=-1mm] at (S2) {$\zeta_6$};
  \node[below] at (P1) {$\zeta_6$};
  \node[below right=-1mm] at (R1) {$\zeta_6$};

  \node[above right] at (25:1.17) {\small $0$};
  \node[above right] at (95:1.17) {\small $1$};
  \node[left] at (145:1.17) {\small $1$};
  \node[left] at (215:1.17) {\small $\infty$};
  \node[below right] at (265:1.17) {\small $\infty$};
  \node[below right] at (335:1.17) {\small $0$};
  \node[above right] at (V0) {\small $\infty$};
  \node[left=1mm] at (V1) {\small $0$};
  \node[below right] at (V2) {\small $1$};
\end{scope}

\begin{scope}[shift={(3.8,-3.5)}]
  \begin{scope}
    \clip (-1,0) rectangle (1,-1);
    \fill[gray!20] (0,0) circle (1);
  \end{scope}
  \draw (0,0) circle (1);
  \draw (-1,0) -- (1,0);
  \fill (-0.7,0) circle (0.06) node[above] {$0$};
  \fill (0,0) circle (0.06) node[above] {$1$};
  \fill (0.7,0) circle (0.06) node[above] {$\infty$};
\end{scope}

\draw[<->,very thick] (2.4,0) -- (3.3,0);
\draw[->,very thick] (1.7,-2.6) -- node[below left=-1mm] {$30:1$} (2.5,-3.2);
\draw[->,very thick] (5.7,-2.2) -- node[below right=-1mm] {$16:1$} (4.9,-3.0);

\end{tikzpicture}
\end{center}
\caption{The modular correspondence of $g$.}
\label{fig:pilgrimcorr}
\end{figure}

We repeat the analysis of the action of peripheral elements on
$M(g,A)/V$, and derive information on $\mathscr W$.  It is readily
checked that the elements $s,t$ act on the coset space $5\perm/V$ as a
product of five $6$-cycles, each carrying the conjugacy classes
$(s^5)^G,(t^5)^G,(u^5)^G,1^G,1^G$; and that $u$ acts as a product of
twelve involutions (with six fixed points).

Therefore, the correspondence $\mathscr W$ has five degree-$6$
punctures above $0$, mapping respectively by degree $5$ to
$0,1,\infty$ and two other points (the sixth roots of unity); it has
similarly five degree-$6$ punctures above $\infty$, mapping
respectively by degree $5$ to $0,1,\infty$ and the sixth roots of
unity; and eighteen punctures above $1$, twelve of degree $2$ and six
of degree $1$, mapping by degree $1$ to $0,1,\infty$ six times
each. The surface $\mathscr W$ has therefore $28$ punctures and its
Euler characteristic is $30\cdot(2-3)=-30$, so $\mathscr W$ has genus
$2$. The correspondence is given in Figure~\ref{fig:pilgrimcorr}, in
the standard description of holomorphic maps by shading the lower half
plane.

The central component of the figure on the right is homeomorphic to a
punctured torus; it maps $4:1$ to the disk, with the boundary mapping
$4:1$, one point (the sixth root of unity $\overline{\zeta_6}$) having
two order-$2$ preimages, and three points having one order-$2$
preimage (marked by a simple $\times$) and two regular preimages.

\subsection{A Thurston map with infinitely generated centralizer}\label{ss:complicated centralizer}
Our last example shows that centralizers of Thurston maps can be
sometimes quite complicated, and in particular not finitely generated
(whence our notion of ``sub-computable''). We will also explicitly
compute a biset of the form $M(f,A,\CC)$.

\subsubsection{General construction}
We consider a Thurston map $f\colon(S^2,A,\CC)\selfmap$ with
$A=A_0\sqcup A_1\sqcup A_2$. The map admits an annular obstruction
$\CC=\{s,t\}$, with $s$ separating $A_0$ from $A_1\cup A_2$ and $t$
separating $A_2$ from $A_0\cup A_1$. The small spheres $S_0,S_1,S_2$
containing $A_0,A_1,A_2$ respectively are fixed by $f$, which acts on
$S_0$ and $S_2$ as the identity and acts on $S_1$ as a rational map of
degree $2$.

All $f$-preimages of $s$ and $t$ map by degree $1$. The curve $s$ has
a unique essential preimage that is isotopic to $s$ while $t$ has
$2\ell\ge 2$ preimages isotopic to $s$ and $2\ell+1\ge 3$ preimages
isotopic to $t$. The Thurston matrix of $f$ is therefore
\[T_{f,\CC}=\begin{pmatrix}1 & 2\ell\\ 0 & 2\ell+1\end{pmatrix}.
\]
The case $\ell=1$ is show in Figure~\ref{fig:infinitely generated}.
  
We assume that all trivial $f$-preimages of $S_0,S_1,S_2$ are of
degree $1$ and all annular $f$-preimages of $S_1,S_2$ are of degree
$2$. (We remark that $S_0$ has no annular preimages.) The degree of
$f$ is thus $4\ell+2$.

The spheres $S_1$ has $\ell$ annular preimages isotopic to $s$ and
$\ell$ annular preimages isotopic to $t$. We denote them by
$U_1,\dots,U_{2\ell}$ in order of increasing distance to
$S_0$. Similarly, $S_2$ has $2\ell$ annular preimages isotopic to $s$
or $t$, denoted $T_1,\dots,T_{2\ell}$ in order of increasing distance
to $S_0$.

Let us next assume that $\#A_1=2$ and that $A_1$ is fixed by
$f$. Therefore, we may normalize the restriction of $f$ to $S_1$ as
$z^2\colon (\hC, \{0,\infty,1,-1\})\selfmap$ so that $1$ and $-1$
correspond to curves $t$ and $s$ respectively. For $i>1$ we normalize
$f\colon \widehat U_i\to \widehat S_1$ as
$z^2\colon (\hC, \{1,-1\})\to (\hC, \{0,\infty,1,-1\})$ so that
$1,-1\in \widehat U_i$ encode curves isotopic to $s$ if $i\le \ell$,
and isotopic to $t$ if $i>\ell$. Finally we normalize
$f\colon \widehat U_1\to \widehat S_1$ as
$z^2\colon (\hC, \{1,i\})\to (\hC, \{0,\infty,1,-1\})$ so that
$1,i\in \widehat U_i$ encode curves isotopic to $s$.

Recall that the restrictions of $f$ to $S_0$ and $S_2$ are set to be
the identity. We assume $\#A_2\ge 3$ and we mark the points
$\infty, x_6,x_7\in A_2$. For $i\le \ell$ we normalize
$f\colon \widehat T_i\to\widehat S_2$ as
$\frac{z^2+x_6}{1+x_6}\colon (\hC,\{1,-1\})\to (\hC,A_2\cup \{1\})$ so
that $x_6, \infty\in A_2$ are the critical values. The points
$1,-1\in \widehat T_i$ encode curves isotopic to $s$ while
$1\in \widehat S_2$ encodes the curve $t$. For $i> \ell$ we normalize
$f\colon\widehat T_i\to \widehat S_2$ as
$\frac{z^2+x_7}{1+x_7}\colon (\hC,\{1,-1\})\to (\hC,A_2\cup \{1\})$ so
that $x_7, \infty\in A_2$ are the critical values. The points
$1,-1\in \widehat T_i$ encodes curves isotopic to $t$ and
$1\in \widehat S_2$ encodes the curve $t$.
 
\subsubsection{The mapping class biset $M(f,A,\CC)$}
We describe $M(f,A,\CC)$ following the recipe
of~\S\ref{ss:DecompOfMCB}.  Write
$\Mod(S^2,A,\CC)=\Mod[e](S^2,A,\CC)\times\Mod[v](S^2,A,\CC)$ with
$\Mod[e](S^2,A,\CC)\cong \Z^{\CC}$ and
$\Mod[v](S^2,A,\CC)\cong \Mod(\widehat S_0)\times \Mod(\widehat
S_1)\times\Mod(\widehat S_2)$.

Let us first compute $G_f\le \Mod[v](S^2,A,\CC)$,
see~\eqref{eq:defn:G_f}. Write
$G_{f}=G_{f,0}\times G_{f,1}\times G_{f,2}$ with
$G_{f,i}\le \Mod(\widehat S_i)$. Since
$f\colon (\widehat S_1, \{0,\infty, 1,-1\})\selfmap$ factors as
$ (\widehat S_1, \{0,\infty, 1,-1\})\overset{z^2}{\longrightarrow}
(\widehat S_1, \{0,\infty, 1\}) \overset{\one}{\longrightarrow}
(\widehat S_1, \{0,\infty, 1,-1\})$, every element in
$\Mod(\widehat S_1)$ is liftable through
$f\colon (S_1, \{0,\infty, 1,-1\})\selfmap$; the lift is trivial in
$\Mod(\widehat S_1)$ but needs not be trivial in
$\Mod(S^2,A,\CC)$. Namely, a Dehn twist about a curve surrounding
$1,-1$ lifts to a Dehn twist around $1$ which encodes $t$. By the same
reasoning, every element in $\Mod(\widehat S_1)$ is liftable trough
$f\colon (\widehat U_i,\{1,-1\})\to (\widehat S_1, \{0,\infty, 1,-1\})$ for $i>1$. On
the other hand, the subgroup of liftable elements in
$\Mod(\widehat S_1)$ under
$f\colon (\widehat U_1,\{1,i\})\to (\widehat S_2,\{0,\infty, 1,-1\})$
has index two in $\Mod(\widehat S_1)$: a Dehn twist about a simple
essential closed curve in $(\widehat S_2,\{0,\infty, 1,-1\})$
separating $1$ from $-1$ lifts to a half twist about the simple closed
curve in $(\widehat U_1,\{1,i\})$ separating $1$ from $i$; this gives
an index two condition. Therefore, $G_{f,1}\le \Mod(\widehat S_1)$ has
index two.

Since
$f\colon (\widehat T_i, \{ 1,-1\})\to (\widehat S_2,A_2\cup \{1\})$
factors as
$ (\widehat T_i, \{0,\infty, 1,-1\})\overset{f}{\longrightarrow}
(\widehat S_2, \{f(0),\infty, 1\}) \overset{\one}{\longrightarrow}
(\widehat S_2, A_2\cup \{ 1\})$, every element in $\Mod(\widehat S_2)$
is liftable trough
$f\colon (\widehat T_i,\{1,-1\})\to (\widehat S_2, A_2\cup \{
1\})$. Since the restrictions of $f$ to $S_0$ and to $S_2$ are the
identity, every element in $\Mod(\widehat S_0)$ and in
$\Mod(\widehat S_2)$ is liftable through the global map
$f\colon (S^2,A,\CC)\selfmap$. Therefore, $G_{f,0}=\Mod(\widehat S_0)$
and $G_{f,2}=\Mod(\widehat S_2)$.

Since all $f$-preimages of $s$ and $t$ map by degree $1$, the group
$\Lambda$ from~\eqref{eq:defn:Lambda} is $\Z^{\CC}$ with the actions
of $\Z^{\CC}$ given by $n_1\cdot b\cdot n_2=
n_1+b+T_{f,\CC}(n_2)$. The map
$\theta_f\colon G_{f,0}\times G_{f,1}\times G_{f,2}\to \Lambda
\cong\Z^{\CC}$ is computed as
\begin{equation}
\label{eq:Ex:Defn:theta_f}
\theta_f(n_0,n_1,n_2) =(\theta'_1(n_1)+(2\ell-1)\theta_{1}(n_1)+ 2\ell \theta_{2,s}(n_2), (2\ell+1)\theta_{1}(n_1)+2\ell \theta_{2,t}(n_2))\in \Z^{\{s,t\}}
\end{equation}
with $\theta'_1,\theta_1, \theta_{2,s},\theta_{2,t}$ defined as follows.

The map $\theta'_1\colon G_{f,1}\to \Z$ corresponds to lifting through
$f\colon (\widehat U_1,\{1,i\})\to (\widehat S_1,\{0,\infty,
1,-1\})$. If $n$ be a positive Dehn twist about an essential simple
closed curve $c$ in $(\widehat S_2,\{0,\infty, 1,-1\})$, then
$\theta'_1(n)=0$ if $c$ does not separate $1$ from $-1$, and
$\theta'_1(n^2)=1$ otherwise. (Recall that in the last case
$n\notin G_{f,1}$).

The map $\theta_{1}\colon G_{f,1}\to \Z$ is the restriction of
$\theta_{1}\colon \Mod(\widehat S_1)\to \Z$ defined in the following
way.  If $n$ be a Dehn twist about an essential simple closed curve
$c$ in $(\widehat S_1,\{0,\infty, 1,-1\})$, then $\theta_{1}(n)=1$ if
$c$ is a peripheral curve around $1$ in
$(\widehat S_1,\{0,\infty, 1\})$, and $\theta_{1}(n)=0$ otherwise. In
the former case $c$ has exactly $2\ell-1$ lifts isotopic to $s$ (not
counting lifting trough
$f\colon (\widehat U_1,\{1,i\})\to (\widehat S_1,\{0,\infty, 1,-1\})$)
and $2\ell+1$ lifts isotopic to $t$. In the latter case none of the
lifts of $c$ is isotopic to a curve in $\{s,t\}$.

Suppose $n$ is a Dehn twist about a curve $c$ in
$(\widehat S_2, A_2\cup\{1\})$. If $c$ be a peripheral curve in
$(\widehat S_2,\{x_6,\infty, 1\})$ around $1$, then
$\theta_{2,s}(n)=1$; otherwise $\theta_{2,s}(n)=0$. In the former case
$c$ has exactly $2\ell$ lifts isotopic to $s$. Similarly, if $c$ is a
peripheral curve in $(\widehat S_2,\{x_7,\infty, 1\})$ around $1$,
then $\theta_{2,t}(n)=1$; otherwise $\theta_{2,s}(n)=0$. In the former
case $c$ has exactly $2\ell$ lifts isotopic to $t$.

Denote by $\subscript{\Mod(\widehat S_1)}M'_{G_{f,1}}$ the mapping
class biset of $f\colon \widehat S_1\selfmap$ with the right action
restricted to $G_{f,1}$. It was already shown that the action of
$G_{f,1}$ on $M'$ is trivial; thus
\[\subscript{\Mod(\widehat S_1)}M'_{G_{f,1}}\cong \subscript{\Mod(\widehat S_1)}\Mod(\widehat S_1)_\one \otimes \subscript{\one}\{\cdot\}_{G_{f,1}}.  \]
Then the decomposition in Theorem~\ref{thm:M
  xproduct} takes the form
\[\subscript{\Mod(S^2,A,\CC)}(\Z^{\CC}\ltimes \Mod(\widehat S_0) \times M'\times \Mod(\widehat S_2))_{\Mod[e](S^2,A,\CC)\times G_f}\otimes \Mod(S^2,A,\CC).\] 

\subsubsection{Computation of the centralizer}
Since the spectrum radius of $\CC$ is greater than $1$, the canonical
obstruction of $f$ contains $\CC$; but since $\CC$ separates $f$ into
rational and finite order maps, $\CC$ is the canonical
obstruction~\cite{pilgrim:combinations}. (Note that every curve in
$S_0$ or $S_2$ is a Levy cycle; but such curves are not part of the
canonical obstruction.) Therefore, $Z(f)$ fixes $\CC$ and we may write
\[Z(f)\le \Mod(S^2,A,\CC)\cong \Mod(\widehat S_0)\times \Mod(\widehat S_1)\times
  \Mod(\widehat S_2)\times \Z^{\CC}.
\] 
Furthermore, the projection of $Z(f)$ into $\Mod(\widehat S_1)$ is trivial
because the restriction of $f$ to $S_1$ is a rational
self-map. Therefore, $Z(f)$ is a subgroup of $\Mod(S_0)\times \Mod(S_2)\times\Z^\CC$. 

Consider an element
$(n_0,n_1,n_2,v)\in \Mod(\widehat S_0)\times\Mod(\widehat
S_1)\times\Mod(\widehat S_2)\times\Z^\CC$. Then
$(n_0,n_1,n_2,v)\in Z(f)$ if and only if
$(n_0,n_1,n_2,v)\cdot f=f\cdot (n_0,n_1,n_2,v)$; namely, if $n_1=\one$
and $v=\theta(n_0,n_1,n_2)+T_{f,\CC}(v)$. Writing $v=(v_s,v_t)$, we
obtain the equations $\theta_{2,s}(n_2)=\theta_{2,t}(n_2)=-v_t$.
Therefore, $Z(f)$ is isomorphic to
\[\Mod(\widehat S_0)\times \Z^{\{s\}}\times  \{n\in \Mod(\widehat S_2)\mid \theta_{2,s}(n_2)= \theta_{2,t}(n_2)\}.
\]
It is easy to see that the last factor is infinitely generated, since
it is the kernel of an epimorphism from $\Mod(\widehat S_2)$ to $\Z$. We
compute it explicitly below in one case.

\subsubsection{The case $\ell=1$ and $\#A=7$} 
These seem to produce the Thurston map of smallest degree and size of
critical set whose centralizer is infinitely generated. The map $f$
has degree $6$ and has $7$ marked points, see
Figure~\ref{fig:infinitely generated}. They are labeled as
$A_0=\{x_3,x_4\}$, $A_1=\{x_2,x_5\}$ and $A_2=\{x_1,x_6,x_7\}$.

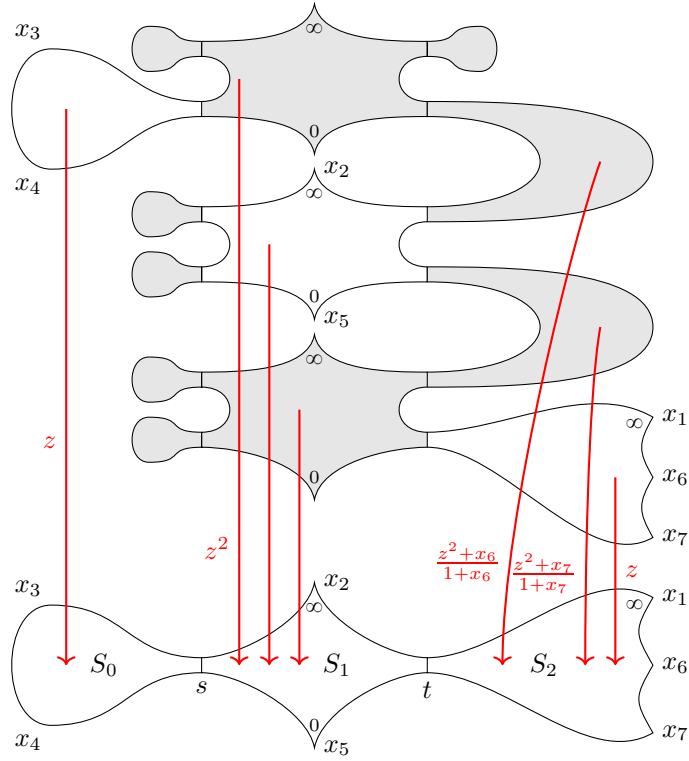
\begin{figure}
  \begin{center}
    \begin{tikzpicture}
      \def\smallsphere{.. controls +(180:0.5) and +(0:0.5) .. +(-0.7,0.2)
        .. controls +(180:0.3) and +(180:0.3) .. +(-0.7,-0.4)
        .. controls +(0:0.5) and +(180:0.5) .. +(0,-0.2) -- +(0,0.0) -- +(0,-0.2)}
      \def\smallrightsphere{-- +(0,0.2) .. controls +(0:0.5) and +(180:0.5) .. +(0.7,0.4)
        .. controls +(0:0.3) and +(0:0.3) .. +(0.7,-0.2)
        .. controls +(180:0.5) and +(0:0.5) .. +(0,0)}
      \draw (-3,7.5) .. controls +(180:1) and +(0:1) .. +(-2,0.7) node[above left] {$x_3$}
      .. controls +(180:0.7) and +(180:0.7) .. +(-2,-0.9) node[below left] {$x_4$}
      .. controls +(0:1) and +(180:1) .. ++(0,-0.2);
      \draw[fill=gray!20] (0,7.5) .. controls +(180:0.5) and +(180:0.5) .. ++(0,0.6) \smallrightsphere --
      ++(0,0.2) .. controls +(180:0.5) and +(-80:0.5) .. +(-1.5,0.5) node[below=4pt] {\scriptsize $\infty$}
      .. controls +(-100:0.5) and +(0:0.5) .. ++(-3,0) \smallsphere -- ++(0,-0.2)
      .. controls +(0:0.5) and +(0:0.5) .. ++(0,-0.6) -- ++(0,-0.2)
      .. controls +(0:0.5) and +(100:0.5) .. +(1.5,-0.5) node[above=3pt] {\scriptsize $0$}
      .. controls +(80:0.5) and +(180:0.5) .. ++(3,0) -- +(0,0.2);
      \draw[fill=gray!20] (0,7.5) .. controls +(0:1) and +(90:0.8) .. +(3,-0.8)
      .. controls +(-90:0.8) and +(0:1) .. ++(0,-1.6) +(0,0.2) .. controls +(0:1) and +(-90:0.3) .. +(1.5,0.8)
      .. controls +(90:0.3) and +(0:1) .. ++(0,1.4) -- ++(0,0.2);	   
      \draw (0,5.3) .. controls +(180:0.5) and +(180:0.5) .. ++(0,0.6) --
      ++(0,0.2) .. controls +(180:0.5) and +(-80:0.5) .. +(-1.5,0.5) node[below=4pt] {\scriptsize $\infty$} node[right] {$x_2$}
      .. controls +(-100:0.5) and +(0:0.5) .. ++(-3,0) -- ++(0,-0.2)
      .. controls +(0:0.5) and +(0:0.5) .. ++(0,-0.6) -- ++(0,-0.2)
      .. controls +(0:0.5) and +(100:0.5) .. +(1.5,-0.5) node[above=3pt] {\scriptsize $0$} node[right] {$x_5$}
      .. controls +(80:0.5) and +(180:0.5) .. ++(3,0) -- +(0,0.2);
      \draw[fill=gray!20] (-3,5.3) \smallsphere ++(0,0.8) \smallsphere
      (0,5.3) .. controls +(0:1) and +(90:0.8) .. +(3,-0.8)
      .. controls +(-90:0.8) and +(0:1) .. ++(0,-1.6) +(0,0.2) .. controls +(0:1) and +(-90:0.3) .. +(1.5,0.8)
      .. controls +(90:0.3) and +(0:1) .. ++(0,1.4) -- ++(0,0.2);
      \draw[fill=gray!20] (0,3.1) .. controls +(180:0.5) and +(180:0.5) .. ++(0,0.6) --
      ++(0,0.2) .. controls +(180:0.5) and +(-80:0.5) .. +(-1.5,0.5) node[below=4pt] {\scriptsize $\infty$}
      .. controls +(-100:0.5) and +(0:0.5) .. ++(-3,0) \smallsphere -- ++(0,-0.2)
      .. controls +(0:0.5) and +(0:0.5) .. ++(0,-0.6) \smallsphere -- ++(0,-0.2)
      .. controls +(0:0.5) and +(100:0.5) .. +(1.5,-0.7) node[above=3pt] {\scriptsize $0$}
      .. controls +(80:0.5) and +(180:0.5) .. ++(3,0) -- ++(0,0.2);
      \draw (0,2.9)
      .. controls +(0:1) and +(-150:1) .. +(3,-1.2) node[right] {$x_7$}
      .. controls +(120:0.4) and +(-120:0.4) .. +(3,-0.4) node[right] {$x_6$}
      .. controls +(120:0.4) and +(-120:0.4) .. +(3,0.4) node[right] {$x_1$} node[below left=-1mm and 0mm] {\scriptsize $\infty$}
      .. controls +(150:1) and +(0:1) .. ++(0,0.2);
      \draw (0,0.1) .. controls +(180:0.5) and +(-80:0.5) .. +(-1.5,1) node[below=4pt] {\scriptsize $\infty$} node[right] {$x_2$}
      .. controls +(-100:0.5) and +(0:0.5) .. ++(-3,0)
      .. controls +(180:1) and +(0:1) .. +(-2,0.7) node[above left] {$x_3$}
      .. controls +(180:0.7) and +(180:0.7) .. +(-2,-0.9) node[below left] {$x_4$}
      .. controls +(0:1) and +(180:1) .. ++(0,-0.2) node[below] {$s$} -- +(0,0.2) ++(0,0)
      .. controls +(0:0.5) and +(100:0.5) .. +(1.5,-1) node[above=3pt] {\scriptsize $0$} node[right] {$x_5$}
      .. controls +(80:0.5) and +(180:0.5) .. ++(3,0) node[below] {$t$} -- +(0,0.2) ++(0,0)
      .. controls +(0:1) and +(-150:1) .. +(3,-0.8) node[right] {$x_7$}
      .. controls +(120:0.5) and +(-120:0.5) .. +(3,0.1) node[right] {$x_6$}
      .. controls +(120:0.5) and +(-120:0.5) .. +(3,1.0) node[right] {$x_1$} node[below left=-1mm and 0mm] {\scriptsize $\infty$}
      .. controls +(150:1) and +(0:1) .. ++(0,0.2);
      \draw[<-,red,thick] (-4.8,0) -- node[left,pos=0.4] {$z$} +(0,7.4);
      \node at (-4.3,0) {$S_0$};
      \draw[<-,red,thick] (-2.5,0) -- node[left,pos=0.2] {$z^2$} +(0,7.8);
      \draw[<-,red,thick] (-2.1,0) -- +(0,5.6);
      \draw[<-,red,thick] (-1.7,0) -- +(0,3.4);
      \node at (-1.2,0) {$S_1$};
      \draw[<-,red,thick] (1.0,0) .. controls +(90:2) and +(-110:1) .. node[left,pos=0.2] {$\frac{z^2+x_6}{1+x_6}$} +(1.3,6.7);
      \node at (1.55,0) {$S_2$};
      \draw[<-,red,thick] (2.1,0) .. controls +(90:1) and +(-100:1) .. node[left,pos=0.3] {$\frac{z^2+x_7}{1+x_7}$} +(0.2,4.5);
      \draw[<-,red,thick] (2.5,0) -- node[right,pos=0.5] {$z$} +(0,2.5);
    \end{tikzpicture}
  \end{center}
  \caption{A Thurston map with infinitely generated centralizer}\label{fig:infinitely generated}
\end{figure}

To compute the centralizer of $f$, we write down a presentation of $B(f)$, and
compute some relations in its mapping class biset. We set
\[G=\langle x_1,x_2,x_3,x_4,x_5,x_6,x_7\mid x_1x_2x_3x_4x_5x_6x_7\rangle,\]
write $s=x_3x_4$ and $t=x_2x_3x_4x_5$, and in a basis
$\{\ell_1,\dots,\ell_7\}$ we compute the presentation
\begin{equation}\label{eq:infinite centralizer}
  \begin{aligned}
  x_1 &=\pair{1,s,s^{-1},t,t^{-1},x_1}(2,3)(4,5),\\
  x_2 &=\pair{1,1,s^{-1},x_2s,t^{-1},t}(1,2)(3,4)(5,6),\\
  x_3 &=\pair{x_3,1,1,1,1,1},\\
  x_4 &=\pair{x_4,1,1,1,1,1},\\
  x_5 &=\pair{1,1,x_5,1,1,1}(1,2)(3,4)(5,6),\\
  x_6 &=\pair{1,1,1,1,1,x_6}(2,3),\\
  x_7 &=\pair{1,1,1,1,1,x_7}(4,5),
\end{aligned}
\end{equation}
giving $s=\pair{s,1,1,1,1,1}$ and $t=\pair{1,s,s^{-1},t,t^{-1},t}$.
We write $\sigma,\tau,\alpha,\beta$ for Dehn twists about $s,t,x_1x_6$
and $x_6x_7$ respectively; their actions on $G$ are given respectively
by
\begin{align*}
  \sigma:\qquad & x_3\mapsto x_3^s,\;x_4\mapsto x_4^s,\\
  \tau:\qquad & x_2\mapsto x_2^t,\;x_3\mapsto x_3^t,\;x_4\mapsto x_4^t,\;x_5\mapsto x_5^t,\\
  \alpha:\qquad & x_1\mapsto x_1^{t x_6t^{-1}},\;x_6\mapsto x_6^{t^{-1}x_1t x_6},\\
  \beta:\qquad & x_6\mapsto x_6^{x_6x_7},\;x_7\mapsto x_7^{x_6x_7},
\end{align*}
all other generators being fixed. Naturally
$[\sigma,\alpha]=[\tau,\alpha]=[\sigma,\beta]=[\tau,\beta]=1$ while
$\langle\alpha,\beta\rangle$ is a free group of rank $2$. We then compute
\begin{xalignat*}{2}
  B(f)\cdot\sigma&\cong\sigma\cdot B(f),& B(f)\cdot\tau&\cong\sigma^2\tau^3\cdot B(f),\\
  B(f)\cdot\alpha&\cong\alpha\cdot\sigma^2 B(f),& B(f)\cdot\beta&\cong\beta\cdot B(f).
\end{xalignat*}
For the second equality, the recursion of
$\sigma^{-2}\tau^{-3}\cdot B(f)\cdot\tau$ in basis
$\{s^2t^3\ell_1,st^3\ell_2,st^3\ell_3,t^2\ell_4,t^2\ell_5,\ell_6\}$
coincides with~\eqref{eq:infinite centralizer}, while for the third
equality, the recursion of
$\sigma^{-2}\alpha^{-1}\cdot B(f)\cdot\alpha$ in basis
$\{s^2\ell_1,s^2\ell_2,\ell_3,\dots,\ell_6\}$ coincides
with~\eqref{eq:infinite centralizer}.

Consider the homomorphism $\phi\colon\langle\alpha,\beta\rangle\to\Z$
which counts the total exponent in $\alpha$ of a word; it is the
quotient by the normal closure of $\beta$. Then, for
$w\in\langle\alpha,\beta\rangle$, the element $w\sigma^m\tau^n$
belongs to the centralizer of $f$ if and only if
$(m,n)=(m+2n+\phi(w),3n)$, if and only if $n=0$ and
$w\in\ker(\phi)$. Therefore,
\[Z(f)=\langle\sigma\rangle\times\ker(\phi)=\langle\sigma\rangle\times\langle\beta,\beta^\alpha,\beta^{\alpha\beta},\dots\rangle\cong\Z\times F_\infty.\]

These calculations were checked using the computer algebra program
\textsc{Gap}~\cite{gap4:manual}, and its package \textsc{Img},
specially designed to manipulate wreath recursions. The commands issued were:
{\footnotesize\begin{verbatim}
gap> m := NewSphereMachine("x1=<,x3*x4,x4^-1*x3^-1,x2*x3*x4*x5,x5^-1*x4^-1*x3^-1*x2^-1,x1>(2,3)(4,5)",
	"x2=<,,x4^-1*x3^-1,x2*x3*x4,x5^-1*x4^-1*x3^-1*x2^-1,x2*x3*x4*x5>(1,2)(3,4)(5,6)",
	"x3=<x3,,,,,>","x4=<x4,,,,,>","x5=<,,x5,,,>(1,2)(3,4)(5,6)",
	"x6=<,,,,,x6>(2,3)","x7=<,,,,,x7>(4,5)","x1*x2*x3*x4*x5*x6*x7");;
gap> g := StateSet(m);;
gap> AssignGeneratorVariables(g);; s := x3*x4;; t := x2*s*x5;;
gap> sigma := GroupHomomorphismByImages(g,[x1,x2,x3^s,x4^s,x5,x6,x7]);;
gap> tau := GroupHomomorphismByImages(g,[x1,x2^t,x3^t,x4^t,x5^t,x6,x7]);;
gap> alpha := GroupHomomorphismByImages(g,[x1^(t*x6/t),x2,x3,x4,x5,x6^(x1^t*x6),x7]);;
gap> beta := GroupHomomorphismByImages(g,[x1,x2,x3,x4,x5,x6^x7,x7^(x6*x7)]);;

gap> sigma*m = m*sigma;
true
gap> ChangeFRMachineBasis(tau*m/sigma^2/tau^3,[t^3*s^2,t^3*s,t^3*s,t^2,t^2,t^0]) = m;
true
gap> beta*m = m*beta;
true
gap> ChangeFRMachineBasis(alpha*m/alpha/sigma^2,[s^2,s^2,s^0,s^0,s^0,t^0]) = m;
true
\end{verbatim}}

\end{document}